\documentclass[a4paper,twoside,11pt]{article}
\usepackage[utf8]{inputenc}
\usepackage[english]{babel}
\usepackage{amssymb, amsfonts, amsthm, amsmath}
\usepackage[hidelinks]{hyperref}
\usepackage{pgf,tikz} 
\usepackage[margin=2.5cm]{geometry} 
\usepackage[mathscr]{euscript}
\usepackage{psfrag}
\usepackage[percent]{overpic}
\usepackage{graphicx,accents}
\usepackage{caption} 
\usepackage{multirow} 
\usepackage{mathtools} 
\usepackage{blkarray}
\usepackage{color}

\mathtoolsset{showonlyrefs}
\usetikzlibrary{math,arrows,calc,through}

\makeatletter
\DeclareRobustCommand{\cev}[1]{%
	\mathpalette\do@cev{#1}%
}
\newcommand{\do@cev}[2]{%
	\fix@cev{#1}{+}%
	\reflectbox{$\m@th#1\vec{\reflectbox{$\fix@cev{#1}{-}\m@th#1#2\fix@cev{#1}{+}$}}$}%
	\fix@cev{#1}{-}%
}
\newcommand{\fix@cev}[2]{%
	\ifx#1\displaystyle
	\mkern#23mu
	\else
	\ifx#1\textstyle
	\mkern#23mu
	\else
	\ifx#1\scriptstyle
	\mkern#22mu
	\else
	\mkern#22mu
	\fi
	\fi
	\fi
}

\makeatother

\usetikzlibrary{math,arrows,calc,through}

\numberwithin{equation}{section}

\theoremstyle{plain}
\newtheorem{theorem}{Theorem}[section]
\newtheorem{corollary}[theorem]{Corollary}

\newtheorem{lemma}[theorem]{Lemma}
\newtheorem{proposition}[theorem]{Proposition}

\theoremstyle{definition}
\newtheorem{definition}[theorem]{Definition}

\theoremstyle{remark}
\newtheorem{remark}[theorem]{Remark}
\newtheorem{example}[theorem]{Example}

\renewcommand{\S}{\mathbb{S}}
\newcommand{\Z}{\mathbb{Z}}
\newcommand{\R}{\mathbb{R}}
\newcommand{\C}{\mathbb{C}}
\newcommand{\CP}{{\mathbb{C}\mathrm{P}}}

\newcommand{\calC}{\mathcal{C}}
\newcommand{\calF}{\mathcal{F}}

\newcommand{\calL}{\mathcal{L}}
\newcommand{\calU}{\mathcal{U}}

\newcommand{\calG}{\mathcal{G}}
\newcommand{\calA}{\mathcal{A}}

\newcommand{\calI}{\mathcal{I}}
\newcommand{\Bound}{\partial V}

\newcommand{\hC}{\hat{\C}}

\newcommand{\F}{\mathscr{F}}

\newcommand{\Fs}{\mathsf{F}}
\newcommand{\Ts}{\mathsf{T}}
\newcommand{\Ms}{\mathsf{M}}
\newcommand{\M}{\mathscr{M}}
\newcommand{\Md}{\mathscr{M}^\dagger}
\renewcommand{\S}{\mathscr{S}}
\newcommand{\vphi}{\varphi}

\newcommand{\az}[2]{{\mathscr{A}_{#1}\kern-0.1em\left[#2\right]}}

\DeclareMathOperator{\sgn}{sgn}

\DeclareMathOperator{\sign}{sign}
\DeclareMathOperator{\col}{col}
\DeclareMathOperator{\lc}{lc}

\newcommand{\GD}{G^D}

\tikzstyle{bvert}=[draw,circle,fill=black,minimum size=5pt,inner sep=0pt]
\tikzstyle{blvert}=[draw,circle,fill=blue,draw=blue,minimum size=5pt,inner sep=0pt]
\tikzstyle{wvert}=[draw,circle,fill=white,minimum size=5pt,inner sep=0pt]
\tikzstyle{blwvert}=[draw,circle,draw=blue,fill=white,minimum size=5pt,inner sep=0pt]
\tikzset{circle through 3 points/.style n args={3}{%
		insert path={let    \p1=($(#1)!0.5!(#2)$),
			\p2=($(#1)!0.5!(#3)$),
			\p3=($(#1)!0.5!(#2)!1!-90:(#2)$),
			\p4=($(#1)!0.5!(#3)!1!90:(#3)$),
			\p5=(intersection of \p1--\p3 and \p2--\p4)
			in
			node at (\p5) [draw,circle through= {(#1)}]{}}
}}

\title{The Schwarzian octahedron recurrence (dSKP equation) I: explicit solutions}

\author{Niklas Christoph Affolter
	\thanks{TU Berlin, Institute of Mathematics, Strasse des 17. Juni 136, 10623 Berlin, Germany.
		Départment de mathématiques, ENS, Université PSL, 45 rue d'Ulm, 75005 Paris, France. \textit{E-mail address}: \texttt{affolter~at~posteo.net}} ,
	Béatrice de Tilière
	\thanks{PSL University-Dauphine, CNRS, UMR 7534, CEREMADE, 75016 Paris, France. Institut Universitaire de France. \textit{E-mail address}: \texttt{detiliere at ceremade.dauphine.fr}} ,
	Paul Melotti
	\thanks{Université Paris-Saclay, CNRS, Laboratoire de mathématiques d’Orsay, 91405, Orsay, France.
		\textit{E-mail address}: \texttt{melotti at posteo.net}}
}

\date{June 7, 2024}

\begin{document}

\maketitle

\begin{abstract}
  We prove an explicit expression for the solutions of the discrete Schwarzian octahedron
recurrence, also known as the discrete Schwarzian KP equation (dSKP),
as the ratio of two partition functions. Each one counts weighted oriented dimer configurations of an associated bipartite graph, and is equal to the determinant of a Kasteleyn matrix. This is in the spirit of Speyer's result on the dKP equation, or octahedron recurrence~\cite{Speyer}.
One consequence is that dSKP has zero algebraic entropy, meaning that the
growth of the degrees of the polynomials involved is only
polynomial. There are cancellations in the partition function, and we prove an alternative, cancellation free explicit expression involving \emph{complementary trees and forests}. Using all of the above, we show several instances of the Devron property for dSKP, \emph{i.e.}, that certain singularities in initial data repeat after a finite number of steps. This has many applications for discrete geometric systems and is the subject of a companion paper~\cite{paper2}.
We also find limit shape results analogous to the arctic circle of the Aztec diamond. Finally, we discuss the combinatorics of all the other octahedral equations in the classification of Adler, Bobenko and Suris~\cite{abs}.
\end{abstract}


\section{Introduction}\label{sec:introduction}

The dSKP equation is a relation on six variables that arises in the study of the Krichever-Novikov equation \cite[Equation (30)]{dndskp}, and as a
discretization of the Schwarzian Kadomtsev-Petviashvili
hierarchy \cite{BK1,BK2}, hence its name. Note that it can be traced back to \cite[Equation (4)]{ncwqdskp} as a special case when  $p,q,r = 0$, $\alpha = -\beta = \varepsilon^{-1}$ in the limit $\varepsilon \rightarrow 0$.
It appears in a number of
systems such as: Menelaus' theorem and Clifford configurations
\cite{ksclifford}, evolutions of $t$-embeddings of dimer models (or
Miquel dynamics) \cite{klrr,amiquel}, consistent octahedral
equations \cite{abs}. These examples and many more are described in a companion paper~\cite{paper2}.

In this paper we embed this relation on a lattice to get the so-called
\emph{dSKP recurrence}. Formally, we consider the
\emph{octahedral-tetrahedral lattice $\calL$} defined as:
\begin{align*}
  \calL = \left\{p=(i,j,k) \in \Z^3 : i+j+k \in 2\Z \right\}.
\end{align*}
Consider the space $\hC = \C \cup \{\infty\}$, where $\C$ is an \emph{affine chart} of the \emph{complex projective line} $\CP^1$, and a function $x:\calL \to \hat{\C}$. We say that $x$ satisfies the \emph{dSKP recurrence}, or \emph{Schwarzian octahedron recurrence}, if
\begin{align}\label{eq:dskp_x_intro}
  \frac{(x_{-e_3}-x_{e_2})(x_{-e_1}-x_{e_3})(x_{-e_2}-x_{e_1})}{
  (x_{e_2}-x_{-e_1})(x_{e_3}-x_{-e_2})(x_{e_1}-x_{-e_3})} = -1,
\end{align}
where $(e_1,e_2,e_3)$ is the canonical basis of $\Z^3$, $x_q(p) :=
x(p+q)$ for every $q\in\{\pm e_i\}_{i=1}^3$, and the relation is evaluated at any $p \in \Z^3\setminus\calL$.

Suppose that we are given an \emph{initial data}
$a=\left( a_{i,j} \right)_{i,j\in \Z^2}$ located at vertices
$(i,j,h(i,j)) \in \calL$ for some \emph{height function} $h$, see Section~\ref{sec:recurrence} for the definition. One starts with values
\begin{equation*}
        (x(i,j,h(i,j))) = (a_{i,j}),
\end{equation*}
and apply the dSKP recurrence to get any value $x(i,j,k)$ with $(i,j,k) \in \calL$
and $k>h(i,j)$. This takes the form of a rational function in the
variables $a$. One of the main purposes of this paper is to prove
a combinatorial expression of this
rational function. The corresponding problem has been solved for various similar
recurrences \cite{CScube,Speyer,KPddim,Mkash}, and has led to fruitful developments such as
limit shapes results \cite{PSarctic,dFSG,Ggroves}.

It turns out that the combinatorics fitted to the dSKP recurrence leads
to the introduction of the \emph{oriented dimer model}. Consider a
finite planar graph $G=(V,E)$, and let $F$ be its set of faces, equipped with
weights $(a_f)_{f\in F}$. Suppose that we are given a particular
orientation known as a \emph{Kasteleyn orientation}, seen as a skew
symmetric function $\varphi: \vec{E} \to \{-1,1\}$. An \emph{oriented
  dimer configuration} is a subset of oriented edges $\vec{\Ms}$ such
that every vertex is either the origin or the tip of an oriented edge
in $\vec{\Ms}$. For an oriented edge $\vec{e}$, we denote by
$f(\vec{e})$ the face to the right of
$\vec{e}$. Then we define the weight of $\vec{\Ms}$ as
\begin{equation*}
  w(\vec{\Ms}) = \prod_{\vec{e} \in \vec{\Ms}} \varphi_{\vec{e}} \ a_{f(\vec{e})},
\end{equation*}
and the corresponding \emph{partition function} is given by the signed enumeration of all oriented dimer configurations:
\begin{equation*}
  Z(G,a,\varphi) = \sum_{\vec{\Ms}} w(\vec{\Ms}).
\end{equation*}

The following is a loose statement of one of our main results, see Theorem~\ref{theo:dskp_dimers} for a precise statement.
\begin{theorem} \label{theo:dskp_dimers_intro}
Let $x:\calL\rightarrow\hC$ be a function that satisfies the dSKP recurrence. Let $h$ be a height function and consider an initial data $a=(a_{i,j})=(x(i,j,h(i,j)))$. Then, for every point $(i,j,k)\in
\calL$ with $k>h(i,j)$, the value $x(i,j,k)$ is expressed as a function of $(a_{i,j})$ as
\begin{equation*}
x(i,j,k) = C(G,a) \
\frac{Z(G,a^{-1},\varphi)}{Z(G,a,\varphi)},
\end{equation*}
where $G$ is the \emph{crosses-and-wrenches graph} explicitly constructed from $(i,j,k)$ and $h$~\cite{Speyer}, whose faces are indexed by a subset of $\Z^2$, and equipped with weights $a=( a_{i,j})$ or $a^{-1}=( a^{-1}_{i,j})$; $C(G,a)$ is
the product of all face weights with some exponents.
\end{theorem}

An immediate consequence of Theorem~\ref{theo:dskp_dimers_intro} and
the construction of the graph is that the dSKP recurrence has zero
\emph{algebraic entropy} \cite{BV}, meaning that the degree
of the
function $x(i,j,k)$ in the initial data grows sub-exponentially (and,
in fact, polynomially) in $k$. This unusual property is
typical of integrable rational systems.

\begin{example} \label{ex:onestep}
As an example, let us take $h(i,j)=[i+j]_2$, where
  $[n]_p\in \{0,1,\dots,p-1 \}$ denotes the value of $n$ modulo $p$. Then
  $x(0,0,2)$ may be expressed as a function of
  $a_{0,0}=x(0,0,0), a_{1,0}=x(1,0,1), a_{-1,0}=x(-1,0,1),
  a_{0,1}=x(0,1,1), a_{0,-1}=x(0,-1,1)$, namely
  \begin{equation*}
   \frac{
      a_{1,0}a_{-1,0}a_{0,1} + a_{1,0}a_{-1,0}a_{0,-1}
      - a_{0,0}a_{1,0}a_{-1,0} - a_{-1,0}a_{0,1}a_{0,-1} -
      a_{1,0}a_{0,1}a_{0,-1} - a_{0,0}a_{0,1}a_{0,-1}}
    {a_{0,0}a_{0,-1} + a_{0,0}a_{0,1}
      - a_{0,1}a_{0,-1} -a_{0,0}a_{1,0} - a_{0,0}a_{-1,0} + a_{1,0}a_{-1,0}
    }.
  \end{equation*}
  The corresponding graph is shown in Figure~\ref{fig:ex1}, with the
  list of its oriented dimer configurations and their weights, whose
  sum is indeed the denominator. To check that the numerator of
  Theorem~\ref{theo:dskp_dimers_intro} also matches, let us mention
  that in this case, $C(G,a)=a_{0,0}a_{0,1}a_{0,-1}a_{1,0}a_{-1,0}$, so
  that the denominator is the ``complement polynomial'' of the
  numerator -- meaning that every monomial is replaced by its
  complement in the five variables, since the maximum degree of each
  variable is $1$ in this case.
\end{example}
\begin{figure}[ht]
  \centering
  \includegraphics[width=11cm]{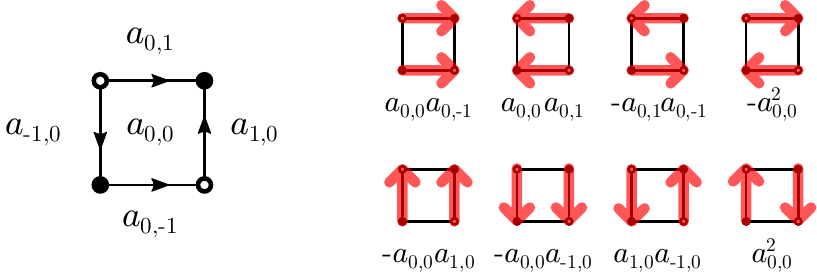}
  \caption{Left: the graph $G$ corresponding to one application of the
    dSKP equation, with a Kasteleyn orientation (showing the oriented
    edges $\vec{e}$ such that $\varphi_{\vec{e}}=1$). Right: its eight oriented
    dimer configurations, and their weights.}
  \label{fig:ex1}
\end{figure}

Another contribution of this paper is to show that
the partition function $Z(G,a)$ can be expressed as a determinant. More precisely, the crosses-and-wrenches graph $G$ is bipartite so that its vertex set can be split into $V=W\sqcup B$; consider the weighted adjacency $K=(K_{w,b})$ of $G$,
whose rows are indexed by white vertices of $W$, columns by black
vertices of $B$, and whose non-zero entries are given by, for every
edge $wb$ of $G$,
\begin{equation}
  \label{eq:defKast}
  K_{w,b} = a_{f(w,b)}-a_{f(b,w)}.
\end{equation}

Then, in Proposition~\ref{prop:oriented_dimers_adjacency_matrix}, we prove that
\begin{equation*}
    Z(G,a,\varphi) = \pm \det(K).
\end{equation*}

It is to be noted that this exact matrix $K$ appears in the recent
introduction of Coulomb gauges, or $t$-embeddings of dimer models. In
\cite{klrr} the authors start with a planar graph equipped with a
standard dimer model, and (under some conditions) find an embedding of
the dual graph: every face $f$ is sent to a point $a_f$. The Kasteleyn
matrix of the initial dimer model is gauge equivalent to a matrix $K$
that is then exactly \eqref{eq:defKast}, therefore $\det(K)$ is equal, up to a constant, to the
partition function of the initial dimer model. As a result, our
approach may be seen as a way to expand the partition function of the
initial dimer model in terms of these variables $a_f$, taken as formal
variables, instead of the usual edge weights. With this perspective of formal variables living in $\hC$, we have no control on the sign of the entries of the matrix $K$, and they do not satisfy the Kasteleyn condition in general~\cite{Kuperberg}.

\begin{figure}[tbp]
  \centering
  \includegraphics[width=11cm]{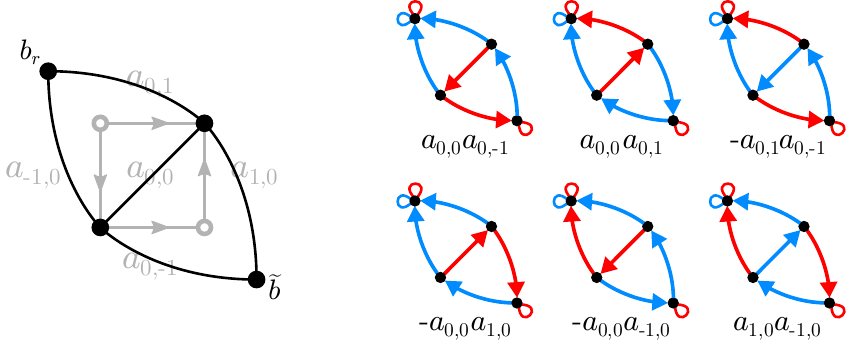}
  \caption{Left: the graph $G^\bullet$ in the running example; two
    vertices $b_r,\tilde{b}$ have been added at the boundary. Right:
    its six complementary trees and forests configurations, and their
    weights. Forests are shown in red and may be rooted at $b_r$
    or $\tilde{b}$, while trees are shown in blue and may only be
    rooted at $b_r$. Roots are represented by small loops.}
  \label{fig:ex1_forests}
\end{figure}

Note that in the previous example, two configurations have weights
that cancel each other. This is a generic fact which, for instance,
occurs at all faces of degree 4 of $G$. Our next main contribution consists in introducing
another model, whose partition function is
equal to $Z(G,a,\varphi)$, but whose configurations are in \emph{one-to-one
correspondence} with the monomials in the $a$ variables. More precisely, in
Section~\ref{sec:comb_sol_II}, for any quadrangulation $G$ of the sphere, we introduce a model of \emph{complementary trees and
forests} on the graph $G^\bullet$ formed by the black vertices of $G$ and
diagonals joining them, with some boundary conditions. A
configuration consists of two subsets $(\Ts,\Fs)$ of edges of
$G^\bullet$ such that $\Ts$ is a spanning tree, and its complement
$\Fs$ is a spanning forest, rooted at some specific vertices, see Section~\ref{sec:comb_sol_II} for details, and Figure~\ref{fig:ex1_forests} for an example. We show the
following, see also Theorem~\ref{thm:comb_int_II}.
\begin{theorem}\label{thm:comb_int_II_intro}
For any Kasteleyn orientation $\vphi$, the oriented dimer partition function is equal to
\begin{equation*}
Z(G,a,\vphi)=\pm
\sum_{(\Ts,\Fs)} \sign(\Ts,\Fs) \prod_{\vec{e}\in \Fs} a_{f_{\vec{e}}},
\end{equation*}
where the sum is over all complementary trees and forests
configurations of $G^\bullet$. Moreover, there is a bijection between
terms in the sum on the right-hand-side and monomials of
$Z(G,a,\vphi)$ in the variables $a$.
\end{theorem}

The tools we develop to get the previous two results have several
applications, in particular they allow us to study
\emph{singularities} of the dSKP recurrence. Although interesting in their own respect, the introduction of such
singularities is motivated by their occurrence in geometric systems.
The study of these systems and their singularities is the subject of the companion paper~\cite{paper2},
where the results of the current paper allow us to provide a unified treatment of the description of singularities of several geometric systems, recovering previously known results~\cite{gdevron,yao}, solving three conjectures of~\cite[Section~9]{gdevron}, and showing counterparts in other systems.

\begin{figure}[ht]
  \centering
  \includegraphics[width=11cm]{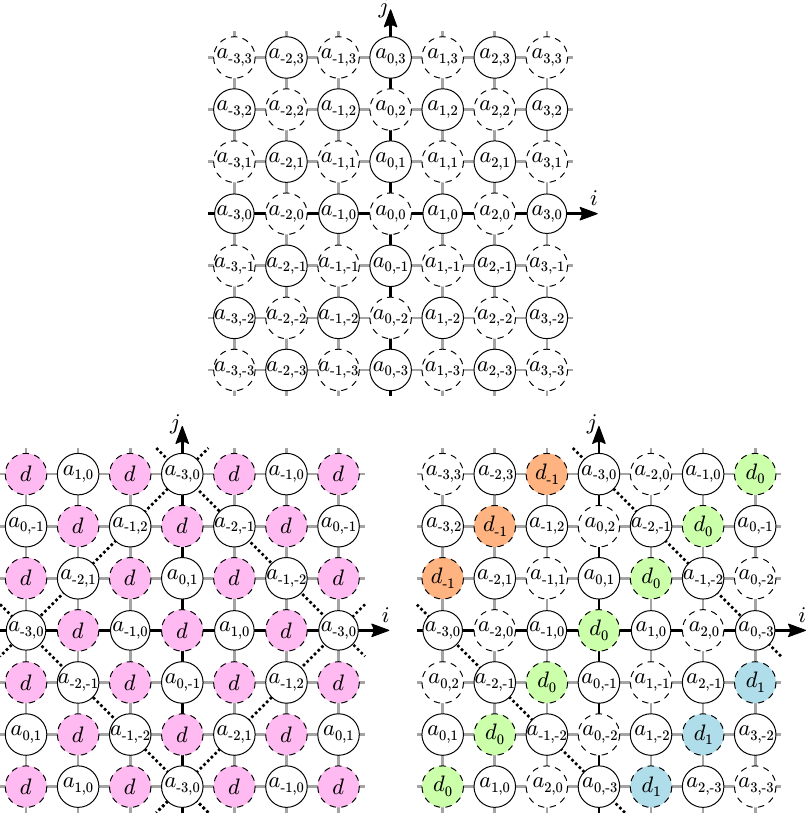}
  \caption{Initial data $(a_{i,j})$ for the dSKP recurrence, with
    height function $h(i,j)=[i+j]_2$. Variables in dashed circles lie
    at height $k=0$, while those in solid circles lie at height
    $k=1$. Top: generic. Bottom left: $3$-Dodgson; an
    elementary pattern is shown as a dashed square, and particular
    values at height $0$ are shown in purple. Bottom right:
    $(3,2)$-Devron; constant ``columns'' shown in
    orange, green, blue.}
  \label{fig:sing_data}
\end{figure}

Consider the height function $h(i,j)=[i+j]_2$, so that
the initial data $a=(a_{i,j})=(x(i,j,[i+j]_2))$ occupies heights $k=0$ and
$k=1$ in $\calL$. Suppose that all of the values at height $k=0$ are
equal to a single value $d\in\hat\C$, and that values at height $k=1$
are $m$-doubly periodic, meaning that
$a_{i,j}=a_{i+m,j+m}=a_{i+m,j-m}$; see Figure~\ref{fig:sing_data},
bottom left. We
call such initial data $m$-\emph{Dodgson initial conditions}. In this
case, the dSKP recurrence fails to define $x(i,j,k)$ for $k<0$, as
trying to apply \eqref{eq:dskp_x_intro} leads to a
singularity. However, getting the values at heights $k>1$ seems
possible. As the dSKP recurrence is one of the octahedral consistent
equations \cite{abs}, it is expected to have features common to
integrable system, one of them being the Devron property
\cite{gdevron}: if some data is singular for the backwards dynamics,
it should become singular after a finite number of applications of the forward
dynamics. In our case, this means that at some height $k\geq 1$
(\emph{i.e.} after $k-1$ \emph{iterations}, understood as applications
of the recurrence~\eqref{eq:dskp_x_intro} to a whole level), all values of $x(i,j,k)$
will be equal. After this point the forward dynamics becomes
impossible. We prove this Devron property, see also Section~\ref{sec:dskp-devr-prop}.
\begin{theorem}
  \label{theo:sing_dodgson_intro}
  For $m$-Dodgson initial data, the values of $x$ are
  constant after $m-1$ iterations of the dSKP recurrence. In other
  words,
    for every $(i,j)\in\Z^2$ such that $(i,j,m) \in \calL$,
    $x(i,j,m)$ is independent of $(i,j)$.
\end{theorem}

In fact, we prove a stronger invariance result for the partition function $Z(G,a,\vphi)$ itself (and not only the ratio), via the combinatorics of particular trees and forests configurations named \emph{permutation spanning forests}, see Section~\ref{sec:cst_col}, Theorems~\ref{thm:vert_translation} and~\ref{thm:col_degen_I}. As a consequence, we are able to explicitly compute this final value using the
determinant and minors of the $m \times m$ matrix with entries
$\frac{1}{a_{i,j}-d}$, which is reminiscent of Dodgson's condensation
\cite{dodgson}, hence the name of these initial data; see
Corollary~\ref{cor:dodgson}. Note that this matrix is much smaller
than the previous matrix $K$, whose size is roughly $m^2 \times
m^2$. For instance, when $m=2$, a single iteration produces a constant
layer; this can be seen in the running example, with $a_{0,0}$ playing
the role of $d$. In this case, the explicit value of
Corollary~\ref{cor:dodgson} uses the matrix
\begin{align*}
        N = \begin{pmatrix}
                \frac{1}{a_{0,1}-a_{0,0}} & \frac{1}{a_{1,0}-a_{0,0}} \\
                \frac{1}{a_{-1,0}-a_{0,0}} & \frac{1}{a_{0,-1}-a_{0,0}}
        \end{pmatrix}.
\end{align*}
It states that
\[Z(G,a,\varphi) = \pm \prod_{i,j}(a_{i,j}-a_{0,0}) \cdot \det N,
\] and that
for any $(i,j)\in\Z^2$ with $[i+j]_2=0$,
\begin{align*}
        x(i,j,2) = a_{0,0} + \sum_{i,j}\left( N^{-1} \right)_{i,j},
\end{align*}
with the sum and product being over $(i,j)\in \{(0,1),(1,0),(0,-1),(-1,0)\}$.

Assuming more symmetries in the initial data, we prove an even simpler form for the final value.
\begin{corollary}
  \label{cor:harm_mean}
  For $m$-Dodgson initial data, suppose in addition
  that for some $p\not\in m\Z$, when $[i+j]_2=1$,
  $a_{i,j}=a_{i+p+1,j-p+1}$. Then after $m-1$ iterations of the dSKP
  recurrence, $x(i,j,m)$ is the shifted harmonic mean of the $m$ different
  values of the initial data:
  \begin{equation*}
   x(i,j,m) = d+ \left( \frac{1}{m} \sum_{i=0}^{m-1} \frac{1}{a_{i,1-i}-d}\right) ^{-1}.
  \end{equation*}

\end{corollary}

We then consider a generalization of Dodgson initial conditions. Suppose
that the initial data is $m$-\emph{simply periodic}, meaning that for
all $(i,j)\in\Z^2$, $a_{i,j}=a_{i+m,j+m}$. We also assume that for some
$p \geq 1$, for all $(i,j)\in\Z^2$ with $[i-j]_{2p}=0$,
$a_{i,j} = a_{i+1,j+1}$. This amounts to having every $p$-th SW-NE
diagonal at height $0$ constant, see Figure~\ref{fig:sing_data},
bottom right; we denote these constant values by
$a_{i,j}=:d_{(j-i)/(2p)}$. In this case, it is convenient to rotate the
lattice by $45$ degrees, so the singularity becomes constant columns,
which are easier to visualize; every $p$-th column of height $0$ is constant.
We call these $(m,p)$-Devron initial data. Again, we
expect a Devron property to hold for this kind of singular data, which
here means that at some height $k\geq 1$, values of $x(i,j,k)$ also
have $p$-periodic constant columns.

\begin{theorem}
  \label{theo:sing_devron_intro}
  For $(m,p)$-Devron initial data, let $k=(m-2)p+2$. Then after $k-1$ iterations of the dSKP
  recurrence, the values of $x$ also have $p$-periodic constant columns.

  More precisely, for all $(i,j)\in\Z^2$ such that $[i-j-mp]_{2p}=0$,
  \begin{equation*}
    x(i,j,k) = x(i+1,j+1,k).
  \end{equation*}
\end{theorem}
When $p=1$, \emph{i.e.}, when all columns at height $0$ are constant, the proof of the strong invariance result of the $m$-Dodgson case also works, meaning that we have invariance of the partition function itself. For
generic $p$ we cannot provide such a combinatorial proof -- in fact
the values of $Z(G,a,\varphi)$ are generically not invariant, while
their ratio in Theorem~\ref{theo:dskp_dimers_intro} is -- and we
resort to more algebraic tools, in particular to Theorem~\ref{theo:kery}.

Another case of study is when the initial data is periodic with
respect to two non-collinear vectors $(s,t)$ and $(u,v)$ in $\Z^2$. We can also
predict at which height singularities reoccur in that case, as
consequences of Theorem~\ref{theo:sing_devron_intro}; see Corollary~\ref{cor:devron_periodic_vectors}.

As mentioned, such combinatorial solutions of discrete evolution
equations are often related to limit shapes phenomena for the
associated statistical mechanics models, which generalize the
celebrated arctic circle phenomenon for tilings of the Aztec diamond; on this classical theory, see \cite{CEP,CKP,Gbook} and references therein, and for approaches similar to ours, \cite{dFSG,PSarctic,Mkash,Ggroves}. In our case, it is unclear if
one can hope for probabilistic interpretations of this sort for
oriented dimers or complementary trees and forests, first because the
solution is not a partition function but a ratio of partition
functions, and second because configurations come with signs. However, if $h$ is fixed, for any $(i,j,k)$ with $k>h(i,j)$, we may see $x(i,j,k)$ as a rational function of initial data $(a_{i',j'})$ (as was the case in Example~\ref{ex:onestep}), and consider the partial derivative
\begin{equation*}
  \rho(i,j,k) = \dfrac{\partial
    x(i,j,k)}{\partial a_{0,0}}.
\end{equation*}
By adapting the techniques developed in the previous references, we
can compute the asymptotic behaviour of $\rho(i,j,k)$.
More precisely, we are able to study
$\rho(i,j,k)$ when evaluated at some specific solutions of the dSKP
recurrence, namely $x(i,j,k)=ia+jb+kc+d$ and $x(i,j,k)=a^i b^j c^k d$,
where $a,b,c,d$ are real parameters, and $h(i,j)=[i+j]_2$. In these cases, we compute
the asymptotics of $\rho(xk,yk,k)$ when $k\to \infty$,
which depend on $x,y$. In some regimes of the $a,b,c,d$ parameters,
this quantity behaves like $k^{-1}$ in some region of $x,y$
corresponding to an ``arctic ellipse'', and decays exponentially
outside of this region. For other choices of $a,b,c,d$, the behaviour
is always exponential and sometimes divergent; that is, in those cases, we show
\begin{equation*}
  \lim_{k\to \infty} \frac{1}{k} \log \rho(xk,yk,k) = \xi(x,y)
\end{equation*}
with an explicit rate function $\xi(x,y)$. In terms of dynamical systems, we can see this rate $\xi(x,y)$ as a Lyapunov exponent for the dynamics (see e.g. \cite{BSbook}); we show that it can be positive in a range of $(x,y)$. This positivity of the Lyapunov exponent is often associated with chaos. The previous results are made precise in
Proposition~\ref{prop:limitshape}. In the absence of a probabilistic
interpretation, we may thus view these results as a way to quantify the
influence of initial conditions on solutions of the dSKP recurrence.

Finally, we give exact solutions for all other equations of the
classification of integrable equations of octahedron type by Adler,
Bobenko and Suris \cite{abs}. In this reference the authors classify all
equations on octahedra that satisfy some \emph{multi-dimensional
  consistency} condition, up to admissible transformations, and come
up with a finite list $\chi_1, \dots, \chi_5$. Equation $\chi_1$ is
the standard octahedron, or dKP, equation, whose solution was found by
Speyer \cite{Speyer} in terms of the dimer model; $\chi_2$ is the dSKP
equation. In Theorem~\ref{theo:chi_k}, we show how explicit solutions
of the $\chi_3,\chi_4$ and $\chi_5$ recurrences can be found from our
$\chi_2$ solution as leading coefficient in certain expansions. Then,
for $\chi_4$ and $\chi_5$, we also give direct combinatorial
descriptions of these solutions, at least in the case of the height
function $h(i,j)=[i+j]_2$.

\subsection*{Plan of the paper}
In Section~\ref{sec:recurrence} we set up the definitions and recall
Speyer's solution of the dKP recurrence~\cite{Speyer}. In
Section~\ref{sec:comb_sol_I} we introduce the oriented dimer model,
prove its determinantal structure, and state Theorem~\ref{theo:dskp_dimers} (Theorem~\ref{theo:dskp_dimers_intro}
of the introduction); we then prove it, extending some of
Speyer's tools and techniques. In Section~\ref{sec:comb_sol_II} we
introduce the complementary trees and forests model; we show how it
relates generically to oriented dimers and Kasteleyn matrices, and we
prove Theorem~\ref{thm:comb_int_II} (Theorem~\ref{thm:comb_int_II_intro} of the introduction).
Then in Section~\ref{sec:singularity}
we turn to the study of singularities of the dSKP recurrence, by
studying oriented dimers, or trees and forests, on Aztec diamonds; we
prove
Theorem~\ref{theo:sing_dodgson_intro}, Theorem~\ref{theo:sing_devron_intro},
and Corollary~\ref{cor:harm_mean}. Section~\ref{sec:limitshapes} is
concerned with ``limit shapes'' phenomena, with a proper statement of
the asymptotic behaviour of $\rho(i,j,k)$. Finally, in
Section~\ref{sec:other_consistent_equ} we extend the combinatorial
solution to all consistent equations of the octahedral family of
\cite{abs}.

\subsection*{Acknowledgments}We would like to thank the anonymous referees for their careful reading of the manuscript and for their many useful suggestions that helped increase the quality of the paper.
The first author would like to thank Boris Springborn and Sanjay Ramassamy for helpful discussions. He is supported by the Deutsche Forschungsgemeinschaft (DFG) Collaborative Research Center TRR 109 ``Discretization in Geometry and Dynamics'' as well as the by the MHI and Foundation of the ENS through the ENS-MHI Chair in Mathematics. The second and third authors are partially supported by the DIMERS project ANR-18-CE40-0033 funded by the French National Research Agency.

\section{The dSKP recurrence \& some tools of Speyer}\label{sec:recurrence}

In Section~\ref{sec:dSKP_recurrence_def}, we give a precise definition of the dSKP
recurrence then, in Section~\ref{sec:crosses_and_wrenches}, we introduce the method of crosses and wrenches of Speyer~\cite{Speyer}
and, in order to put one of our main results (Theorem~\ref{theo:dskp_dimers}) into perspective, we state the result of~\cite{Speyer} on the dKP recurrence in Section~\ref{sec:dKP_recurrence}.

\subsection{Definition}\label{sec:dSKP_recurrence_def}

The dSKP recurrence lives on vertices of the \emph{octahedral-tetrahedral
  lattice $\calL$} defined as:
        \begin{align*}
                \calL = \left\{p=(i,j,k) \in \Z^3 : i+j+k \in 2\Z \right\}.
        \end{align*}
Projecting $\calL$ vertically onto the plane yields the lattice
$\Z^2$, whose bipartite coloring of the vertices corresponds to even
and odd levels of $\calL$.

\begin{remark}
  \label{rk:A3}
  A somewhat more symmetric lattice that is in fact isomorphic to
  $\calL$ is defined in \cite{abs} as the \emph{root lattice}:
  \begin{equation*}
    Q(A_3) = \{n=(n_0,n_1,n_2,n_3) \in \Z^4 \mid n_0 + n_1 + n_2 + n_3 =
    0\}.
  \end{equation*}
  More precisely, an octahedral cell in $\calL$ is given by six
  vertices $p \pm e_i$, where $p \in \Z^3 \setminus \calL$ and $(e_1,e_2,e_3)$ is
  the canonical basis of $\Z^3$; in
  $Q(A_3)$ they are $n+e_i+e_j$ where $n_0+n_1+n_2+n_3 = -2$ and $\{i,j\}$
  runs through the $6$ pair sets in $\{0,1,2,3\}$. An example of a graph
  isomorphism between $Q(A_3)$ and $\calL$ is $(n_0,n_1,n_2,n_3)\mapsto(n_0+n_1, n_0+n_2, n_0+n_3)$
  with inverse given by  $(i,j,k)\mapsto\frac12
  (i+j+k,i-j-k,-i+j-k,-i-j+k)$.
  These will be useful to translate some results of \cite{abs} into our setting, see Section~\ref{sec:other_consistent_equ}.

  In this paper we use $\calL$ even though its symmetries are less
  apparent, as we are interested in iterating an equation on octahedra
  towards the distinguished $e_3$ direction.
\end{remark}

Up to now we have discussed the definition space of the dSKP recurrence, we now turn to the natural target space: the \emph{complex projective line} $\CP^1$. Consider the equivalence relation $\sim$ on $\C^2$ such that for $v,v'\in \C^2$ we have $v\sim v'$ if there is a $\lambda \in \C\setminus \{0\}$ such that $v = \lambda v'$. Every point in the projective line is an equivalence class $[v] = \{v' : v' \sim v\}$ for some $v \in \C^2 \setminus \{(0,0)\}$, thus
\begin{align*}
        \CP^1= \{[v] : v \in \C^2 \setminus \{(0,0)\} \} = \left(\C^2 \setminus \{(0,0)\}\right)/\sim.
\end{align*}
It is practical to consider an \emph{affine chart} $\C$ of $\CP^1$, and the set $\hC = \C \cup \{\infty \}$ which we identify with $\CP^1$. 
Every point $z\in \C \subset \hC$ corresponds to $[z,1]$ in $\CP^1$ and $\infty\in \hC$ corresponds to $[1,0]$. In $\hC$ one can perform the usual arithmetic operations on $\C$. One can even apply the naive calculation rules $z + \infty = \infty, z/\infty = 0$ etc., see \cite[Section 17]{rgbook}.

\begin{definition}\label{def:dSKP_recurrence}
        A function $x: \calL \rightarrow \hC$ satisfies the
        \emph{dSKP recurrence}, if
        \begin{align}\label{eq:dskp_x}
          \frac{(x_{-e_3}-x_{e_2})(x_{-e_1}-x_{e_3})(x_{-e_2}-x_{e_1})}{
          (x_{e_2}-x_{-e_1})(x_{e_3}-x_{-e_2})(x_{e_1}-x_{-e_3})} = -1.
        \end{align}
        holds evaluated at every point $p$ of $\Z^3\setminus \calL$, where $x_q(p):=x(q+p)$ for every $q\in \{\pm e_i\}_{i=1}^3$.
        More generally, if $A \subset \calL$ and $x:A\to \hat{\C}$, we
        say that $x$ \emph{satisfies the dSKP recurrence on} $A$ when
        \eqref{eq:dskp_x} holds whenever all the points are in $A$.
\end{definition}

\begin{remark}\label{rem:dskpprojsymmetry}\leavevmode
\begin{enumerate}
\item By direct calculation, one sees that the dSKP recurrence features \emph{octahedral symmetry}, \emph{i.e.}, if it holds then it is also satisfied for any permutation of the unit vectors $(e_1,e_2,e_3)$, and for any reflection $e_i \mapsto -e_i$.

\item A projective transformation $f:\CP^1 \rightarrow \CP^1$ is a
  bijection such that there is a matrix $F \in GL(2,\C)$ with
  $f([v]) = [F v]$ for all $v\in \C^2 \setminus\{(0,0)\}$. Conversely, any matrix $F$
  with $\det F\neq 0$ defines a projective transformation of
  $\CP^1$. In $\hC$, any projective transformation
  acts by $f(z) = \frac{a z +b}{c z +d}$ for some $a,b,c,d\in \C$ with
  $ad-bc \neq 0$ and some special rules for $\infty$, in particular
  $f(\infty) = \frac{a}{c}$ and $f(-\frac{d}{c}) = \infty$. It is a
  direct calculation to verify that the dSKP recurrence is invariant
  under projective transformations of $\hC$, and even by Möbius
  transformations, which are projective transformations possibly
  composed with complex conjugation $z \mapsto \bar{z}$. This is
  the first reason why $\hC$ is the natural target space for the
  dSKP recurrence.
  \item The other reason is that
  if 
  $x_{-e_1},x_{-e_2},x_{-e_3},x_{e_1},x_{e_2}$ are given such that
  \begin{align*}
    x_{e_1} \neq x_{e_2}, \ x_{e_2} \neq x_{-e_1}, \ x_{-e_1} \neq
    x_{-e_2}, \ x_{-e_2} \neq x_{e_1},
  \end{align*}
    then $x_{e_3}$ is well-defined by the dSKP recurrence, while this is not generally true in $\C$. In fact, if the condition above is satisfied, then there is a unique projective involution $f$ of $\CP^1$ such that $f(x_{e_1}) = x_{-e_1}$ and $f(x_{e_2}) = x_{-e_2}$. A quick calculation shows that $f(x_{e_3}) = x_{-e_3}$ if and only if dSKP is satisfied. However, $f(x_{e_3})$ may be $\infty$ which is fine in $\CP^1$ but not in $\C$.
\end{enumerate}
\end{remark}
\begin{example}
  \label{ex:sol}
  If $a,b,c,d\in\hat{\C}$, the function $x:\calL \to \hat{\C}$ given by
  $x(i,j,k)=ia+jb+kc+d$ satisfies the dSKP recurrence. The same is true
  for the function $x(i,j,k)=a^i b^j c^k d$.
\end{example}

Following~\cite[Section 2.1]{Speyer}, we now define initial conditions for this recurrence.
Let $h:\Z^2 \to \Z$ be a function such that, for all $(i,j)\in\Z^2$, $(i,j,h(i,j))\in \calL$; we say that $h$ is a \emph{height function} if the following holds:
\begin{enumerate}
\item If $(i,j)$ and $(i',j')$ are neighbors in $\Z^2$, then
  $|h(i,j)-h(i',j')|=1$,
\item $\lim_{|i|+|j|\to \infty}h(i,j)+|i|+|j| = \infty$.
\end{enumerate}
Consider the following subset of $\Z^3$ that will
play the role of initial data locations for the dSKP recurrence:
\begin{equation}
  \label{eq:defI}
  \calI_{h} = \{(i,j,h(i,j)) \mid (i,j)\in\Z^2\}.
\end{equation}
The idea is that fixing
the values of $x$ at the points of $\calI_{h}$ is enough to define $x$
on the \emph{upper set $\calU_{h}$ of $h$}, defined as
\begin{equation}
  \label{eq:defU}
  \calU_{h} = \{(i,j,k)\in \calL \mid k > h(i,j)\}.
\end{equation}
When there is no ambiguity, we will simply denote these by $\calI,\calU$.

For the sequel, we also need the following definition. The
\emph{closed (resp. open) square cone} of a vertex $(i,j,k)\in \calL$
(roughly speaking a semi-infinite square-pyramid with its tip at $(i,j,k)$,
see Figure~\ref{fig:cross_wrench}):
\begin{equation}
\label{eq:defC}
  \begin{split}
    \calC_{(i,j,k)} &= \{(i',j',k')\in \calL \mid k' \leq k - |i-i'| -
    |j-j'|\}\\
    \mathring{\calC}_{(i,j,k)} &= \{(i',j',k')\in \calL \mid k' < k - |i-i'| -
    |j-j'|\}.
  \end{split}
\end{equation}
Note that Condition 2. on the height function is equivalent to the fact that for any $p\in \calL$,
$\calC_p \cap \calU$ is finite~\cite{Speyer}.

Consider a height function $h$, and an \emph{initial condition} $a$, that is a function $a:\Z^2\rightarrow \C$ such that on $\calI$,
\[
x(i,j,h(i,j))=a_{i,j}.
\]
Our goal is to analyze the solution $x$ to the dSKP recurrence on the set $\calI\cup \calU$ when the initial condition on $\calI$ is given by $a$. 

\subsection{The method of crosses and wrenches} \label{sec:crosses_and_wrenches}

Let us now turn to the \emph{method of crosses and wrenches} of
Speyer~\cite[Section 3]{Speyer}. All graphs considered are simple, connected, planar and embedded, implying that they also have \emph{faces}; in order to alleviate the text, these assumptions will not be repeated.

We first need the definition of a \emph{graph with open faces}. Consider
a finite graph $G=(V,E)$. Denote by $F^i$ the set of internal
faces. Partition the external boundary into sets of adjacent edges;
this partitions the outer face into a finite number of
  faces referred to as
\emph{open faces} and denoted by $F^o$. Set $F=F^i
\cup F^o$.

Given a height function $h$, Speyer defines an infinite graph $\calG$, referred to as the
\emph{(infinite) crosses-and-wrenches graph},
in the following way. Faces of $\calG$ are indexed by points of
$\calI$, and are in bijection with $\Z^2$: the face corresponding to
$(i,j,h(i,j))$ is centered at the vertex $(i,j)$ of $\Z^2$. Every unit
square of $\Z^2$ is bounded by four vertices corresponding to four
faces of $\calG$; the way the four faces meet depends on the values of
$h$ at the vertices. Since $h$ is a height function, we are in one of
the following cases: either both diagonally opposite heights are equal
(and differ by 1), in which case we put a ``cross'' (a vertex of
degree $4$) at the intersection of the four faces; or two
diagonally-opposite faces have the same height $h_0$, and the other
two have heights $h_0-1$ and $h_0+1$ respectively, in which case we
put a ``wrench'' (an edge with two endpoints of degree $3$) where
those four faces meet, with the ``handle'' separating the faces of
height $h_0$, see Figure~\ref{fig:cross_wrench} (center) for an
example, and \cite{Speyer} for more details. Note that the graph
$\calG$ is bipartite with faces of degree 4, 6 or 8.

Then, to every point $p$ of $\calU$, one assigns a finite subgraph
with open faces of $\calG$, denoted by $G_p=(V_p,E_p)$ and referred to as the \emph{crosses-and-wrenches graph corresponding to $p$}, constructed
as follows.  The internal faces $F_p^i$ are indexed by elements of
$\calI \cap \mathring{\calC}_{p}$, and the edges $E_p$ and vertices
$V_p$ are those of $\calG$ that belong to at least one of these
faces. The open faces $F_p^o$ are indexed by the elements of $\calI$
that share some (but not all) edge(s) with $E_p$; note that there are
no edges separating open faces. We have $F_p=F_p^i\cup F_p^o$.  The
vertices of the external face that separate open faces are called
\emph{boundary vertices} and their set is denoted by
$\partial V_p \subset V_p$; see again Figure~\ref{fig:cross_wrench}.
Whenever no confusion occurs, we will remove the subscript $p$ from
the notation. Each face $f$ of $F$ corresponds to some $(i,j,h(i,j))$
in $\calI$ and to some $(i,j)$ in $\Z^2$, and we assign it a weight
\[
a_f = a_{i,j}=x(i,j,h(i,j)),
\]
corresponding to the initial condition $a$ for the dSKP recurrence.
The \emph{degree} $d(f)$ of the face
$f$ is defined as the number of edges of $G_p$ adjacent to $f$.

\begin{figure}[tbp]
  \centering
  \def\svgwidth{15cm} 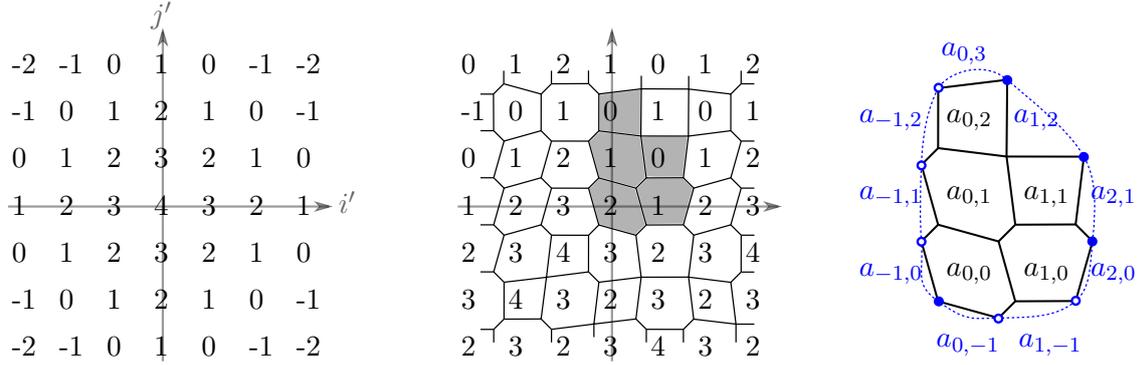
  \caption{Left: for $(i,j,k)=(0,0,4)$, the function
    $(i',j') \mapsto k-|i-i'|-|j-j'|$ that defines the closed square
    cone $\calC_{(0,0,4)}$. Center: an example of height function $h:\Z^2\to \Z$,
    with the corresponding infinite crosses-and-wrenches graph $\calG$. Note that the function on the left and the height function are not related a priori. The
    places where $(i',j',h(i',j')) \in \mathring{\calC}_{(0,0,4)}$, \emph{i.e.}, where the height is strictly smaller than on the left, are colored
    in gray; this is used to build the finite crosses-and-wrenches graph $G_{(0,0,4)}$. Right: the graph $G_{(0,0,4)}$, equipped with the
    corresponding variables $a_{i,j}=x(i,j,h(i,j))$. The vertices of
    $\Bound_{(0,0,4)}$ are shown in blue and joined by a dashed ``boundary''.}
  \label{fig:cross_wrench}
\end{figure}

\subsection{dKP recurrence}\label{sec:dKP_recurrence}

In \cite{Speyer}, Speyer solved the case of the dKP or octahedron
recurrence  \footnote{In fact Speyer solved a slightly more general
  version of the octahedron recurrence. In the following, we set the
  additional coefficients of Speyer $(a,b,c,d)$ to one, which
  specializes the generalized recurrence of Speyer to the standard
  octahedron recurrence \cite{hirota}
}, and a correspondence with the dimer model was
established. In order to make
the context of the present paper more clear, and to put our
forthcoming Theorem~\ref{theo:dskp_dimers} into perspective, we rephrase these results in our notation.

\begin{definition}
  \label{def:dkp_recurrence}
  A function $x: \calL \rightarrow \hat{\C}$ satisfies the
  \emph{dKP or octahedron recurrence} if
  \begin{align*}
   x_{e_3}x_{-e_3} = x_{e_1}x_{-e_1} + x_{e_2}x_{-e_2},
  \end{align*}

  holds evaluated at any $p$ of $\Z^3\setminus \calL$.
\end{definition}

The main result of Speyer~\cite{Speyer} relies on the following definitions.
Let $G$ be a finite graph. Then, a \emph{dimer configuration} $\Ms$ of $G$, or
\emph{perfect matching}, is a subset of edges such that every
vertex of $G$ is incident to exactly one edge of $\Ms$; we denote by
$\M$ the set of dimer configurations of $G$. When $G$ is a graph with open faces $F=F^i\cup F^o$, equipped with face weights $(a_f)_{f\in F}$, define the
weight of $\Ms$ as
\begin{equation*}
  w(\Ms) = \prod_{e\in \Ms}\frac{1}{a_{f(e)}a_{f'(e)}},
\end{equation*}
where $f(e),f'(e)$ denote the two faces adjacent to $e$. The corresponding
\emph{partition function}, denoted $Z_{\text{dim}}(G,a)$, is
\begin{equation*}
  Z_{\mathrm{dim}}(G,a) = \sum_{\Ms \in \M} w(\Ms).
\end{equation*}

\begin{theorem}[\cite{Speyer}]
  \label{theo:speyer}
  Let $x: \calL \rightarrow \hat{\C}$ be a function that satisfies
  the dKP recurrence. Let $h$ be a height function, and $\calI, \calU$
  be defined as in Equations~\eqref{eq:defI},~\eqref{eq:defU}. Let
  $a=(a_{i,j})=(x(i,j,h(i,j)))$ be the initial data indexed by
  points of $\calI$. Then for every point $p$ of $\calU$,
\begin{equation*}
    x(p) = C_{\mathrm{dim}}(G_p,a) \cdot \ Z_{\mathrm{dim}}(G_p,a),
  \end{equation*}
  where
  \begin{equation*}
    C_{\mathrm{dim}}(G_p,a)= \prod_{f=(i,j) \in F_p^i}a_{i,j}^{\frac{d(f)}{2}-1} \prod_{f=(i,j) \in F_p^o}a_{i,j}^{\lceil \frac{d(f)}{2} \rceil},
  \end{equation*}
and $G_p$ is the crosses-and-wrenches graph corresponding to $p$.
\end{theorem}

\section{dSKP: combinatorial solution I - oriented dimers}
\label{sec:comb_sol_I}

In Section~\ref{sec:main_theorem}, we state
Theorem~\ref{theo:dskp_dimers}, our main result on the dSKP
recurrence. In Section~\ref{sec:infinite}, we introduce the ratio
function of oriented dimers in the setting of infinite completions, a
tool that allows to smoothly handle boundary issues. Then, in
Section~\ref{sec:invariance}, we prove invariance of the ratio
function under two types of moves on the underlying graph, namely
\emph{contraction/expansion of a degree 2 vertex} and \emph{spider
move}. Using all of this we proceed with the proof of
Theorem~\ref{theo:dskp_dimers} in Section~\ref{sec:proof_main_thm},
following an argument of~\cite{Speyer}.

\subsection{Definitions and main dSKP theorem}\label{sec:main_theorem}

Consider a finite, bipartite graph $G$ with open faces $F=F^i\cup F^o$, equipped with face weights $(a_f)_{f\in F}$. The set of vertices $V$ is naturally split into black and white, $V=B\sqcup W$, and from now on we assume that $|W|=|B|$. Denote by $\vec{E}$ the set of directed edges of $G$, \emph{i.e.}, given an edge $wb$ of $E$ there corresponds two directed edges $(w,b)$, $(b,w)$ of $\vec{E}$; when vertices are not specified a directed edge is also denoted as $\vec{e}$.

A \emph{Kasteleyn orientation} \cite{Kasteleyn} is a skew-symmetric
function $\vphi$ from $\vec{E}$ to $\{-1,1\}$ such that, for every
internal face $f$ of $F^i$ of degree $2k$, we have
\begin{equation*}
  \prod_{wb\in\partial f}\vphi_{(w,b)}=(-1)^{k+1}.
\end{equation*}
This corresponds to an orientation of edges of the graph: an edge
$e=wb$ is oriented from $w$ to $b$ when $\varphi_{(w,b)}=1$, and from
$b$ to $w$ when $\varphi_{(w,b)}=-1$.  By Kasteleyn~\cite{Kasteleyn2}, such an
orientation exists when $G$ is planar.

An \emph{oriented dimer configuration} of $G$ is a subset of oriented
edges $\vec{\Ms}$ such that its undirected version $\Ms$ is a dimer configuration. Denote by $\vec{\M}$ the set of oriented dimer configurations of $G$. Note that given a dimer configuration $\Ms$ there corresponds $2^{|\Ms|}$ oriented dimer configurations, where $|\Ms|$ denotes the number of edges of $\Ms$.

An oriented edge $\vec{e}$ separates two (inner or open) faces, and we denote by $f(\vec{e})$ the one that is on the right relative to the orientation of $\vec{e}$. Given a Kasteleyn orientation $\varphi$, the \emph{weight} of an oriented dimer configuration $\vec{\Ms}$ is
\begin{equation*}
  w(\vec{\Ms}) = \prod_{\vec{e} \in \vec{\Ms}} \varphi_{\vec{e}} \ a_{f(\vec{e})},
\end{equation*}
and the corresponding \emph{partition function} is
\begin{equation*}
  Z(G,a,\varphi) = \sum_{\vec{\Ms}\in\vec{\M}} w(\vec{\Ms}).
\end{equation*}

\begin{remark}
By grouping the two possible orientations of an edge, the partition function can be rewritten as
\begin{equation}\label{equ:partition_function_2}
Z(G,a,\vphi)=\sum_{\Ms\in\M(G)} \prod_{wb\in \Ms}\vphi_{(w,b)}(a_{f(w,b)}-a_{f(b,w)}).
\end{equation}
It is thus the partition function of usual dimers, with edge weights
$(\vphi_{(w,b)}(a_{f(w,b)}-a_{f(b,w)}))_{wb\in E}$.
These weights need not be real positive numbers. However, if $\vphi$ is allowed to take complex values of modulus 1, and chosen to be equal to $\vphi_{(w,b)}= \frac{|a_{f(w,b)}-a_{f(b,w)}|}{a_{f(w,b)}-a_{f(b,w)}}$, then the edge weights are real and positive, equal to $|a_{f(w,b)}-a_{f(b,w)}|$. In the case where the variables $a$ are taken to be the complex positions of circle centers in a circle pattern, or t-embedding \cite{klrr,clrtembeddings}, then $\varphi$ is gauge equivalent to a Kasteleyn orientation, see~\cite{Kuperberg} and~\cite{klrr} for more details, and the forthcoming Proposition~\ref{prop:oriented_dimers_adjacency_matrix} also holds up to a complex constant of modulus 1.

\end{remark}

As we will now see, going back to $\varphi$ with values in $\{-1,1\}$, the quantity \eqref{equ:partition_function_2} depends on $\varphi$ only up to a global sign.

Let $K=(K_{w,b})$ be the weighted adjacency matrix of $G$, whose rows are indexed by white vertices of $W$, columns by black vertices of $B$, and whose non-zero entries are given by, for every edge $wb$ of $G$,
\begin{equation*}
K_{w,b} = a_{f(w,b)}-a_{f(b,w)}.
\end{equation*}
Then we prove the following.

\begin{proposition}\label{prop:oriented_dimers_adjacency_matrix}
  For every finite, bipartite graph $G$ with open faces,
  weights $a=(a_f)_{f\in F}$ on the faces, and Kasteleyn orientation $\varphi$, there exists
  $\epsilon(\varphi) \in \{-1,+1\}$ depending on $\varphi$ only, such that
  \begin{equation*}
    Z(G,a,\varphi) = \epsilon(\varphi) \det(K).
  \end{equation*}
\end{proposition}

\begin{proof}
  Using the alternative expression~\eqref{equ:partition_function_2} and the Kasteleyn theory~\cite{Kasteleyn}, see also~\cite{TemperleyFisher,Percus}, we know that, up to a sign depending on $\vphi$ only, the partition function $Z(G,a,\varphi)$ is equal to the determinant of the Kasteleyn matrix corresponding to $\vphi$, which is the weighted adjacency matrix whose non-zero coefficients are given by, for every edge $wb$,
\[
\vphi_{(w,b)}\cdot \vphi_{(w,b)} (a_{f(w,b)}-a_{f(b,w)})=a_{f(w,b)}-a_{f(b,w)},
\]
\emph{i.e.}, it is equal to the determinant of the matrix $K$.
\end{proof}

\begin{definition}
  Consider a finite,
  bipartite graph $G$ with open faces, together with face weights $a=\left( a_f \right)_{f \in F}$ in $\hat{\C}$, and a Kasteleyn orientation $\varphi$.
  When it is well-defined in $\hat{\C}$,
  we denote the \emph{ratio function of oriented
    dimers} as
  \begin{equation}
    \label{eq:defY}
    Y(G,a) = C(G,a) \
    \frac{Z(G,a^{-1},\varphi)}{Z(G,a,\varphi)},
  \end{equation}
  where $a^{-1} = \bigl( a^{-1}_f \bigr)_{f\in F}$ and
  \begin{equation*}
    C(G,a) = i^{|V|} \prod_{f \in F^i}a_{f}^{\frac{d(f)}{2}-1}
    \prod_{f \in F^o}a_{f}^{\lceil \frac{d(f)}{2} \rceil}.
  \end{equation*}
  By Proposition~\ref{prop:oriented_dimers_adjacency_matrix}, the
  ratio in Equation~\eqref{eq:defY} does not depend on $\varphi$, hence the
  same goes for $Y(G,a)$. The normalization $C(G,a)$, as we will
  see in the next paragraph, is such that $Y(G,a)$ is invariant
  under several local modifications of the weighted graph $(G,a)$.
\end{definition}

We can now precisely state the main result of this section, which is the
pendent of Speyer's Theorem~\ref{theo:speyer} in the case of the
dSKP recurrence. Note that it involves a ratio of partition functions rather than only a partition function as in~\cite{Speyer}.

\begin{theorem}
  \label{theo:dskp_dimers}
  Let $x: \calL \rightarrow \hat{\C}$ be a function that satisfies
  the dSKP recurrence. Let $h$ be a height function, and $\calI, \calU$
  be defined as in Equations~\eqref{eq:defI},~\eqref{eq:defU}. Let
  $(a_{i,j})=(x(i,j,h(i,j)))$ be the initial data indexed by points of $\calI$. Then for every point $p$ of $\calU$,
  \begin{equation*}
    x(p) = Y(G_p,a),
  \end{equation*}
  where $G_p$ is the crosses-and-wrenches graph corresponding to $p$.
\end{theorem}

Before proving Theorem~\ref{theo:dskp_dimers}, let us illustrate this theorem
with an example.
\begin{example}[Aztec diamond]
  \label{ex:Aztec}
  Consider the height function $h:\Z^2 \to \{0,1\}$, given by
  $h(i,j)=[i+j]_2$. Let $x:\calL \to \hat{\C}$ be a function that satisfies the
  dSKP recurrence. The initial data are again $(a_{i,j})=(x(i,j,h(i,j)))$. For
  $p=(i,j,k+1) \in \calL$, let us explicitly describe  $x(p)$ in terms of the
  initial data.

  In the cross and wrenches construction, this height function
  $h$ only produces crosses. The crosses-and-wrenches graph $G_p$ is commonly known as the
  \emph{Aztec diamond} of \emph{size} $k$,
  where the
  \emph{size} is the number of squares per ``side'' of the
  Aztec diamond, and the
  central face is at $(i,j)$; in this case, the graph $G_p$ is commonly denoted by $A_k$; see Figure~\ref{fig:ex_Aztec}. Using
  Theorem~\ref{theo:dskp_dimers} and computing the prefactor in Equation~\eqref{eq:defY},
  we get that for $k\geq 1$, the value of $x(p)$ is
  \begin{equation}
    \label{eq:dskp_Aztec_dim}
    x(p) = Y(A_k,a) = \biggl(\,\prod_{f\in F_p}a_f\biggr) \frac{Z(A_k,a^{-1},\varphi)}{Z(A_k,a,\varphi)},
  \end{equation}
  where the product is over all faces of $A_k$, internal or open, and the values
  of $a$ are displayed in Figure~\ref{fig:ex_Aztec}.
  \begin{figure}[h]
    \centering
    \def\svgwidth{15cm} \begin{footnotesize}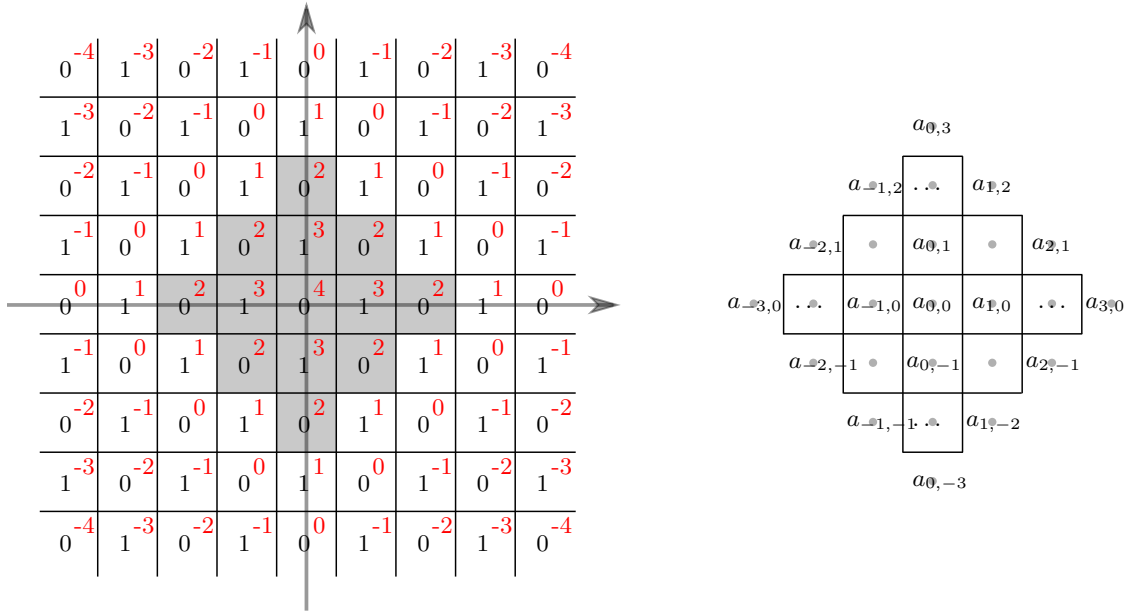\end{footnotesize}
    \caption{Left: the cross and wrenches graphs in Example~\ref{ex:Aztec},
      for $k=3, p=(0,0,4)$. The
      height function $h$ is shown in black, and the function defining
      $\calC_{p}$ is shown in red, giving the finite graph with open
      faces $G_p$ (gray); in this case, it is an Aztec diamond of
      size $k=3$ centered at face $(i,j)=(0,0)$, denoted by
      $A_k$. Right: the same graph with open faces
      with its face weights $a$; every gray dot
      represents a face (inner or open) and is equipped with a
      weight.
  }
    \label{fig:ex_Aztec}
  \end{figure}

\begin{remark}\label{rem:Aztec_diamond}
Note that oriented dimer configurations and monomials of $Z(A_k,a,\vphi)$ in the $a$ variables in the denominator of Equation~\eqref{eq:dskp_Aztec_dim} are not in one-to-one correspondence. The same of course also holds true for the numerator of~\eqref{eq:dskp_Aztec_dim}. For
instance, for $k=1$ there are $8$ oriented dimer configurations but
$6$ monomials; for
$k=2$ there are $512$ configurations but $220$ monomials; for
$k=3$ there are $262144$ configurations but $49224$
monomials. Unfortunately the sequence of number of monomials is not in
OEIS.

What happens is that several oriented dimer configurations cancel each
other. An example is when a square face is surrounded by two clockwise
dimers. Changing these dimers by the other two edges, oriented
clockwise, has the effect of negating the weight. As a result, the
variables $a$ can only appear with exponent $1$ in the monomials.

Finding a model that gives such a one-to-one correspondence is one of the goals of Section~\ref{sec:comb_sol_II}, and the final statement for the
Aztec diamond is given in Corollary~\ref{cor:Z_Aztec_diamond}.
  \end{remark}
\end{example}

The proof of Theorem~\ref{theo:dskp_dimers} is the subject of the next
three sections. The method follows that of Speyer~\cite{Speyer}. The
first part, Section~\ref{sec:infinite}, relates dimers on $G_p$ to
dimers on an infinite graph, with some asymptotic conditions; this
trick is useful to get rid of issues at the boundary. The second part,
Section~\ref{sec:invariance}, consists in proving that $Y(G,a)$ is
invariant under natural modifications of the underlying graph and
weight function. The third part, Section~\ref{sec:invariance}, relies
on the first two and is an induction argument on the height
functions. Note that the main contributions of
Theorem~\ref{theo:dskp_dimers} are the identification of the function
$Y(G_p,a)$ satisfying the invariance relations and handling ratios of
partition functions in the proof.

\subsection{Infinite completions}
\label{sec:infinite}

We follow \cite[Section~4.1]{Speyer} for the following definition.
Consider a height function $h$, a point $p=(i,j,k) \in \calU_h$, and
introduce the function
\begin{equation}
  \label{eq:defht}
  h_p(i',j') =
  \min\left(h(i',j'),k-|i-i'|-|j-j'|\right).
\end{equation}
Note that the minimum is between $h$ and the function defining
$\calC_p$. Then $h_p$ is not a height function, as it does not satisfy
Condition~2. but, since it still satisfies Condition 1.,
we may produce an infinite graph $\calG_p$ from $h_p$ using the method
of crosses and wrenches\footnote{In Speyer's paper, the
    notation $\tilde{\calG}$ is used for the graph that we call
    $\calG_p$ here.}. Then
$G_p$ is also a subgraph of $\calG_p$, consisting of
  hexagons at face distance 1 from $G_p$, see
Figure~\ref{fig:infinite}.

\begin{figure}[tbp]
  \centering
  \def\svgwidth{6cm} 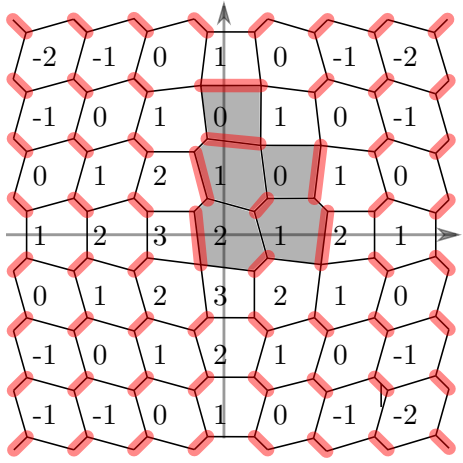
  \caption{The function $h_p$ and the corresponding infinite
    graph $\calG_p$, using the same example as
    Figure~\ref{fig:cross_wrench}. The subgraph $G_p$ is colored in
    gray. An acceptable dimer configuration is shown in red.}
  \label{fig:infinite}
\end{figure}

A perfect matching $\Ms$ on $\calG_p$ is said to be
\emph{acceptable} if there exists a finite subgraph outside of which
$\Ms$ contains only the middle edge of every wrench. We denote by
$\M_\infty\bigl( \calG_p \bigr)$ the set of acceptable perfect
matchings of $\calG_p$. By \cite[Proposition~6]{Speyer},
acceptable perfect matchings on $\calG_p$ are in bijection with
perfect matchings of $G_p$, and can in fact always be obtained by extending a
perfect matching of $G_p$ to $\calG_p$ using all wrenches of
$\calG_p \setminus G_p$; see also Figure~\ref{fig:infinite}.

The introduction of ratios of partition functions on the infinite graph
$\calG_p$ requires a bit more care than in \cite{Speyer}, as we
would like to develop and factor infinite products of face weights in
two partition functions, and then simplify them; these last steps were
not required in \cite{Speyer}. We proceed in the following way: fix an
acceptable perfect matching
$\Ms_0 \in \M_\infty\bigl( \calG_p \bigr)$, then for any
$\Ms \in \M_\infty\bigl( \calG_p \bigr)$, $\Ms_0$ and $\Ms$
differ only at a finite number of edges, hence the following
\emph{ratio of weights} is well-defined:
\begin{equation}
  \label{eq:defwinf}
  \frac{w_\infty(\Ms)}{w_\infty(\Ms_0)} = \frac{\prod_{wb \in \Ms
      \setminus \Ms_0} \varphi_{(w,b)}\left( a_{f(w,b)}-a_{f(b,w)}
    \right)}{\prod_{wb \in \Ms_0 \setminus \Ms} \varphi_{(w,b)}\left( a_{f(w,b)}-a_{f(b,w)} \right)}.
\end{equation}
We define the partition function relative to $\Ms_0$ as
\begin{equation}
  \label{eq:defzinf}
  Z_\infty\left(\calG_p,a,\varphi,\Ms_0 \right) = \sum_{\Ms \in
    \M_\infty\left( \calG_p \right)} \frac{w_\infty(\Ms)}{w_\infty(\Ms_0)}.
\end{equation}
Let $\calF_p$ be the set of faces of $\calG_p$. For any
$f\in\calF_p$, let $d_{\Ms_0}(f)$ be the number of dimers in
$\Ms_0$ adjacent to $f$. Note
that all but a finite number of faces $f \in \calF_p$ have degree
$6$ and are such that $d_{\Ms_0}(f)=2$. This implies that the
prefactor in the following expression is well-defined:
\begin{equation}
  \label{eq:defyinf}
  Y_\infty\left(\calG_p,a \right) = \left(\prod_{f\in
    \calF_p}a_f^{\frac{d(f)}{2} - 1 - d_{\Ms_0}(f)}\right) \frac{Z_\infty\left(\calG_p,a^{-1},\varphi,\Ms_0 \right)}{Z_\infty\left(\calG_p,a,\varphi,\Ms_0 \right)}.
\end{equation}
As suggested by the notation, $Y_\infty\left(\calG_p,a \right)$
does not depend on $\Ms_0$ nor on $\varphi$; this will be a consequence of the forthcoming
Proposition~\ref{prop:inf}.

The main point of this construction is that, since $\calG_p$ has no outer face, the quantity $Y_\infty$ treats
internal and outer faces in the same way. Together with the next proposition, which states that using $\calG_p$ we recover the usual ratio of partition functions on $G_p$, this allows us to avoid tedious case handling at the boundary.

\begin{proposition}
 \label{prop:inf}
  Let $h$ be a height function, let $p\in \calU_h$, let
  $h_p$ be defined by \eqref{eq:defht}, and let $\calG_p$
  be the crosses-and-wrenches graph corresponding to $h_p$. Then
  \begin{equation*}
    Y_\infty\left(\calG_p,a \right) = Y(G_p,a).
  \end{equation*}
\end{proposition}

\begin{proof}
  Recall that $E_p$ are the edges of the finite subgraph $G_p$. Again
  by \cite[Proposition~6]{Speyer},
  \begin{equation*}
    Z_\infty\left(\calG_p,a,\varphi,\Ms_0 \right) =
    \frac{Z(G_p,a,\varphi)}{\prod_{wb \in \Ms_0 \cap E_p}\varphi_{(w,b)} \left( a_{f(w,b)}-a_{f(b,w)} \right)}.
  \end{equation*}
  Doing the same for face weights $a^{-1}$ and taking the ratio, we get
  \begin{equation}
    \label{eq:ratiozinf}
    \begin{split}
      \frac{Z_\infty\left(\calG_p,a^{-1},\varphi,\Ms_0
        \right)}{Z_\infty\left(\calG_p,a,\varphi,\Ms_0 \right)}
      = & \biggl(\,\prod_{wb \in \Ms_0 \cap E_p}
        \frac{a_{f(w,b)}-a_{f(b,w)}}{a^{-1}_{f(w,b)}-a^{-1}_{f(b,w)}}\biggr)
      \frac{Z(G_p,a,\varphi)}{Z(G_p,a^{-1},\varphi)}, \\
      = & (-1)^{|\Ms_0 \cap E_p|} \biggl(\,\prod_{wb \in \Ms_0 \cap E_p}
        a_{f(w,b)}a_{f(b,w)}\biggr)
      \frac{Z(G_p,a,\varphi)}{Z(G_p,a^{-1},\varphi)}.
    \end{split}
  \end{equation}
  Therefore,
  \begin{equation*}
    Y_\infty\left(\calG_p,a \right) = (-1)^{|\Ms_0 \cap E_p|} \biggl(\,\prod_{f\in
        \calF_p}a_f^{\frac{d(f)}{2} - 1 - d_{\Ms_0}(f)}\biggr) \biggl(\,\prod_{wb \in \Ms_0 \cap E_p}
      a_{f(w,b)}a_{f(b,w)}\biggr) \frac{Z(G_p,a,\varphi)}{Z(G_p,a^{-1},\varphi)}.
  \end{equation*}
  We claim that
  \begin{equation}
    \label{eq:eqprefact}
    (-1)^{|\Ms_0 \cap E_p|} \biggl(\,\prod_{f\in
        \calF_p}a_f^{\frac{d(f)}{2} - 1 - d_{\Ms_0}(f)}\biggr) \biggl(\,\prod_{wb \in \Ms_0 \cap E_p}
      a_{f(w,b)}a_{f(b,w)}\biggr) = i^{|V_p|} \prod_{f \in F_p^i}a_{f}^{\frac{d(f)}{2}-1}
    \prod_{f \in F_p^o}a_{f}^{\lceil \frac{d(f)}{2} \rceil},
  \end{equation}
  which implies that
  $Y_\infty\left(\calG_p,a \right) = Y(G_p,a)$. First, $\Ms_0$
  reduces to a perfect matching of $G_p$, so
  $|M_0\cap E_p|=\frac{|V_p|}{2}$, proving that the complex prefactors
  are the same.
    Then, let $f\in\calF_p$. If $f\notin F_p$, as
  argued previously, $a_f$ has exponent $0$ in the left-hand side of
  \eqref{eq:eqprefact}, in accordance with the right-hand side. If
  $f \in F_p^i$, then the second product on the left-hand side of
  \eqref{eq:eqprefact} produces a factor $a_f^{d_{\Ms_0}(f)}$,
  simplifying with the first product to give $a_f^{\frac{d(f)}{2}-1}$
  as on the right-hand side. Finally, if $f\in F_p^o$, then by the
  same argument, the left-hand side gives an exponent
  $\frac{d(f)}{2}-1-d_{\Ms_0 \setminus E_p}(f)$, which exactly corresponds
  to the normalization obtained for acceptable dimers in
  \cite[Section~4.1]{Speyer} (we recall that in Speyer's case, the
  weight of an edge $wb$ is $a^{-1}_{f(w,b)}\cdot a^{-1}_{f(b,w)}$, so the
  dimers in $\Ms$ outside of $G_p$ contribute with a factor
  $a_f^{-d_{\Ms \setminus E_p}(f)}$). By Speyer's computation, this is
  equal to the normalization in $C_{\mathrm{dim}}(G_p,a)$ for this
  face, which is also $\lceil \frac{d(f)}{2} \rceil$.
\end{proof}

\subsection{Invariance of ratio function of oriented dimers}
\label{sec:invariance}

In this section, $G$ is a finite bipartite graph with open faces $F=F^i\cup F^o$,
equipped with face weights $a=\left( a_f \right)_{f\in F}$.
We will often identify the names of faces with the weights
attached to it.

\subsubsection{Contraction/Expansion of a vertex of degree 2}

In the graph $G$, consider a vertex $v$ adjacent to at least
two distinct \emph{inner} faces with respective weights $a_1,a_2$. A
new graph $G'$ can be obtained by replacing $v$ with two vertices $v_1,v_2$
joined by a vertex $u$ of degree~$2$, such that the two new edges separate
$a_1$ from $a_2$; see Figure~\ref{fig:contraction}.

This produces $G'$ which is also bipartite, and
naturally equipped with face weights still denoted by $a$. The
graphs $G$ and $G'$ are said to be related by the
\emph{contraction/expansion} of a vertex of degree~$2$.

\begin{figure}[tb]
  \centering
  \def\svgwidth{5cm} 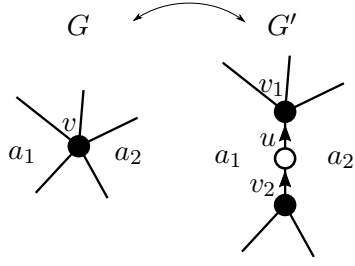
  \caption{Local operation of contracting/expanding a vertex of degree
  $2$, with an example Kasteleyn orientation on the expanded side.}
  \label{fig:contraction}
\end{figure}

\begin{proposition}
  \label{prop:contraction}
  Let $G$, $G'$ be two graphs as above related by the contraction/expansion of a vertex of
  degree~$2$. Then
  \begin{equation*}
    Y(G,a) = Y(G',a).
  \end{equation*}
\end{proposition}

\begin{proof}
  Suppose that the vertex $v$ is black, the case where $v$ is white
  being similar. Let $\varphi$ be a Kasteleyn orientation on
  $G$. We can get a Kasteleyn orientation of $G'$, also denoted $\vphi$, by setting
  $\vphi_{(u,v_1)}=1$, $\vphi_{(u,v_2)}=-1$,
  see Figure~\ref{fig:contraction}. As $Y(G',a)$ does not depend on the Kasteleyn
  orientation, we can use $\varphi$ in the proof.

  We claim that
  $Z(G',a,\varphi) = (a_2-a_1) Z(G,a,\varphi)$. Indeed,
  using Expression~\eqref{equ:partition_function_2} for the partition
  function, in a perfect matching of $G'$, $u$ has to be matched
  either to $v_1$ or to $v_2$. In the first case, this gives a
  contribution $(a_2-a_1)$ that factors in the corresponding sub-sum
  of $Z(G',a,\varphi)$, and in the second case, it gives a
  contribution $-1(a_1-a_2)=a_2-a_1$ to the second sub-sum. As the sum
  of these two sub-sums is $Z(G,a,\varphi)$, we get the claim.

  Therefore,
  $Z(G',a^{-1},\varphi) = \left(a_2^{-1} - a_1^{-1} \right)
  Z(G,a^{-1},\varphi)$, from which we get
  \begin{equation*}
    -a_1 a_2 \frac{Z(G',a^{-1},\varphi)}{Z(G',a,\varphi)} = \frac{Z(G,a^{-1},\varphi)}{Z(G,a,\varphi)}.
  \end{equation*}
  Accounting for the discrepancy of degree of faces $a_1,a_2$ between $G$
  and $G'$, it is straightforward to check that
  $Y(G,a)=Y(G',a)$. Note that this is where we use the hypothesis that faces
  $a_1,a_2$ are inner faces.
\end{proof}

\subsubsection{Spider move}

We state the central invariance result, which is the application
of a \emph{spider move}, also known as a \emph{square move}, or
\emph{urban renewal}~\cite{ProppUR,CiucuUR}.
Suppose that $G$
contains a face of degree four, with
vertices $\{v_1,v_2,v_3,v_4\}$ in counterclockwise cyclic order, surrounded
by four \emph{distinct} (inner or open) faces; the vertices
$v_i$ may belong to $\Bound$ or not. Then we replace this square
with a smaller one surrounded by four edges, as in
Figure~\ref{fig:urban}, and obtain a graph denoted by $G'$.

Suppose that $G$ has face weights $a$; denote by $a_{0,0}$ the weight
of the center face, and by $a_{-1,0}, a_{0,-1},a_{1,0},a_{0,1} $ the
weights at the four boundary faces. Then, we set $G'$ to have a weight
function $a'$ equal to $a$ everywhere except at the center face where
it is equal to $a_{0,0}'$.

\begin{definition}\label{def:DSKP_relation}
  Under the above assumptions, the weight functions $a$ and $a'$ are
  said to satisfy the \emph{dSKP relation if}:
  \begin{equation}\label{equ:DSKP_relation}
    \frac{(a_{0,0}-a_{0,1})(a_{-1,0}-a_{0,0}')(a_{0,-1}-a_{1,0})
    }{(a_{0,1}-a_{-1,0})(a_{0,0}'-a_{0,-1})(a_{1,0}-a_{0,0})}=-1.
  \end{equation}
\end{definition}
\begin{remark}\label{rem:DSKP_relation_recurrence}
  Note that by setting $x_{i,j,0}=a_{i,j}$ whenever $i$ or $j$ is equal to $\pm 1$, and
  $x_{-e_3}=x_{0,0,-1}=a_{0,0}$, $x_{e_3}=x_{0,0,1}=a_{0,0}'$, we recover the dSKP recurrence of Definition~\ref{def:dSKP_recurrence} evaluated at the point $p=(0,0,0)$.
\end{remark}

\begin{figure}[h]
  \centering
  \def\svgwidth{13cm} 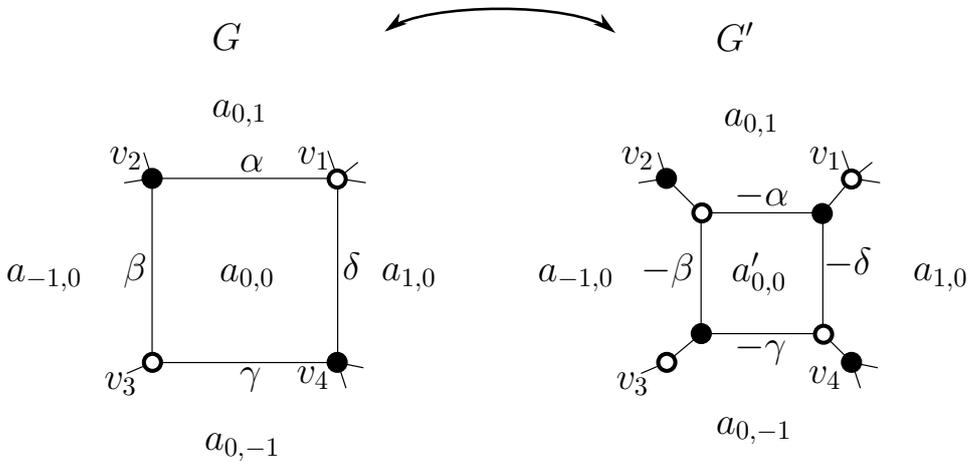
  \caption{The ``spider move'' transformation; the
    Greek letters represent the Kasteleyn orientation, and non-indexed
  edges on the right-hand side are given orientation $1$, \emph{i.e.}, from
  white to black.}
  \label{fig:urban}
\end{figure}

\begin{proposition}
  \label{prop:urban}
  Let $G,G'$ be two graphs as above related by a spider move, and
  suppose that the respective weight functions $a$ and $a'$ (as in
  Figure~\ref{fig:urban}) satisfy the dSKP
  relation~\eqref{equ:DSKP_relation}. Then
  \[
    Y(G,a)=Y(G',a').
  \]
\end{proposition}

\begin{proof}
  Fix a Kasteleyn orientation $\varphi$ on $G$. We get a Kasteleyn
  orientation on $G'$ as in Figure~\ref{fig:urban}, by multiplying
  $\varphi$ by $-1$ on the square and setting it to $1$ on the four newly
  created edges; we denote it by $\varphi'$.

  Consider a perfect matching on $G$ or $G'$. In both cases, among
  $\{v_1,v_2,v_3,v_4\}$, those matched inside the center region are
  either all of them, none of them, or some $\{v_i,v_{i+1}\}$ (taken
  cyclically); see Figure~\ref{fig:urban2}. We partition $Z(G,a,\varphi)$
  and $Z(G',a',\varphi')$, each into six sub-sums, depending on these six
  cases. We show that, for each case, the sub-sum of $Z(G,a,\varphi)$ is
  proportional to that of $Z(G',a',\varphi')$, with a common factor
  \[
    \lambda = \alpha \gamma
    \frac{a_{0,1}a_{0,-1}  - a_{1,0}a_{-1,0} - a_{0,0}a_{0,1} + a_{0,0}a_{-1,0} -
      a_{0,0}a_{0,-1} + a_{0,0}a_{1,0}
    }{(a_{0,1}-a_{1,0})(a_{0,1}-a_{-1,0})(a_{0,-1}-a_{-1,0})(a_{0,-1}-a_{1,0})}.
  \]

  \begin{figure}[h]
    \centering
    \def\svgwidth{6cm} 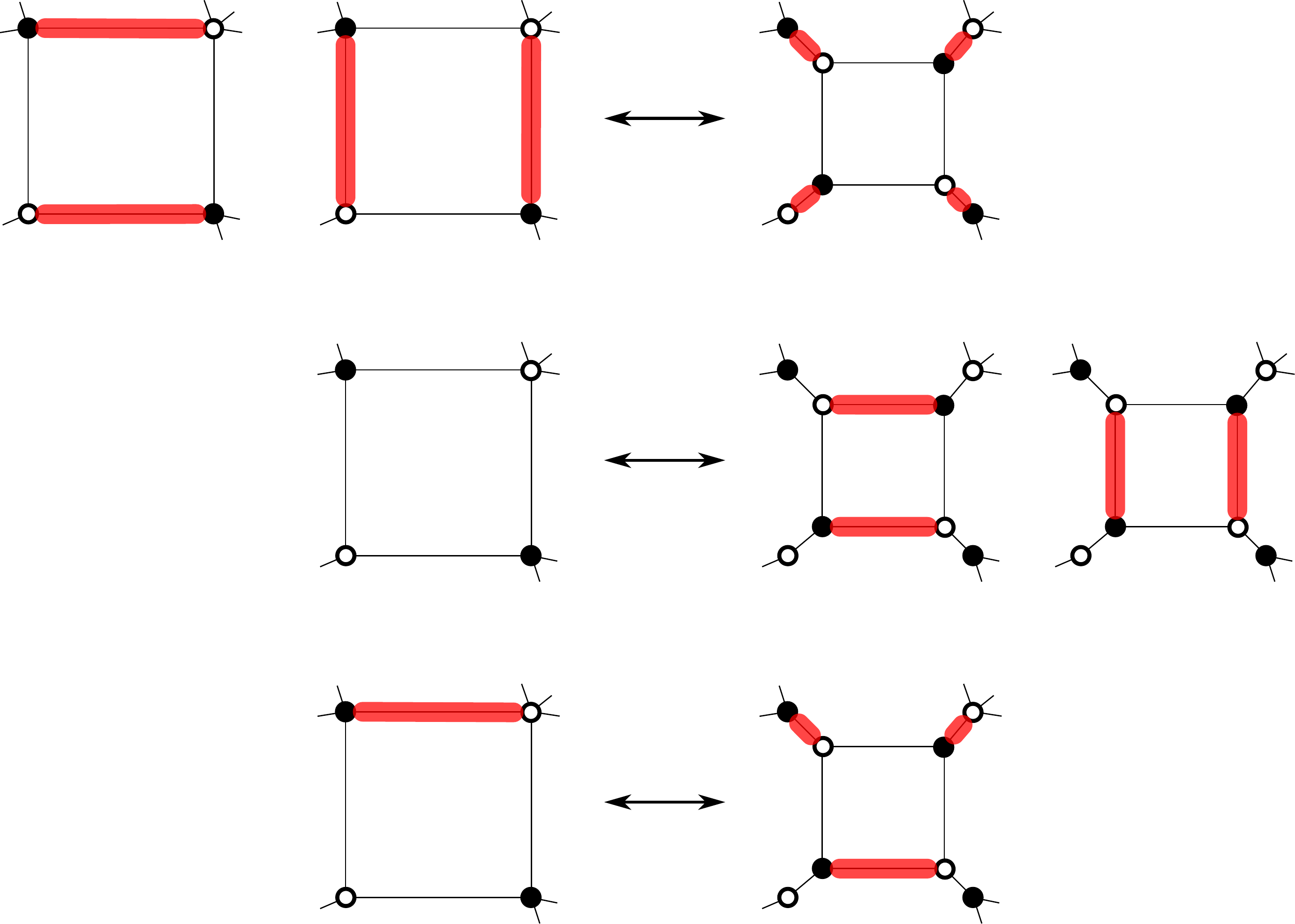
    \caption{Corresponding matchings between $G$ and $G'$. Three extra
      symmetries of the last row have to be considered.}
    \label{fig:urban2}
  \end{figure}

  Consider the first case (all of the $v_i$ are matched internally). In the
  sub-sum of $Z(G,a,\varphi)$, taking into account the possible
  orientations of each dimer, we can factor in
  \begin{equation}\label{proof:urban_renewal_1}
  \alpha(a_{0,1}-a_{0,0})\gamma(a_{0,-1}-a_{0,0}) +
  \beta(a_{0,0}-a_{-1,0})\delta(a_{0,0}-a_{1,0}).
  \end{equation}
  In that of
  $Z(G',a',\varphi')$, we can factor in
  \begin{equation}\label{proof:urban_renewal_2}
  (a_{0,1}-a_{1,0})(a_{0,1}-a_{-1,0})
  (a_{0,-1}-a_{-1,0})(a_{0,-1}-a_{1,0}).
  \end{equation}

  After
  this factorization, what is left of both sub-sums is equal. We claim
  that the term~\eqref{proof:urban_renewal_1} is equal to $\lambda$ times the term \eqref{proof:urban_renewal_2}, which implies that the two sub-sums indeed differ by a
  multiplicative factor $\lambda$.  This relation can be checked by a
  direct computation, using the fact that
  $\alpha \beta \gamma \delta = -1$.

  The same can be done in the other five cases, with the same constant $\lambda$
  appearing; we omit computations here: they only use
  Equation~\eqref{equ:DSKP_relation} and polynomial manipulations. By summing all
  cases, this implies
  \begin{equation*}
  \frac{Z(G,a,\varphi)}{Z(G',a',\varphi')} =
  \alpha \gamma
  \frac{a_{0,1}a_{0,-1}  - a_{1,0}a_{-1,0} - a_{0,0}a_{0,1} + a_{0,0}a_{-1,0} -
    a_{0,0}a_{0,-1} + a_{0,0}a_{1,0}
    }{(a_{0,1}-a_{1,0})(a_{0,1}-a_{-1,0})(a_{0,-1}-a_{-1,0})(a_{0,-1}-a_{1,0})}
  .
  \end{equation*}

  Therefore,
  \begin{equation*}
  \frac{Z(G,a^{-1},\varphi)}{Z(G',a'^{-1},\varphi')}=
  \alpha \gamma
  \frac{a_{0,1}^{-1}a_{0,-1}^{-1}  - a_{1,0}^{-1}a_{-1,0}^{-1} - a_{0,0}^{-1}a_{0,1}^{-1} + a_{0,0}^{-1}a_{-1,0}^{-1} -
    a_{0,0}^{-1}a_{0,-1}^{-1} + a_{0,0}^{-1}a_{1,0}^{-1}
    }{(a_{0,1}^{-1}-a_{1,0}^{-1})(a_{0,1}^{-1}-a_{-1,0}^{-1})(a_{0,-1}^{-1}-a_{-1,0}^{-1})(a_{0,-1}^{-1}-a_{1,0}^{-1})}
  .
  \end{equation*}
  By taking the ratio of the last two equations, and after some
  computations that again use Equation~\eqref{equ:DSKP_relation}, we get
  \begin{equation*}
    \frac{Z(G,a^{-1},\varphi)}{Z(G,a,\varphi)} =
    \frac{a_{0,1}a_{-1,0}a_{0,-1}a_{1,0}a_{0,0}}{a_{0,0}'}\
    \frac{Z(G',a'^{-1},\varphi')}{Z(G',a',\varphi')}.
  \end{equation*}
  Accounting for the degree of each face in the prefactors $C(G,a)$
  and $C(G',a')$, this directly gives $Y(G,a)=Y(G',a')$.
\end{proof}

 We now extend the two previous invariance results to
  the setting of infinite graphs. Recall the infinite graph $\calG_p$
  of Section~\ref{sec:infinite} and its weight function $a$. A series
  of contractions/expansions of vertices of degree $2$, and spider moves, can also be defined on $\calG_p$, giving a new graph
  $\calG_p'$, and the dSKP relation~\eqref{equ:DSKP_relation} allows
  one to define its weight function $a'$. One can still use
  \eqref{eq:defwinf},~\eqref{eq:defzinf},~\eqref{eq:defyinf} to define
  quantities $Z_\infty(\calG_p',a')$ and $Y_\infty(\calG_p',a')$ on
  this new weighted graph. We claim that the invariance also holds in
  this infinite setting:

\begin{corollary}
  \label{cor:inv_inf}
  Let $\calG_p'$ be obtained from $\calG_p$ by a finite number of
  contractions/expansions of vertices of degree $2$, and of spider moves with the weight functions $a', a$ satisfying the dSKP relation
  \eqref{equ:DSKP_relation}. Then
  \begin{equation}
    \label{eq:inv_inf}
    Y_\infty(\calG_p',a') = Y_\infty(\calG_p,a).
  \end{equation}

\end{corollary}

\begin{proof}
We first suppose that $\calG_p'$ is obtained from
    $\calG_p$ by expanding a vertex $v$ of degree $2$, with the same
    notation as Figure~\ref{fig:contraction}. We fix acceptable
    perfect matchings $\Ms_0$ (resp. $\Ms'_0$) on $\calG_p$
    (resp. $\calG_p'$); note that they may differ only at a finite
    number of edges. Then the same argument as in
    Proposition~\ref{prop:contraction} give
  \begin{equation*}
    Z_\infty(\calG_p',a,\varphi,\Ms'_0) = (a_2-a_1) \ \frac{w_\infty(\Ms_0)}{w_\infty(\Ms'_0)} Z_\infty(\calG_p,a,\varphi,\Ms_0).
  \end{equation*}
  Doing the same for $a^{-1}$ and taking the ratio, we get after
  computations analogous to \eqref{eq:ratiozinf}
  \begin{equation*}
    -a_1a_2 \biggl(\,\frac{\prod_{wb \in \Ms_0 \setminus
        \Ms'_0}(-a_{f(w,b)}a_{f(b,w)})}{\prod_{wb \in \Ms'_0 \setminus
        \Ms_0}(-a_{f(w,b)}a_{f(b,w)})}\biggr)
    \frac{Z_\infty(\calG_p',a^{-1},\varphi,\Ms'_0)}
    {Z_\infty(\calG_p',a,\varphi,\Ms'_0)} =
    \frac{Z_\infty(\calG_p,a^{-1},\varphi,\Ms_0)}
    {Z_\infty(\calG_p,a,\varphi,\Ms_0)}.
  \end{equation*}
  After putting in the prefactors, we get that
  $Y_\infty(\calG_p,a) = Y_\infty(\calG_p',a)$ is equivalent to
  \begin{equation*}
    -a_1a_2 \biggl(\,\frac{\prod_{wb \in \Ms_0 \setminus
          \Ms'_0}(-a_{f(w,b)}a_{f(b,w)})}{\prod_{wb \in \Ms'_0 \setminus
          \Ms_0}(-a_{f(w,b)}a_{f(b,w)})}\biggr)
    \biggl(\,\prod_{f\in
        \calF_p}a_f^{\frac12({d_{\calG_p}(f) - d_{\calG_p'}(f)}) - d_{\Ms_0}(f) +
        d_{\Ms'_0}(f)}\biggr) = 1.
  \end{equation*}
  To check this last equation, first note that $|\Ms'_0|=|\Ms_0|+1$,
  which implies that
  $\left[ |\Ms'_0 \setminus \Ms_0| \right]_2 \neq \left[|\Ms_0 \setminus \Ms'_0|\right]_2$, and this implies that the factors $-1$ cancel out. Then, the
  only faces that have a different degree in $\calG_p,\calG_p'$ are $a_1$ and
  $a_2$, and for these the factor $a_1a_2$ cancels
  $a_f^{\frac12({d_{\calG_p}(f) - d_{\calG_p'}(f)})}$. Finally, the ratio of
  products is equal to
  $\prod_{f\in \calF_p}a_f^{ d_{\Ms_0}(f) -d_{\Ms'_0}(f)}$, canceling out
  with the remaining part of the product over $\calF_p$.

  Now suppose that $\calG_p$ and $\calG_p'$ are related by
  a spider move, with the respective weight functions $a,a'$ satisfying the dSKP
  equation, with the notation of Figure~\ref{fig:urban}. Similarly,
  the same argument as Proposition~\ref{prop:urban} give
  \begin{align*} 
    \frac{Z_\infty(\calG_p,a^{-1},\varphi,\Ms_0)}
    {Z_\infty(\calG_p,a,\varphi,\Ms_0)} &=
    \frac{a_{0,1}a_{-1,0}a_{0,-1}a_{1,0}a_{0,0}}{a_{0,0}'}
    \\ &\phantom{=}  \cdot \biggl(\,\frac{\prod_{wb \in \Ms_0 \setminus
          \Ms'_0}(-a_{f(w,b)}a_{f(b,w)})}{\prod_{wb \in \Ms'_0 \setminus
          \Ms_0}(-a'_{f(w,b)}a'_{f(b,w)})}\biggr)
     \frac{Z_\infty(\calG_p',a'^{-1},\varphi',\Ms'_0)}
    {Z_\infty(\calG_p',a',\varphi',\Ms'_0)}. %
  \end{align*}
  Then one concludes exactly as in the previous case: the contributions of
  $\Ms_0,\Ms'_0$ cancel out with the prefactor.

  Successive applications of the previous two operations
    show that Equation~\eqref{eq:inv_inf} holds.
\end{proof}

\subsection{Proof of Theorem~\ref{theo:dskp_dimers}}\label{sec:proof_main_thm}

With the propositions of
Sections~\ref{sec:infinite} and \ref{sec:invariance} proved, the
argument now follows Speyer's Proof~I of the Main Theorem~\cite{Speyer}.
Consider a point $p=(i_0,j_0,k_0)$ of $\calL$ and the closed square
cone $\calC_p$ defined in Equation~\eqref{eq:defC}. Given a height
function $h$ recall the definition of $\calU_{h}$, the upper set
corresponding to $h$, defined in Equation~\eqref{eq:defU}.
The proof is by induction on
$\# \left( \calU_h \cap \calC_p \right)$, where $h$ is a height
function such that $p\in$ $\calU_h$.

$\bullet$ If $\# \left( \calU_h \cap \calC_p \right) = 1$ (its minimal
  value given that $p \in \calU_h$), then
  $\calU_h \cap \calC_p=\{p\}$. This implies that $h(i_0,j_0)=k_0-2$,
  and $h(i_0-1,j_0)=h(i_0+1,j_0)=h(i_0,j_0-1)=h(i_0,j_0+1)=k_0-1$; the
  values of $h$ at other points of $\Z^2$ do not affect the
  intersection $\calU_h \cap \calC_p$. Denote by $a$ the associated
  initial condition. Let us now return to the construction of the
  crosses-and-wrenches graph $G_p$, see
  Section~\ref{sec:crosses_and_wrenches}: the internal faces of $G_p$
  are indexed by points of $\mathring{\calC}_p\cap \calI$, that is by
  the unique vertex $(i_0,j_0,k_0-2)$ and has weight $a_{i_0,j_0}$;
  the open faces are indexed by points of $\calC_p\cap\calI$, that is
  by the four vertices $(i_0-1,j_0,k_0-1),\dots, (i_0,j_0+1,k_0-1)$,
  and have face weights $a_{i_0-1,j_0},\dots,a_{i_0,j_0+1}$. Then,
  using Proposition~\ref{prop:inf}, we know that
    \[ Y(G_p,a) = Y_\infty(\calG_p,a).\]
  where $\calG_p$ is the infinite graph shown in
  Figure~\ref{fig:induction} (left).
  \begin{figure}[tb]
    \centering
    \def\svgwidth{15cm} \begin{footnotesize}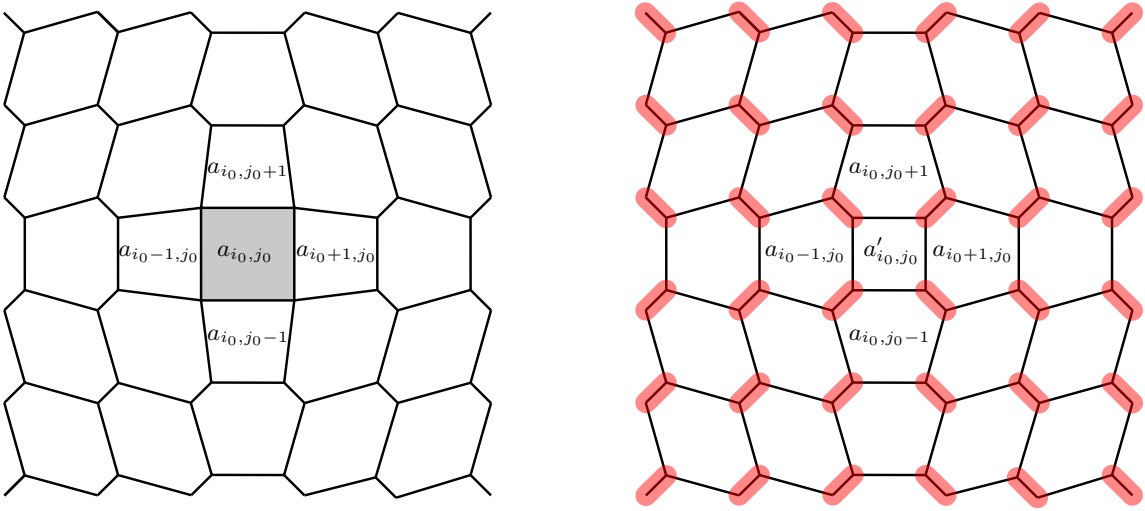\end{footnotesize}
    \caption{Left: the graph $\calG_p$ when $\# \left( \calU_h
        \cap \calC_p \right) = 1$; the finite graph with open faces
      $G_p$ is shown in gray. Right: the graph $\calG_p'$
      obtained after a spider move at $(i_0,j_0)$, and its unique
      acceptable perfect matching.}
    \label{fig:induction}
  \end{figure}
    We then apply a spider move at the face $a_{i_0,j_0}$ in
    $\calG_p$. This yields the graph $\calG_p'$ with weight function $a'$ equal to $a$ except at the center vertex where $a'_{i_0,j_0}$ is chosen so as to satisfy the dSKP
    relation~\eqref{equ:DSKP_relation} translated from $(0,0)$ to
  $(i_0,j_0)$.
    By Corollary~\ref{cor:inv_inf}, we get
    \[ Y(G_p,a) = Y_\infty(\calG_p',a'). \]
  On $\calG_p'$ there is only one
  acceptable perfect matching consisting
  of the four small edges; an explicit computation yields:
  \[
  Y_\infty(\calG_p',a')=a_{i_0,j_0}'.
  \]
  The proof is concluded by using
  Remark~\ref{rem:DSKP_relation_recurrence} to note that
  $a_{i_0,j_0}'$ is also equal to $x(i_0,j_0,k_0)$ the solution at $p$
  of the dSKP recurrence with initial condition $a$.

$\bullet$ If $\#(\calU_h \cap \calC_p) > 1$, then as argued in
    \cite[Section 5.3]{Speyer} there exists $(i,j)$ such that
    $(i,j,h(i,j)) \in \calC_p$ and
    $h(i-1,j)=h(i+1,j)=h(i,j-1)=h(i,j+1)=h(i,j)+1$. In other words,
    $(i,j)$ is a ``local minimum'', and we can define a height
    function $h'$ that is equal to $h$ everywhere except at $(i,j)$,
    where $h'(i,j)=h(i,j)+2$. Then
    $\#(\calU_{h'} \cap \calC_p) = \#(\calU_h \cap \calC_p) - 1$, and
    $p \in \calU_{h'}$ (otherwise we would have
    $\#(\calU_h \cap \calC_p) = 1$). Denote by $G_p'$ the crosses-and-wrenches graph corresponding to $h'$, and by $a'$ the initial condition equal to $a$ everywhere except at the point $(i,j)$ where it is equal to the solution $x(i,j,h(i,j)+2)$ at $(i,j,h(i,j)+2)$ of the dSKP recurrence with initial condition $a$. By induction we have
    \[
    Y(G_p',a')=x(p),
    \]
    where $x(p)$ is the solution at $p$ of the dSKP recurrence with
    initial condition $a'$ which, by our choice of initial condition
    $a'$, is equal to the solution at $p$ of the dSKP recurrence with
    initial condition $a$.  The proof is thus concluded if we can
    prove that $Y(G_p,a)=Y(G_p',a'),$ which by
      Proposition~\ref{prop:inf} is equivalent to
      $Y_\infty(\calG_p,a)=Y_\infty(\calG_p',a')$ where
      $\calG_p'$ is the infinite completion corresponding to
      $h'$.
    \begin{figure}[tb]
      \centering
      \def\svgwidth{6cm} 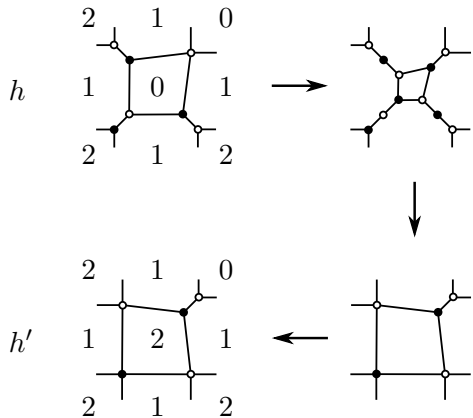
      \caption{Example local minimum of $h$. A spider move and
        three contractions are applied to get to the cross-wrenches
        graph of $h'$.}
      \label{fig:urban3}
    \end{figure}

    To see why this last equation holds, note that the effect of
    going from $h$ to $h'$ on the cross-wrenches graphs
    is exactly to perform a spider move at the face $(i,j)$,
    followed by contractions of vertices of degree $2$ around the
    newly created square; see Figure~\ref{fig:urban3} for an
    example. By Remark~\ref{rem:DSKP_relation_recurrence}, when
    performing a spider move, the corresponding weight functions
    $a$ and $a'$ satisfy the dSKP equation~\eqref{equ:DSKP_relation}
    translated from $(0,0)$ to $(i,j)$. Thus, by
    Corollary~\ref{cor:inv_inf}, we
    conclude that indeed
    $Y_\infty(\calG_p,a)=Y_\infty(\calG_p',a')$.
    \qed

\section{dSKP: combinatorial solution II - trees and forests}\label{sec:comb_sol_II}

Let us recall, see also Remark~\ref{rem:Aztec_diamond}, that in the case of the Aztec diamond of size $k$, oriented dimer configurations that occur in the expansion of $Z(A_k,a,\vphi)$ are not in one-to-one correspondence with the monomials in $a$. Several oriented dimer configurations cancel.

The goal of this section is to give a combinatorial interpretation of
the oriented dimer partition function $Z(A_k,a,\vphi)$, \emph{i.e.},
we introduce combinatorial objects, called \emph{complementary trees
  and forests}, that are in bijection with monomials in the expansion
of $Z(A_k,a,\vphi)$. Actually, the setting where this result can be
obtained is more general that that of the Aztec diamond. In
Section~\ref{sec:comb_sol_II_1} below, we define this setting and
state our main result; it is then proved in
Section~\ref{sec:comb_sol_II_2}; in Section~\ref{sec:singularity}, we
deduce the Aztec diamond applications, and use them to prove Devron
properties.

\subsection{Complementary trees and forests}\label{sec:comb_sol_II_1}

Consider a simple quadrangulation $\tilde{G}$ of the sphere. Since $\tilde{G}$ has faces of degree four (even), it is bipartite and its vertices can be colored in white and black. The set of faces is written as $F$, and the notation $f$ is used for a face of $F$ as well as for the corresponding dual vertex. Assume that faces are equipped with weights $(a_f)_{f\in F}$. For the sequel, let us emphasize the following easy fact: every directed edge $(w,b)$ has a unique face on the left and on the right.

Further suppose that the quadrangulation $\tilde{G}$ has two marked
adjacent vertices $w_r,b_r$, and denote its vertex set by
$(W\cup\{w_r\})\sqcup (\tilde{B}\cup \{b_r\})$.
Without loss of generality, assume that $|W|\leq |\tilde{B}|$, otherwise exchange the black and white colors, see Figure~\ref{fig:G_Gdual} (left).

Let $G^\bullet$, resp. $G^\circ$, be the graph consisting of the diagonals of the quadrangles of $\tilde{G}$ joining black, resp. white, vertices. Note that $G^\bullet$ and $G^\circ$ are dual graphs. We consider $G^\bullet$ and $G^\circ$ as embedded in the plane in such a way that $w_r$ corresponds to the outer face of $G^\bullet$ and $b_r$ is a vertex on the boundary of the outer face of $G^\circ$; the vertex $w_r$ is represented in a spreadout way, see Figure~\ref{fig:G_Gdual} (center).

Consider a subset $B\subset \tilde{B}$ of black vertices, and let $G$ be the graph obtained from $\tilde{G}$ by removing the black vertices $(\tilde{B}\cup\{b_r\})\setminus B$, the white vertex $w_r$ and their incident edges; the vertex set of $G$ is $W\sqcup B$, see Figure~\ref{fig:G_Gdual} (right).

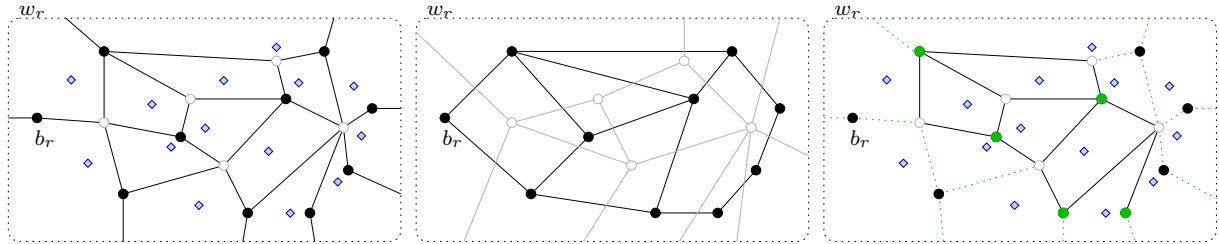
\begin{figure}[ht]
\begin{minipage}[b]{0.33\linewidth}
\begin{center}
\begin{overpic}[width=5.2cm]{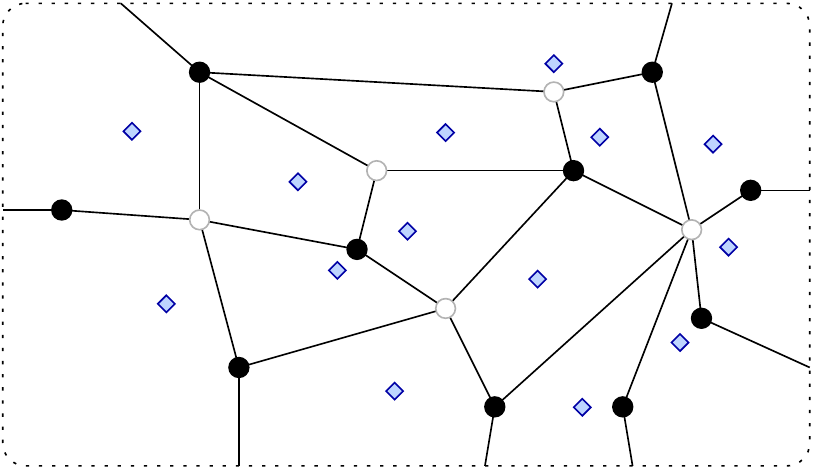}
\put(7,25){\scriptsize $b_r$}
\put(3,58){\scriptsize $w_r$}
\end{overpic}
\end{center}
\end{minipage}
\begin{minipage}[b]{0.33\linewidth}
\begin{center}
\begin{overpic}[width=5.2cm]{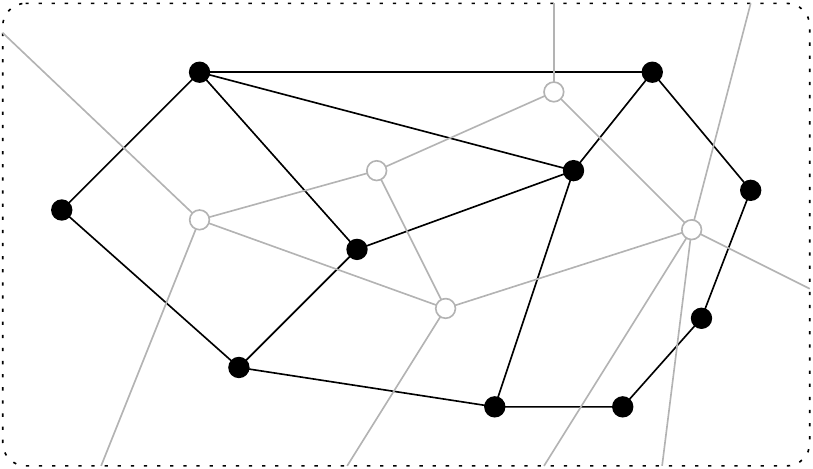}
\put(7,25){\scriptsize $b_r$}
\put(3,58){\scriptsize $w_r$}
\end{overpic}
\end{center}
\end{minipage}
\begin{minipage}[b]{0.33\linewidth}
\begin{center}
\begin{overpic}[width=5.2cm]{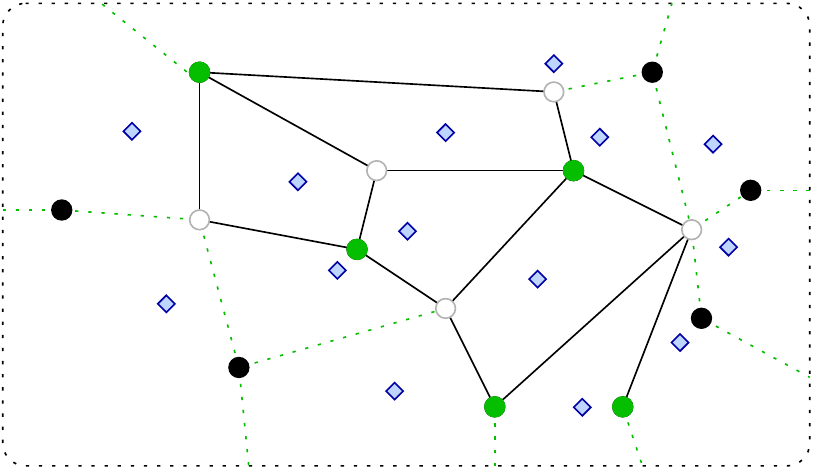}
\put(7,25){\scriptsize $b_r$}
\put(3,58){\scriptsize $w_r$}
\end{overpic}
\end{center}
\end{minipage}

\caption{Left: the quadrangulation $\tilde{G}$ with marked vertices $w_r,b_r$ ($w_r$ is represented in a spreadout way by a dotted line) and the corresponding faces (pictured with diamonds). Center: the graph $G^\bullet$ and its dual graph $G^\circ$. Right: a subset of black vertices $B\subset\tilde{B}$ (pictured as green bullets) and the associated graph $G$ (full edges), dotted edges are those of $\tilde{G}$ that have been removed.}
\label{fig:G_Gdual}
\end{figure}

The graph $G$ plays the role of the graph with open faces of Section~\ref{sec:crosses_and_wrenches}.
Let us recall the definition of the matrix $K=(K_{w,b})$: it is the weighted adjacency matrix of $G$, whose rows are indexed by vertices of $W$, columns by those of $B$, and whose non zero coefficients correspond to edges of $G$. Observing that edges of $G$ are also edges of $\tilde{G}$, non-zero entries are given by, for every edge $wb$ of $G$,
\begin{equation}
  \label{eq:defKtrees}
  K_{w,b}=a_{f(w,b)}-a_{f(b,w)},
\end{equation}
where $f(w,b)$, resp.~$f(b,w)$ denotes the face on the right of $(w,b)$, resp.~$(b,w)$ in $\tilde{G}$.

For the remainder of this section, suppose that $|B|=|W|$ to ensure that the matrix $K$ is square. Our main result is a combinatorial interpretation of $\det K$ establishing a bijection between combinatorial objects and monomials of $\det K$ in the $a$ variables. In order to state it, we need the following definitions.

Given a subset of vertices $\{b_1,\dots,b_\ell\}$ of
$G^\bullet$, referred to as \emph{root vertices} or \emph{root
  set}, a \emph{directed spanning forest rooted at $\{b_1,\dots,b_\ell\}$} is a collection of $\ell$ connected
components, all of which are subsets of directed edges, such that the $i$-th component: contains the vertex $b_i$, is a tree, \emph{i.e.}, has no cycle, and has edges oriented towards the root vertex $b_{i}$; moreover, the union of the $\ell$ components covers all vertices of $G^\bullet$. Note that the $i$-th component is allowed to be reduced to the point $b_{i}$.
If the root set is reduced to a single vertex $\{b_1\}$, we speak of a \emph{directed spanning tree rooted at the vertex $b_1$}. From now on, we will omit the term ``directed'' in the definitions. We will also use the fact that a spanning forest of $G^\bullet$ rooted on $\ell$ vertices has $|\tilde{B}|+1-\ell$ edges.

Let $\F$ be the set of pairs $(\Ts,\Fs)$ of edge configurations of $G^\bullet$ such that:
\begin{itemize}
 \item[$\bullet$] $\Ts$ is a spanning tree of $G^\bullet$ rooted at $b_r$, $\Fs$ is a spanning forest of $G^\bullet$
 rooted at $(\tilde{B}\cup\{b_r\})\setminus B$,
 \item[$\bullet$] the edge intersection of $\Ts$ and $\Fs$ is empty,
\end{itemize}
We refer to $\F$ as the set of \emph{complementary tree/forest configurations of $G^\bullet$ (rooted at $b_r$,\\
$(\tilde{B}\cup\{b_r\})\setminus B$)}, omitting the bracketed part whenever no confusion occurs, see Figure~\ref{fig:Compl_trees} for an example.

\begin{figure}[tb]
\begin{center}
\begin{overpic}[width=6cm]{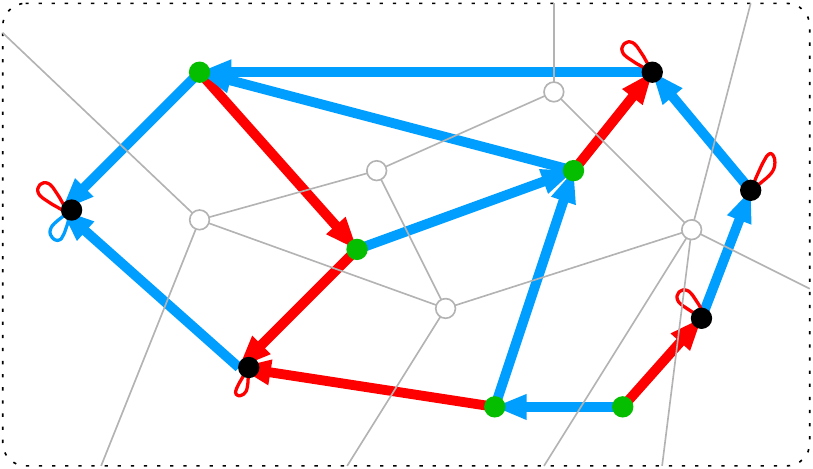}
\put(7,25){\scriptsize $b_r$}
\put(3,58){\scriptsize $w_r$}
\end{overpic}
\end{center}
\caption{A pair $(\Ts,\Fs)$ of complementary tree (blue) /forest (red) of $G$ rooted at $b_r$,
$(\tilde{B}\cup\{b_r\})\setminus B$; the set $B$ consists of the full green bullets. Roots are pictured with a loop.}
\label{fig:Compl_trees}
\end{figure}

\begin{remark}
All edges of the graph $G^\bullet$ are covered by the superimposition $\Ts\cup\Fs$ of both configurations. Indeed, since $\Ts,\Fs$ are disjoint, we have
\[
|\Ts\cup \Fs|=|\Ts|+|\Fs|=(|\tilde{B}|+1)-1+(|\tilde{B}|+1)-(|\tilde{B}|+1-|B|)=|\tilde{B}|+|W|,
\]
where in the last equality we used that $|B|=|W|$. Now by Euler's formula we know that this is equal to the number of edges of $G^\bullet$ (or $G^\circ$).
\end{remark}

Note that every (directed) edge $\vec{e}$ of $G^\bullet$ crosses a unique face $f$ of $F$, which we denote by $f_{\vec{e}}$. We are now ready to state the main result of this section, see Equation~\eqref{equ:def_sign} for a precise definition of $\sign(\Ts,\Fs)$, a function taking values in $\{-1,1\}$.

\begin{theorem}\label{thm:comb_int_II}
For any Kasteleyn orientation $\vphi$, the oriented dimer partition function is the following sum:
\begin{equation*}
Z(G,a,\vphi)=\pm
\sum_{(\Ts,\Fs)\in\F} \sign(\Ts,\Fs) \prod_{\vec{e}\in \Fs} a_{f_{\vec{e}}}.
\end{equation*}
Moreover, there is a bijection between terms in the sum on the right-hand-side and monomials of $Z(G,a,\vphi)$ in the variables $a$.
\end{theorem}
The proof of this theorem is a consequence of intermediate results that are interesting in their own respect. This is the subject of the next section.

\subsection{Proof of Theorem~\ref{thm:comb_int_II}}\label{sec:comb_sol_II_2}

Recall that by Proposition~\ref{prop:oriented_dimers_adjacency_matrix}, up to a sign, the oriented dimer partition function $Z(G,a,\vphi)$ is equal to $\det(K)$, where $\vphi$ is any choice of Kasteleyn orientation.

Theorem~\ref{thm:comb_int_II} is a consequence of Proposition~\ref{prop:matrix_relation_gen} proving a matrix relation, its immediate Corollary~\ref{cor:matrix_relation_gen} and a combinatorial argument. As prerequisites, we need the generalized form of Temperley's bijection~\cite{Temperley} due to Kenyon, Propp and Wilson;~\cite{KPW}, and a few notation used in the statement of Proposition~\ref{prop:matrix_relation_gen}.

\paragraph{Extended Temperley's bijection~\cite{KPW}.} Given a spanning tree of $G^\bullet$ rooted at $b_r$, the dual edge configuration, consisting of the dual edges of the edges absent in the tree is a spanning tree of $G^\circ$; let us orient it towards the root vertex $w_r$. This pair is referred to as a \emph{pair of dual spanning trees of $G^\bullet,G^\circ$ (rooted at $b_r,w_r$)}. Note the difference between complementary trees/forests that live on the same graph $G^\bullet$, and pairs of dual spanning trees that live on $G^\bullet$, $G^\circ$. Note also that, given a spanning tree of $G^\bullet$, its complementary configuration might not be a spanning forest, whereas its dual configuration will always be a tree.

The \emph{double graph}, denoted by $\GD$ is the graph consisting of the diagonals of the quadrangles of $\tilde{G}$ with additional vertices at the crossings of the diagonals, see Figure~\ref{fig:Temperley} (left: full and dotted edges). Vertices of $\GD$ are of three types: black vertices $B\cup\{b_r\}$
 of $G^\bullet$, white vertices $W\cup\{w_r\}$ of $G^\circ$ and additional vertices corresponding to faces of $\tilde{G}$ labeled as $F$. The graph $\GD$ is bipartite with vertices split as
 $(W\cup \tilde{B}\cup \{w_r,b_r\})\sqcup F$. Let $\GD_r$ be the graph obtained from $\GD$ by removing the vertices $w_r,b_r$ and their incident edges, see Figure~\ref{fig:Temperley} (left: full edges).

Recall that by Euler's formula, we have $|W|+|\tilde{B}|=|F|$. By~\cite{KPW} perfect matchings of $\GD_r$ are in bijection with pairs of dual spanning trees of $G^\bullet$, $G^\circ$ rooted at $b_r,w_r$.
Given a perfect matching of $\GD_r$, the pair of dual spanning trees is obtained by adding the half-edge in the prolongation of each dimer edge, thus giving an edge of $G^\bullet$ or $G^\circ$, and orienting it in the direction of the prolongation. This procedure is naturally reversible, see Figure~\ref{fig:Temperley}. We refer to these constructions as the \emph{Temperley} and \emph{reverse Temperley tricks}.

\begin{figure}[ht]
\begin{minipage}[b]{0.49\linewidth}
\begin{center}
\begin{overpic}[width=6cm]{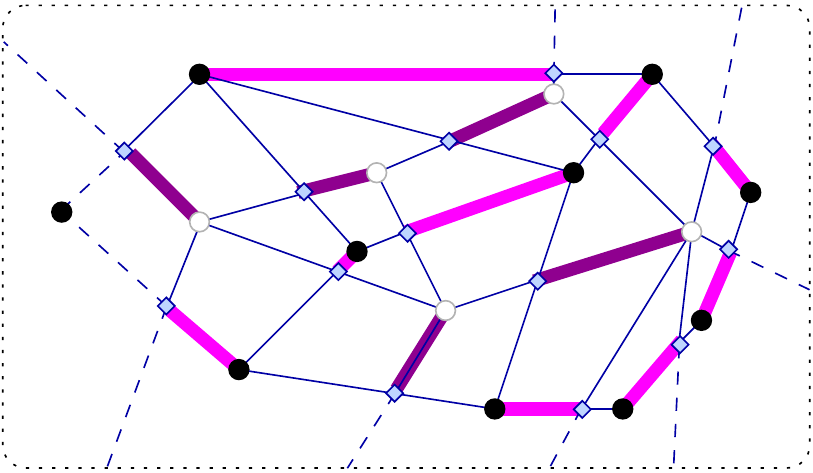}
\put(7,25){\scriptsize $b_r$}
\put(3,58){\scriptsize $w_r$}
\end{overpic}
\end{center}
\end{minipage}
\begin{minipage}[b]{0.49\linewidth}
\begin{center}
\begin{overpic}[width=6cm]{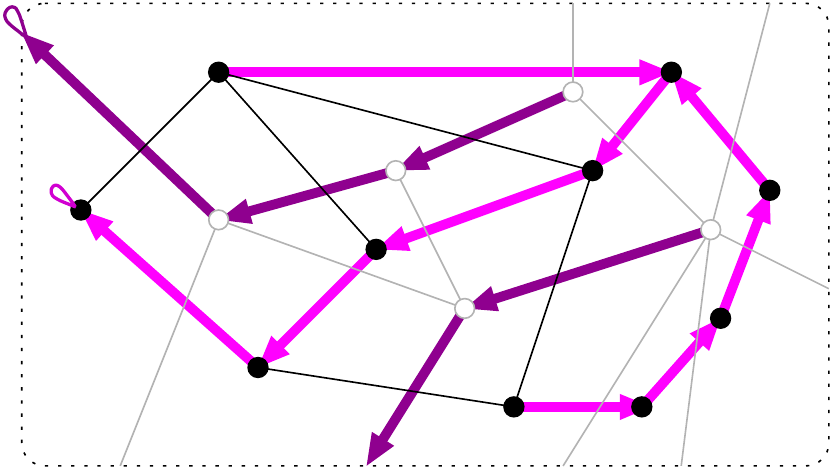}
\put(7,25){\scriptsize $b_r$}
\put(3,58){\scriptsize $w_r$}
\end{overpic}
\end{center}
\end{minipage}
\caption{Temperley's bijection~\cite{KPW}. Left: a perfect matching of
  $\GD_r$; its restriction to the light magenta edges corresponds to
  the non-zero coefficients of the matrix $M$. Right: the
  corresponding pair of dual spanning trees of $G^\bullet,G^\circ$
  rooted at $b_r,w_r$.
}
\label{fig:Temperley}
\end{figure}

\paragraph{Definition of the matrix $M$.}
Fix a pair of \emph{reference} dual spanning trees rooted at $b_r,w_r$, and the corresponding perfect matching of $\GD_r$. Let $M$ be the matrix whose rows are indexed by $\tilde{B}$, columns by $F$, whose non-zero coefficients are equal to 1 and correspond to the restriction of the dimer configuration to vertices of $\tilde{B}$, see Figure~\ref{fig:Temperley} (left).

\paragraph{Definition of the matrix $C(a)$.} Recall that $B$ is a subset of vertices of $\tilde{B}$. For the moment, we do not assume that $|B|=|W|$.
Let $C(a)$ be the weighted adjacency matrix of $\GD$, with rows indexed by vertices of $F$, columns by vertices of $(W\cup \tilde{B}\cup\{w_r,b_r\})$, and whose non-zero coefficients are given by, for every edge $fw$, resp. $fb$, of $\GD$,
\begin{equation}\label{equ:choice_sign}
C(a)_{f,w}=\pm a_f, \, C(a)_{f,b}=\pm a_f;
\end{equation}
the signs of coefficients are chosen so that, when going counterclockwise around each vertex $f$, two consecutive edges $fb$, $fw$ have the same sign and two consecutive edges $fw,fb$ have opposite signs, see Figure~\ref{fig:C_matrix}.

\begin{figure}[tb]
\begin{center}
\begin{overpic}[width=9cm]{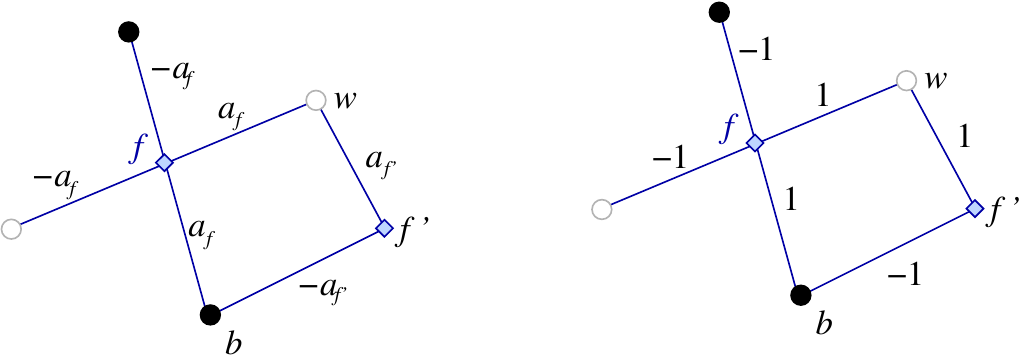}
\end{overpic}
\end{center}
\caption{Weights of the matrix $C(a)$ (left), $C(1)$ (right).}
\label{fig:C_matrix}
\end{figure}

Let $C(1)$ be the matrix $C(a)$ in the case where all faces weights $(a_f)_{f\in F}$ are equal to 1.
Let $C(1)^t_{W}$ be the transpose of $C(1)$ with rows restricted to $W$. In a similar way,
$C(1)^{\tilde{B}}$ is the matrix $C(1)$ with columns restricted to $\tilde{B}$, and $C(a)^{B}$ is the matrix $C(a)$ with columns restricted to $B$.

\begin{proposition}\label{prop:matrix_relation_gen}
The following matrix relation holds,
\begin{equation*}
\begin{blockarray}{cc}
     &\scriptstyle{F} \\
\begin{block}{c(c)}
  \scriptstyle{W}\! & C(1)^t_{W}\\
  \scriptstyle{\tilde{B}}\! & M\\
\end{block}
\end{blockarray}
\begin{blockarray}{ccc}
      &\scriptstyle{\tilde{B}}&\scriptstyle{B}\\
 \begin{block}{c(cc)}
   \scriptstyle{F} & C(1)^{\tilde{B}}& C(a)^{B}\\
 \end{block}
\end{blockarray} =
\begin{blockarray}{ccc}
      &\scriptstyle{\tilde{B}}&\scriptstyle{B}\\
 \begin{block}{c(cc)}
   \scriptstyle{W} & 0 & K\\
   \scriptstyle{\tilde{B}} & \bigstar & \lozenge\\
 \end{block}
\end{blockarray}.
\end{equation*}
Moreover,
\[
\det\begin{pmatrix}
    C(1)^t_{W}\\
    M
    \end{pmatrix}
=\pm 1, \quad \det(\bigstar)=\pm 1.
\]
\end{proposition}
Before turning to the proof of this proposition, let us state an immediate corollary.

\begin{corollary}\label{cor:matrix_relation_gen}
If the subset $B$ of black vertices of $\tilde{B}$ is such that $|B|=|W|$, then
\begin{equation*}
\det K=\pm \det
\begin{pmatrix}
C(1)^{\tilde{B}}& C(a)^{B}
\end{pmatrix}.
\end{equation*}
\end{corollary}

\begin{proof}[Proof of Prosition~\ref{prop:matrix_relation_gen}]
Let us show the identity for the first block row. We consider $w$ a white vertex of $W$, and $b$ a black vertex of $\tilde{B}$, resp. $B$. If $w$ and $b$ are not adjacent in $\tilde{G}$, then
\[
(C(1)^t_W\, C(1)^{\tilde{B}})_{w,b}=0, \text{ resp. } (C(1)^t_W\, C(a)^{B})_{w,b}=0.
\]
Else, if $w$ and $b$ are adjacent in $\tilde{G}$, there are exactly two vertices $f$, $f'$ of $\GD$ that are adjacent to both $w$ and $b$ in $\GD$; let us say that $f$ is on the right of the directed edge $(w,b)$ and $f'$ on the left, see Figure~\ref{fig:C_matrix}.
When $b\in \tilde{B}$ (first block column), the matrix product is
\begin{equation*}
C(1)_{w,f} C(1)_{f,b}+C(1)_{w,f'} C(1)_{f',b}=+1-1=0,
\end{equation*}
using our convention for the choice of sign. In a similar way, when $b\in B$ (second block column), we have
\begin{equation*}
C(1)_{w,f} C(a)_{f,b}+C(1)_{w,f'} C(a)_{f',b}=a_f-a_{f'}=K_{w,b}.
\end{equation*}
Since we are interested in determinants, we do not need to care about the matrix $\lozenge$. We now describe the matrix $\bigstar$ and prove that its determinant is equal to $\pm 1$. We will compute it ``graphically'' by considering the matrix
as a weighted adjacency matrix of a graph. Recall that $M$ is the restriction to $\tilde{B}$ of the perfect matching corresponding to a pair of fixed dual spanning trees of $G^\bullet$, $G^\circ$. As a consequence, when computing
$\bigstar=M\cdot C(1)^{\tilde{B}}$ we have, for every vertices $b,b'$ in $\tilde{B}$,
\[
\bigstar_{b,b'}=
\begin{cases}
\pm 1 & \text{ if $b=b'$}\\
\pm 1 & \text{ if $(b,b')$ is a directed edge of the spanning tree of $G^\bullet$}\\
0 & \text{otherwise}
\end{cases}.
\]
Writing $\det(\bigstar)$ as a sum over permutations, which decompose as cycles, and noting that apart from the diagonal terms, we have no cycle in the graph corresponding to this matrix, we deduce that the only non-zero contribution to the determinant comes from the identity permutation; it is equal to the product of the diagonal terms, that is $\pm 1$.

We are left with proving that
\[
\det \begin{pmatrix}
    C(1)^t_{W}\\
    M
    \end{pmatrix}=\pm 1.
\]
Again we expand this determinant
as a sum over permutations and compute it graphically. Since the graph corresponding to this matrix is bipartite, non-zero terms in the expansion correspond to perfect matchings. Now, this graph is a subgraph of $\GD_r$, and recall that perfect matchings of $\GD_r$ are in one-to-one correspondence with pairs of dual spanning trees of $G^\bullet,G^\circ$ by Temperley's bijection~\cite{KPW}. But the submatrix $M$ is the restriction to $\tilde{B}$ of the perfect matching corresponding to a fixed reference pair of dual spanning trees of $G^\bullet,G^\circ$. This implies that the primal tree is fixed, and hence the dual too. As a consequence, there is only one non-zero term, corresponding to the perfect matching arising from the fixed pair of dual spanning trees of $G^\bullet,G^\circ$. Given that coefficients are all equal to $\pm 1$, this contribution is $\pm 1$.
\end{proof}

We now restrict to the case where $|B|=|W|$. Our goal is to prove Theorem~\ref{thm:comb_int_II} establishing a combinatorial interpretation of $\det K$, but we first need to precisely define
$\sign(\Ts,\Fs)$.

\paragraph{Definition of $\sign(\Ts,\Fs)$.} Denote by $\Md$ the set of pairs of (non perfect) matchings $(\Ms_1,\Ms_2)$ of $\GD_r$ such that: $\Ms_1$ joins every black vertex of $\tilde{B}$ to a vertex of $F$, $\Ms_2$ joins every black vertex of $B$ to a vertex of $F$, and the superimposition $\Ms_1\cup\Ms_2$ is such that every vertex of $F$ is incident to exactly one edge of $\Ms_1$ or $\Ms_2$, see Figure~\ref{fig:Compl_trees_1} (right) for an example.

Suppose that $|F|=\ell$, $|\tilde{B}|=m$ for some $1<m<\ell$. Label the vertices of $F$ as $\{f_1,\dots,f_\ell\}$, those of $\tilde{B}$ as $\{b_1,\dots,b_m\}$ and those of $B$ as $\{b_{m+1},\dots,b_\ell\}$, keeping in mind that $B\subset \tilde{B}$, so that vertices of $B$ receive two labels. Then, every pair $(\Ms_1,\Ms_2)$ of matchings of $\Md$ naturally yields a permutation $\sigma\in \S_\ell$ where, for every $1\leq j\leq \ell$, $b_j$ and $f_{\sigma(j)}$ is a matched edge of $\Ms_1$, resp. $\Ms_2$, if $1\leq j\leq m$, resp. $m+1\leq j\leq \ell$.

Consider a pair $(\Ts,\Fs)$ of complementary tree/forest of $G^\bullet$ rooted at $b_r, (\tilde{B}\cup\{b_r\})\setminus B$. Using the reverse Temperley trick on $(\Ts,\Fs)$ yields a pair of matchings $(\Ms_1,\Ms_2)$ of $\Md$, see Figure~\ref{fig:Compl_trees_1}.

\begin{figure}[tb]
\begin{minipage}[b]{0.49\linewidth}
\begin{center}
\begin{overpic}[width=6cm]{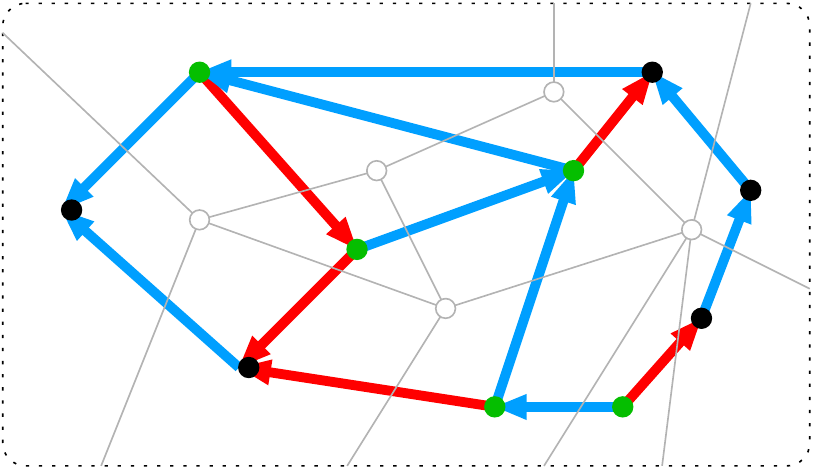}
\put(7,25){\scriptsize $b_r$}
\put(3,58){\scriptsize $w_r$}
\end{overpic}
\end{center}
\end{minipage}
\begin{minipage}[b]{0.49\linewidth}
\begin{center}
\begin{overpic}[width=6cm]{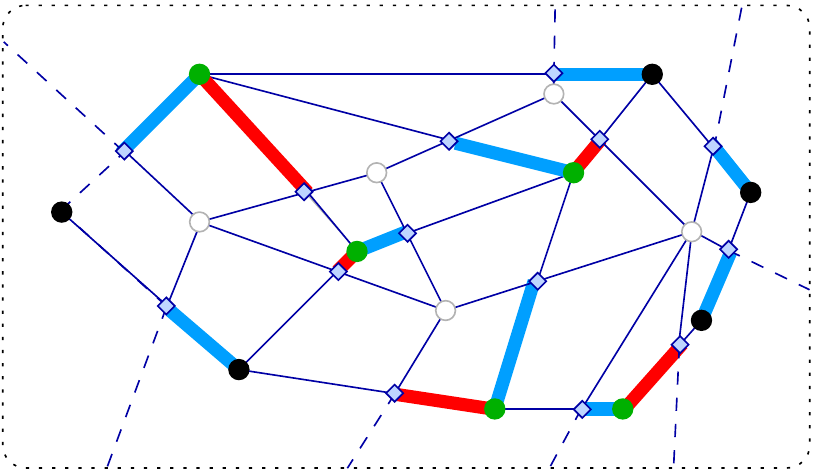}
\put(7,25){\scriptsize $b_r$}
\put(3,58){\scriptsize $w_r$}
\end{overpic}
\end{center}
\end{minipage}
\caption{Left: a pair $(\Ts,\Fs)$ of complementary tree (blue) / forest (red) of $G^\bullet$ rooted at $b_r$, $(\tilde{B}\cup\{b_r\})\setminus B$; the set $B$ consists of the green vertices; roots are not pictured in any specific way as was the case in Figure~\ref{fig:Compl_trees}. Right: corresponding pair of matchings
$(\Ms_1,\Ms_2)$ of $\Md$ obtained by Temperley's reverse trick.}
\label{fig:Compl_trees_1}
\end{figure}

Let $\sigma$ be the associated permutation of $\S_\ell$ as above.
Define the \emph{sign of $(\Ts,\Fs)$}, denoted $\sign(\Ts,\Fs)$, to be
\begin{equation}\label{equ:def_sign}
\sign(\Ts,\Fs)=\sgn(\sigma)C(1)_{f_{\sigma(1)},b_1}\dots C(1)_{f_{\sigma(\ell)},b_\ell}.
\end{equation}

We are now ready to prove Theorem~\ref{thm:comb_int_II}.

\paragraph{Proof of Theorem~\ref{thm:comb_int_II}.}
We use Corollary~\ref{cor:matrix_relation_gen} and graphically compute the determinant
\[
\det
\begin{pmatrix}
C(1)^{\tilde{B}}& C(a)^{B}
\end{pmatrix}.
\]
Non-zero terms in the permutation expansion of the determinant correspond to pairs of matchings of $\Md$. Consider such a pair $(\Ms_1,\Ms_2)$ and do the Temperley trick for $\Ms_1$. This gives a directed edge configurations $\bar{\Ms}_1$ of $G^\bullet$ such that every black vertex of $\tilde{B}$ has an outgoing edge, known as a \emph{directed cycle rooted spanning forest (CRSF) rooted at $b_r$}.
It consists of connected components covering all vertices of $G^\bullet$ each of which is: either a directed tree rooted at $b_r$ or a directed tree rooted on a simple cycle not containing $b_r$, where the cycle is oriented in one of the two possible directions. Note that there is exactly one directed tree component rooted at $b_r$ (which may consist of the vertex $b_r$ only). In a similar way, to the configuration $\Ms_2$ corresponds a directed cycle rooted spanning forest rooted at
$(\tilde{B}\cup\{b_r\})\setminus B$, consisting of connected components covering all vertices of $G^\bullet$ each of which is: either a directed tree rooted at a vertex of $(\tilde{B}\cup\{b_r\})\setminus B$, or a directed tree rooted on a simple cycle not containing any of the vertices of $(\tilde{B}\cup\{b_r\})\setminus B$. There is one directed tree component for each vertex of $(\tilde{B}\cup\{b_r\})\setminus B$ but it can be reduced to a single vertex. Note that since $G^\bullet$ is assumed to be simple, all cycles are of length greater or equal to 3.

Fix a pair of matchings $(\Ms_1,\Ms_2)$ of $\Md$, and suppose that $\bar{\Ms}_1$ contains a cycle of length $n\geq 3$. Then, consider the CRSF obtained from $\bar{\Ms}_1$ by reversing the orientation of the cycle. By using the reverse Temperley trick on $\bar{\Ms}_1$,
this yields a pair of matchings $(\Ms_1',\Ms_2)$ which also contributes to the determinant. Let us look at the quotient of the contributions of $(\Ms_1,\Ms_2)$ and $(\Ms_1',\Ms_2)$ to the determinant. The associated permutations differ by a cycle of length $n$, giving a factor $(-1)^{n+1}$. The only edge-weights contributing to the quotient arise from the matched edges associated to the cycle. Now, by our choice of signs for the matrix $C(1)$, the pair of half-edges of $\GD$ corresponding to each edge of $G^\bullet$ have opposite signs implying that the contribution of the edge-weights to the quotient is $(-1)^n$. As a consequence, the quotient of the contributions is equal to $(-1)^{2n+1}=(-1)$, and we deduce that the terms corresponding to $(\Ms_1,\Ms_2)$ and $(\Ms_1',\Ms_2)$ cancel out. This argument holds as soon as $\bar{\Ms}_1$ has a cycle, so that there only remains configurations where $\bar{\Ms}_1$ is a CRSF rooted at $b_r$ with no cycle, \emph{i.e.}, a spanning tree rooted at $b_r$.

Since the sign convention for $C(a)$ is the same as that of $C(1)$, a similar argument can be done for the matching $\Ms_2$. We deduce that the only configurations remaining are such that $\bar{\Ms}_2$ is a CRSF rooted at $(\tilde{B}\cup\{b_r\})\setminus B$ containing no cycle, \emph{i.e.}, a spanning forest rooted at $(\tilde{B}\cup\{b_r\})\setminus B$.
Since every vertex of $F$ is incident to exactly one edge of $\Ms_1$ or $\Ms_2$, we know that the edge intersection of the corresponding directed spanning trees/forests is empty.

Summarizing, applying Temperley's trick to pairs of matchings $(\Ms_1,\Ms_2)$ of $\Md$ that contribute to the determinant, we obtain pairs of complementary trees/forests of $G$ rooted at $b_r$, $(\tilde{B}\cup\{b_r\})\setminus B$. To compute the contribution of such a configuration, we also use that $C(a)_{f_{\sigma(j)},b_j}=C(1)_{f_{\sigma(j)},b_j}a_{f_{\sigma(j)}}$. Using the reverse Temperley trick yields the converse thus ending the proof of the combinatorial formula.

To establish the bijection between terms in the sum on the right-hand-side and monomials in the variables $a$ it suffices to notice that if we have two distinct pairs $(\Ts,\Fs)$, $(\Ts',\Fs')$ of complementary trees/forests of $G^\bullet$, then $\Fs\neq \Fs'$ implying that there is at least on edge $e$ such that $\vec{e}$ or $\cev{e}$ is present in $\Fs$ and not in $\Fs'$, giving a contribution $a_{f_{\vec{e}}}$ or $a_{f_{\cev{e}}}$ (both are equal) to one and not to the other.
\qed

\begin{remark}
  \label{rk:quadrangulate}
  In this Section, we started with a quadrangulation $\tilde{G}$, and
  considered a subgraph $G$ on which we defined the
  matrix~\eqref{eq:defKtrees}. However, we can switch perspective and think
  that we start with a graph $G$ whose internal faces have degree $4$,
  equipped with a usual dimer model. Then, finding a family of complex
  numbers $a$ such that the Kasteleyn matrix is (gauge equivalent to)
  \eqref{eq:defKtrees} is the point of the construction of Coulomb
  gauges, or of t-embeddings \cite{klrr,clrtembeddings}. Therefore,
  Theorem~\ref{thm:vert_translation} may be seen as a way to
  combinatorially expand the partition function of usual dimers in
  terms of the $a$ variables, taken as formal variables.

  This seems to be limited to dimer graphs $G$ with internal faces of
  degree $4$, but in fact, if $G$ is only bipartite, one can always
  quadrangulate its internal faces, and give the same value $a_f$ to
  all quadrangles coming from an initial face $f$ of $G$. In this way,
  the matrix $K$ defined by \eqref{eq:defKtrees} get entries $0$ on
  newly added edges, so it is not affected. Therefore, we can also
  write the partition function of usual dimers on $G$ as a sum over
  complementary trees and forests on diagonals of the quadrangulation
  of $G$. However, since we set several faces to the same weight
  $a_f$, it is no longer the case that configurations are in
  one-to-one correspondence with monomials.
\end{remark}

\section{Aztec diamond case and Devron property}\label{sec:singularity}

In all of this section we consider an Aztec diamond of size $k\geq 1$,
denoted $A_k$. We will picture $A_k$ turned by $45^\circ$ with respect to its introduction in
Figure~\ref{fig:ex_Aztec} for instance, and we change the labelling of
variables $a$ accordingly, see Figure~\ref{fig:Aztec_0}. The previous
representation naturally came from the method of crosses and wrenches. Here we need simple indexing of diagonals, which is much easier to do when considering them as columns of the $45^\circ$-rotated Aztec diamond.

\begin{figure}[tb]
\begin{minipage}[b]{0.495\linewidth}
\begin{center}
\begin{overpic}[width=4.8cm]{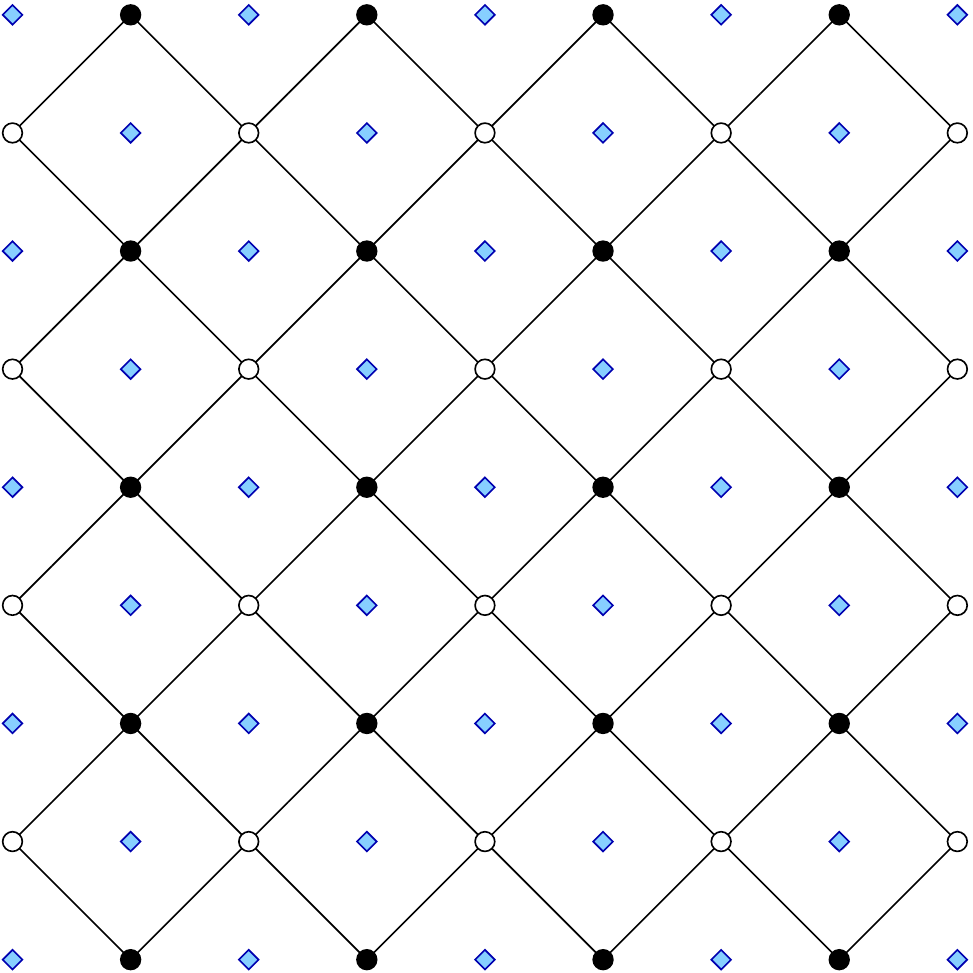}
\put(-5,-4){\scriptsize $a_{0,0}$}
\put(20,-4){\scriptsize $a_{2,0}$}
\put(45,-4){\scriptsize $a_{4,0}$}
\put(70,-4){\scriptsize $a_{6,0}$}
\put(95,-4){\scriptsize $a_{8,0}$}
\put(-5,21){\scriptsize $a_{0,2}$}
\put(-5,46){\scriptsize $a_{0,4}$}
\put(-5,69){\scriptsize $a_{0,6}$}
\put(-5,93){\scriptsize $a_{0,8}$}
\put(9,8){\scriptsize $a_{1,1}$}
\put(33,8){\scriptsize $a_{3,1}$}
\put(55,8){\scriptsize $a_{5,1}$}
\put(80,8){\scriptsize $a_{7,1}$}
\put(10,33){\scriptsize $a_{1,3}$}
\put(10,58){\scriptsize $a_{1,5}$}
\put(10,81){\scriptsize $a_{1,7}$}
\end{overpic}
\end{center}
\end{minipage}
\begin{minipage}[b]{0.495\linewidth}
\begin{center}
\begin{overpic}[width=4.8cm]{fig_Aztec_0.pdf}
\put(-5,-4){\scriptsize $c_{0,0}$}
\put(20,-4){\scriptsize $c_{1,0}$}
\put(45,-4){\scriptsize $c_{2,0}$}
\put(70,-4){\scriptsize $c_{3,0}$}
\put(95,-4){\scriptsize $c_{4,0}$}
\put(-5,21){\scriptsize $c_{0,1}$}
\put(-5,46){\scriptsize $c_{0,2}$}
\put(-5,69){\scriptsize $c_{0,3}$}
\put(-5,93){\scriptsize $c_{0,4}$}
\put(9,8){\scriptsize $d_{0,0}$}
\put(33,7){\scriptsize $d_{1,0}$}
\put(55,8){\scriptsize $d_{2,0}$}
\put(80,8){\scriptsize $d_{3,1}$}
\put(9,32){\scriptsize $d_{0,1}$}
\put(9,57){\scriptsize $d_{0,2}$}
\put(9,80){\scriptsize $d_{0,3}$}
\end{overpic}
\end{center}
\end{minipage}
\caption{Left: Aztec diamond $A_k$, $k=4$, with coordinates for face weights rotated by $45^\circ$. Right: notation $(c_{i,j})_{0\leq i,j\leq 4},\, (d_{i,j})_{0\leq i \leq 3}$ for faces weights}
\label{fig:Aztec_0}
\end{figure}
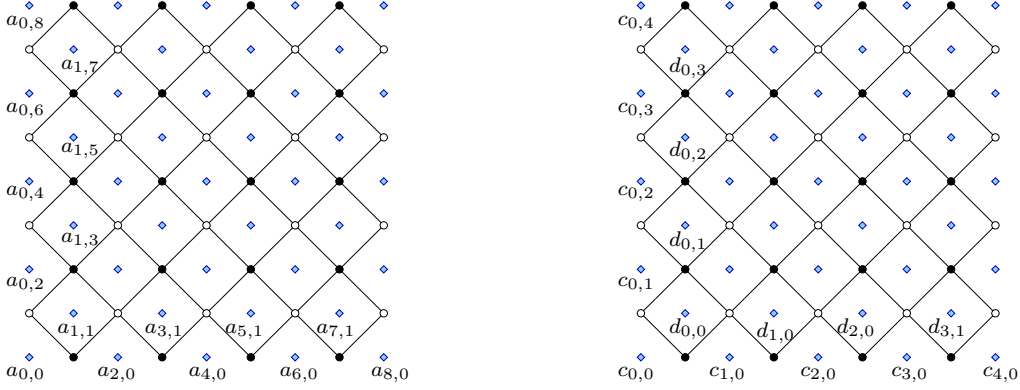

Face weights are now $\left(a_{i,j}\right)$ where $0\leq i,j \leq
2k$ and $[i+j]_2=0$. Another way to see this is to consider
two sets of weights, $\left( c_{i,j} \right)_{0 \leq i,j \leq k}$, and
$\left( d_{i,j} \right)_{0 \leq i,j \leq k-1}$, on even and odd faces,
that is
\begin{equation*}
  \begin{split}
    \forall\, i,j\in\{0,\dots,k\},& \quad a_{2i,2j}=:c_{i,j}, \\
    \forall\, i,j\in\{0,\dots,k-1\},&\quad a_{2i+1,2j+1}=:d_{i,j},
  \end{split}
\end{equation*}
see also Figure~\ref{fig:Aztec_0}.
We are interested in special cases of weights
motivated by their occurrence in
geometric systems, which are studied in the companion paper~\cite{paper2},
where they lead to
new incidence theorems and Devron properties.

The first goal is to specialize
Theorem~\ref{thm:comb_int_II} and
Corollary~\ref{cor:matrix_relation_gen} to the case of the Aztec diamond with no additional assumption on the weights;
we do this in Section~\ref{sec:Appl_Aztec_diamond}, and also
prove matrix identities for the ratio function of oriented dimers $Y(A_k,a)$ of
Definition~\ref{eq:defY}.  Next in Section~\ref{sec:cst_col} we
consider the case where columns of $d$ are constant, \emph{
  i.e.}, $d_{i,j}$ is independent of $j$, half of the column
weights are constant, and prove Theorem~\ref{thm:col_degen_I} which is a
combinatorial identity for the partition function of oriented dimers
involving simpler objects referred to as \emph{permutation spanning
  forests}.
In Section~\ref{sec:Dodgson_condens} we
  specialize further to all variables $d_{i,j}$ being equal, and prove
  Corollary~\ref{cor:dodgson} which is similar to classical Dodgson
  condensation \cite{dodgson}; this
  shows that $Z(A_k,a,\varphi)$ and $Y(A_k,a)$ have way more symmetries in the $\left(c_{i,j}\right)$ variables
  than one would expect. Finally in Section~\ref{sec:periodic_columns}
  we suppose that for some $p\geq 1$, every $p$-th column of $\left( d_{i,j} \right)$ is
  set to a constant, and prove another invariance result for
  $Y(A_k,a)$ in Theorem~\ref{thm:single_col}.

\subsection{Aztec diamond case}\label{sec:Appl_Aztec_diamond}

We use the notation and constructions of
Section~\ref{sec:comb_sol_II_1}. Since $k$ is fixed, we remove the
dependence in $k$ in the following notation except from $A_k$.

Let $W\sqcup B$ be the set of black and white vertices of $A_k$. Consider two additional black vertices $b_r,\tilde{b}$ such that $b_r$, resp. $\tilde{b}$, is on the left, resp. right, and all white vertices of $A_k$ on the left, resp. right, are connected to $b_r$, resp. $\tilde{b}$; denote by $\tilde{B}:=B\cup\{\tilde{b}\}$. Let $w_r$ be an additional white vertex connected to $b_r,\tilde{b}$, and to all black vertices of $A_k$ on the top and bottom rows. This defines a quadrangulation of the sphere $\tilde{G}$ with vertex set $(W\cup\{w_r\})\sqcup (\tilde{B}\cup\{b_r\})$, with two adjacent marked vertices $w_r,b_r$ as in Section~\ref{sec:comb_sol_II_1}, see Figure~\ref{fig:Aztec_1} (left). As before the set of faces is denoted by $F$, the notation $f$ is used for a face of $F$ and for the corresponding dual vertex, and faces are equipped with weights $(a_f)_{f\in F}$,
see Figure~\ref{fig:Aztec_1} (left); weights are also alternatively labeled by $\left(c_{i,j}\right)_{0\leq i,j \leq k}, \left(d_{i,j}\right)_{0\leq i,j \leq k-1} $ as in Figure~\ref{fig:Aztec_0}.

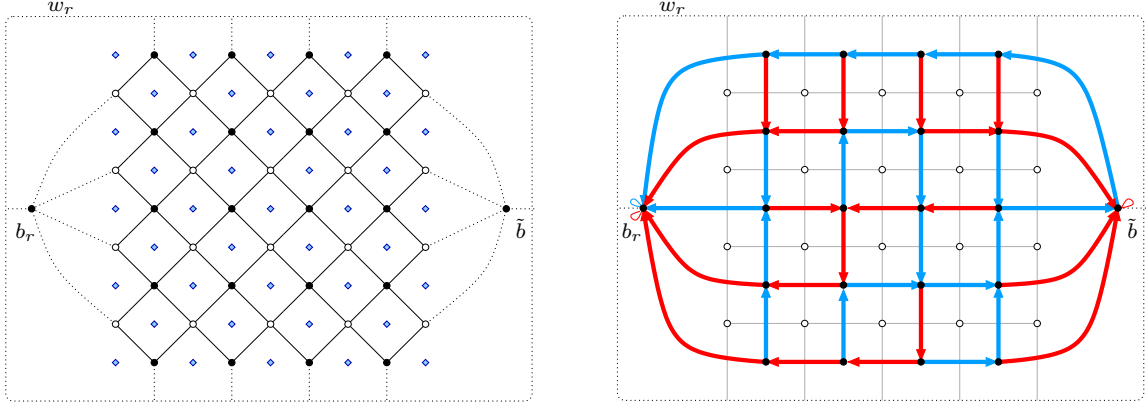
\begin{figure}[tb]
\begin{minipage}[b]{0.495\linewidth}
\begin{center}
\begin{overpic}[width=7cm]{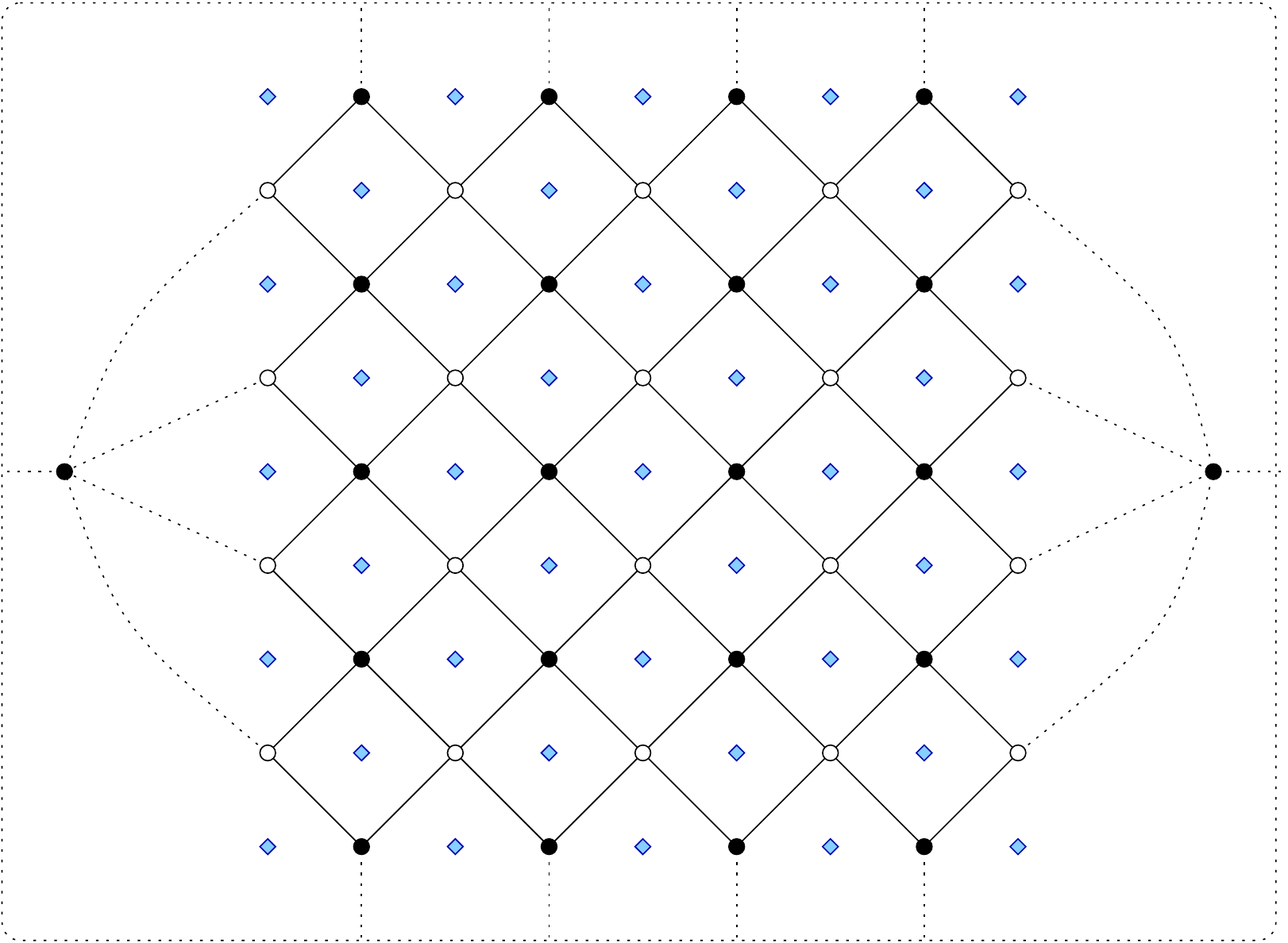}
\put(2,31){\scriptsize $b_r$}
\put(8,74){\scriptsize $w_r$}
\put(96,31){\scriptsize $\tilde{b}$}
\end{overpic}
\end{center}
\end{minipage}
\begin{minipage}[b]{0.495\linewidth}
\begin{center}
\begin{overpic}[width=7cm]{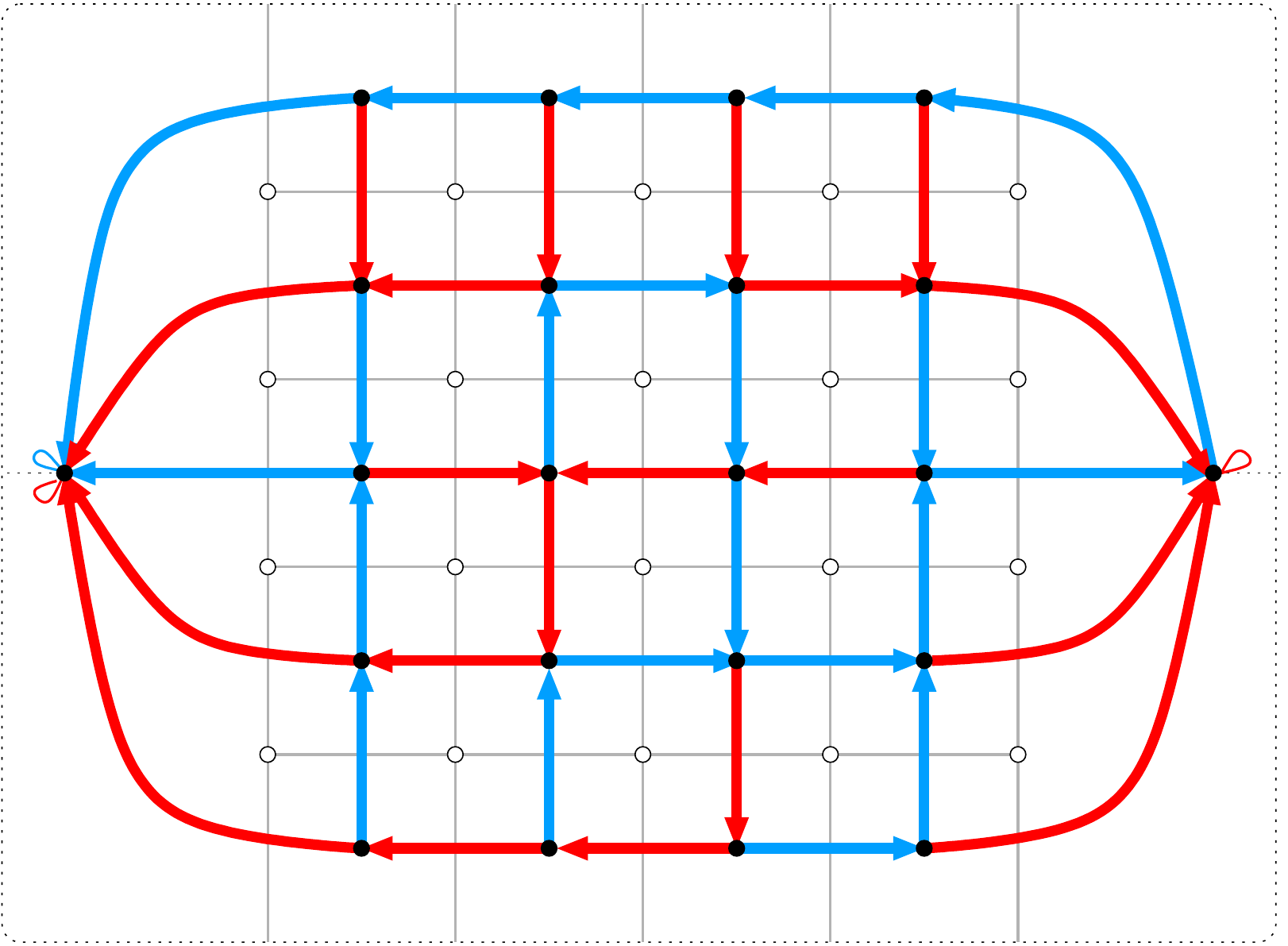}
\put(1,31){\scriptsize $b_r$}
\put(8,74){\scriptsize $w_r$}
\put(96,31){\scriptsize $\tilde{b}$}
\end{overpic}
\end{center}
\end{minipage}
\caption{Left: quadrangulation $\tilde{G}$ (plain and dotted lines)
  with marked vertices $w_r,b_r$, and the corresponding faces
  (pictured with diamonds); the Aztec diamond graph $G=A_k$ (plain
  lines), $k=4$.  Right: the graph $G^\bullet$ with a pair of
  complementary tree (blue) and forest (red) rooted at $b_r$,
  $\{\tilde{b},b_r\}$, and the dual graph $G^\circ$.
    }
\label{fig:Aztec_1}
\end{figure}

Trivially, we have that $B$ is a subset of $\tilde{B}$, and the graph $G$ obtained from $\tilde{G}$ by removing the vertices $\{\tilde{b},b_r,w_r\}$ and all of its incident edges is exactly the Aztec diamond $A_k$. Recall that $K$ denotes the weighted adjacency matrix of $A_k$ with non-zero coefficients given by, for every $b\in B$, $w\in W$ such that $w\sim b$, $K_{w,b}=a_{f(w,b)}-a_{f(b,w)}$.

The corresponding graphs $G^\bullet,G^\circ$ of Section~\ref{sec:comb_sol_II_1} are pictured in Figure~\ref{fig:Aztec_1} (right). The set $\F$ of complementary tree/forest configurations of $G^\bullet$ (rooted at $b_r$ and $\{\tilde{b},b_r\}$) is the set of pairs $(\Ts,\Fs)$ such that: $\Ts$ is a spanning tree of $G^\bullet$ rooted at $b_r$, and $\Fs$ is a spanning forest of $G^\bullet$ (with two components) rooted at $\{\tilde{b},b_r\}$, see Figure~\ref{fig:Aztec_1} for an example. As an immediate corollary to Proposition~\ref{prop:oriented_dimers_adjacency_matrix} and Theorem~\ref{thm:comb_int_II} we have

\begin{corollary}\label{cor:Z_Aztec_diamond}
For every Kasteleyn orientation $\vphi$,
\[
Z(A_k,a,\vphi)=\pm \sum_{(\Ts,\Fs)\in \F}\sign(\Ts,\Fs)\prod_{\vec{e}\in \Fs}a_{f_{\vec{e}}},
\]
where $\sign(\Ts,\Fs)$ is defined in Equation~\eqref{equ:def_sign}, and the sum is over all pairs of complementary trees/forests of $G^\bullet$ rooted at $b_r$, $\{\tilde{b},b_r\}$.
Moreover, there is a bijection between terms in the sum on the right-hand-side and monomials of $Z(A_k,a,\vphi)$ in the variables $a$.
\end{corollary}

Using Corollary~\ref{cor:matrix_relation_gen}, we prove two
interesting identities for the ratio function $Y(A_k,a)$ of oriented
dimers defined in Equation~\eqref{eq:defY}, see also
Equation~\eqref{eq:dskp_Aztec_dim}. This is the content of
Propositions~\ref{prop:Zp_over_Z} and Theorem~\ref{theo:kery} below. The second is used in Section~\ref{sec:periodic_columns} to prove invariance of $Y(A_k,a)$ when columns are shifted periodically.

To simplify notation, choose the signs of Equation~\eqref{equ:def_sign} defining the matrices $C(a)$ as in Figure~\ref{fig:Aztec_3},
that is, around every face corresponding to a weight of type $c_{i,j}$, resp. $d_{i,j}$, we have, starting from the right horizontal edge, $1,1,-1,-1$, resp. $-1,1,1,-1$. Recall that face weights of the right most column are labeled $c_{k,0},\dots,c_{k,k}$.

\begin{figure}[tb]
\begin{center}
\begin{overpic}[width=7cm]{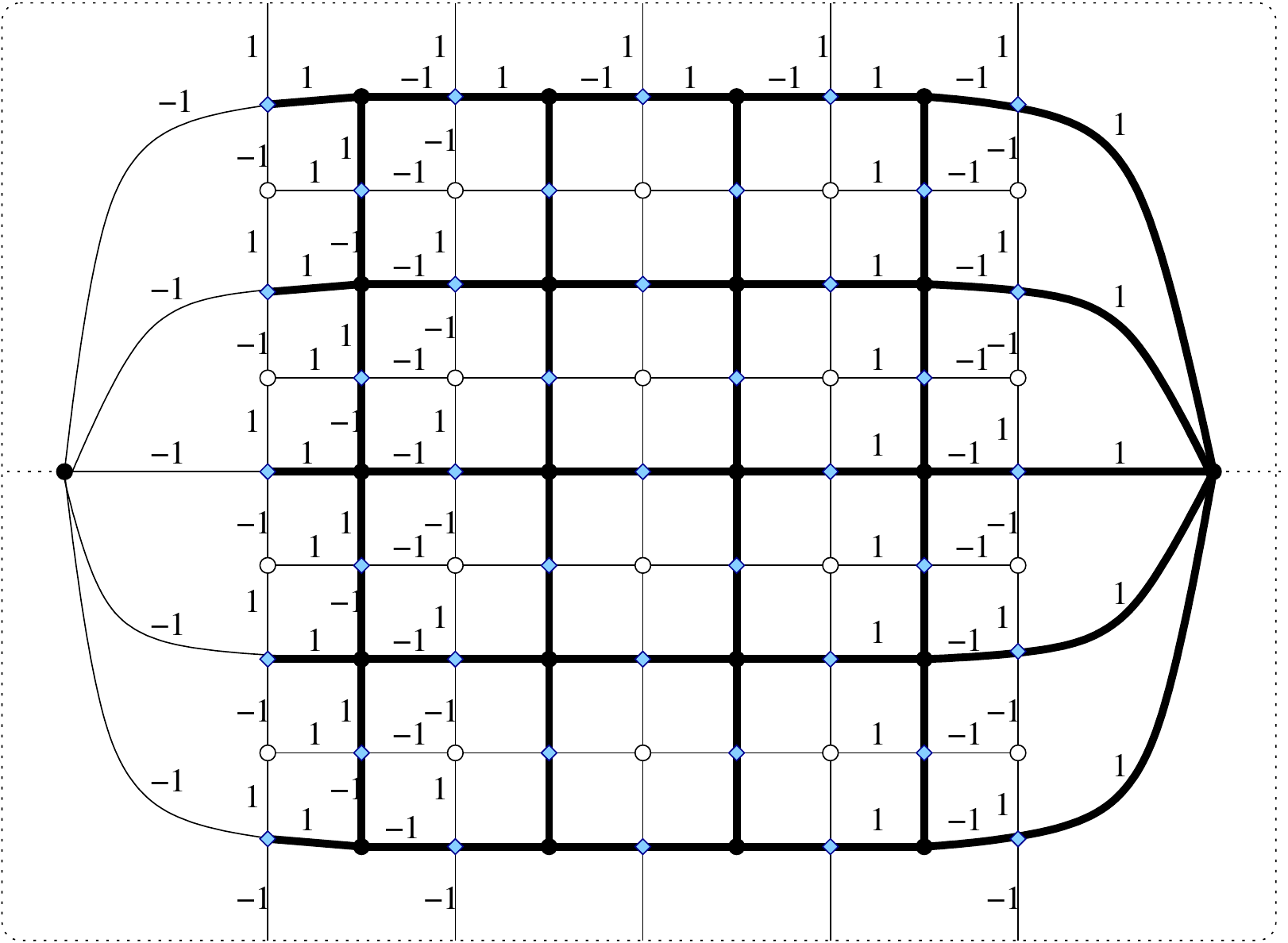}
\put(1,31){\scriptsize $b_r$}
\put(8,74){\scriptsize $w_r$}
\put(96,31){\scriptsize $\tilde{b}$}
\put(81,5){\scriptsize $c_{k,0}$}
\put(81,19){\scriptsize $c_{k,1}$}
\put(81,66){\scriptsize $c_{k,k}$}
\end{overpic}
\end{center}
\caption{Left: The double graph $\GD$ whose edges correspond to non-zero coefficients of the matrix $C(a)$; thicker lines indicate edges corresponding to non-zero coefficients of $C(a)^{\tilde{B}}$. On the edges, is the choice of signs of Equation~\eqref{equ:def_sign} defining $C(a)$ in the case of the Aztec diamond.}
\label{fig:Aztec_3}
\end{figure}
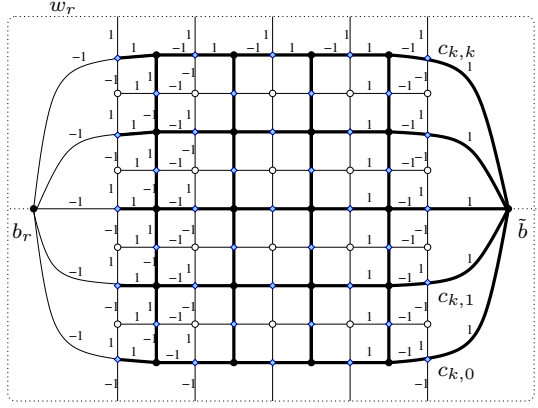

\begin{proposition}\label{prop:Zp_over_Z}
The ratio function of oriented dimers satisfies the following identity:
\begin{equation*}
Y(A_k,a)=\sum_{j=0}^{k} c_{k,j} \cdot C_{1,j}^{-1},
\end{equation*}
where $C=\begin{pmatrix} C(1)^{\tilde{B}} & C(a)^{B} \end{pmatrix}$, and its $k+1$ first rows correspond to elements of $F$ with face weights $c_{k,0},\dots,c_{k,k}$.
\end{proposition}
\begin{proof}
Observe that the matrix $C$ can be written as
\[
C=\begin{pmatrix}
C(1)^{\tilde{b}}&C(1)^{B} & C(a)^{B}
\end{pmatrix}.
\]
Fix a Kasteleyn orientation $\vphi$.
By Corollary~\ref{cor:matrix_relation_gen},
  \begin{equation}
    \label{eq:ZdetC}
    Z(A_k,a,\varphi) = \pm \det(C).
  \end{equation}
Now, applying Equation~\eqref{eq:ZdetC} to the variables $a^{-1}$, we compute the
numerator in the Aztec diamond Equation~\eqref{eq:dskp_Aztec_dim}:
\begin{equation}
  \label{eq:ZpdetC}
  \begin{split}
  \bigl(\prod_{f\in F}a_f\bigr)Z\Bigl(A_k,a^{-1},\vphi\Bigr) &= \pm
\bigl(\prod_{f\in F}a_f\bigr)
\det
 \begin{pmatrix}
 C(1)^{\tilde{b}}&C(1)^{B}&C(a^{-1})^{B}
 \end{pmatrix}\\
&=\pm \det
 \begin{pmatrix}
 C(a)^{\tilde{b}}&C(a)^{B}&C(1)^{B}
\end{pmatrix}\\
&=\pm
\det
 \begin{pmatrix}
 C(a)^{\tilde{b}}&C(1)^{B}&C(a)^{B}
\end{pmatrix},
\end{split}
\end{equation}
where in the penultimate equality we have multiplied, for every vertex
$f$ of $F$, the row of $C$ corresponding to $f$ by $a_f$. The
  signs on the right are all equal, because the number of column
  transpositions between the last two matrices is $|B|=k(k+1)$ which
  is even. Expanding the determinant over the first column, and using Equation~\eqref{eq:ZdetC} gives
\begin{equation*}
 Y(A_k,a) = \sum_{j=0}^{k} c_{k,j}(-1)^{j}\frac{\det( C_{j}^1)}{\det(C)}=
\sum_{j=0}^{k} c_{k,j}\cdot C_{1,j}^{-1},
\end{equation*}
where $C_{j}^1$ is the matrix obtained from $C$ by deleting the $j$-th row and first column.
\end{proof}

The next statement is central in proving Devron properties and
exact values for the singularities of the dSKP recurrence. We state it
as a theorem although its proof is short.

Let us denote by
$D=\begin{pmatrix} C(1)^{B} & C(a)^{B} \end{pmatrix}$, the matrix
obtained from $C$ by removing the first column. Seen as a linear
operator, $D$ takes as input a vector in $\C^{B \sqcup B}$ and its
output is a vector in $\C^F$. The following proposition relates the
kernel of $D^T$ with $Y(A_k,a)$. Note that $D^T$ has a nontrivial
kernel, because it goes from a space of dimension $|F|=2k(k+1)+1$ to a
space of dimension $2|B|=2k(k+1)$.

\begin{theorem} \label{theo:kery}
  Let $v \in \C^F$ be a nonzero vector such that
  \begin{equation} \label{eq:vkerdt}
    D^T \ v = 0.
  \end{equation}
  Let $v_{k,0},\dots, v_{k,k}$ be the entries of $v$ corresponding to the
  $k+1$ elements of $F$ with face weights $c_{k,0},\dots, c_{k,k}$ as in Figure~\ref{fig:Aztec_3}. Then,
  the ratio function of oriented dimers can be expressed as:
  \begin{equation*}
    Y(A_k,a) = \frac{\sum_{j=0}^k c_{k,j}\, v_{k,j}}{\sum_{j=0}^k v_{k,j}}.
  \end{equation*}
\end{theorem}

  \begin{proof}
    By transposing Equation~\eqref{eq:vkerdt}, because of the choice of signs, we get
    \begin{equation}
      \label{eq:vTC}
      v^T \ C =
      \begin{pmatrix}
     \sum_{j=0}^k v_{k,j} & 0 & \dots & 0
      \end{pmatrix}.
    \end{equation}
    For generic $a$ (or $c,d$) variables, we have
      $\det C \neq 0$ (as we can use \eqref{eq:ZdetC} and the fact
      that there exists at least one complementary tree/forest
      configuration), in particular by~\eqref{eq:vTC},
      $\sum_{j=0}^k v_{k,j} \neq 0$. Similarly, using
      \eqref{eq:ZpdetC}, for generic weights,
      $\sum_{j=0}^k c_{k,j} v_{k,j} \neq 0$.

    We right multiply both sides of equation~\eqref{eq:vTC} by
    $C^{-1}w$, where $w\in\C^F$ is the vector whose only non-zero
    entries are equal to $c_{k,0},\dots, c_{k,k}$. This gives
    \begin{equation*}
      v^T \
      \begin{pmatrix}
        c_{k,0} \\ \vdots \\ c_{k,k} \\ 0 \\ \vdots \\ 0
      \end{pmatrix}
      =
      \begin{pmatrix}
     \sum_{j=0}^k v_{k,j} &  0 & \dots & 0
      \end{pmatrix}
      C^{-1}
      \begin{pmatrix}
        c_{k,0} \\ \vdots \\ c_{k,k} \\ 0 \\ \vdots \\ 0
      \end{pmatrix}.
    \end{equation*}
    By Proposition~\ref{prop:Zp_over_Z}, the right-hand side is just
    $ Y(A_k,a)\cdot \sum_{j=0}^k v_{k,j}$, while the
    left-hand side is $\sum_{j=0}^k c_{k,j}\, v_{k,j}$. This shows
    that, at least as formal expression in the $a$ variables, the two
    sides of the statement of the theorem are equal. Since they are
    both analytic, this also holds when the ratio on the right is
    well-defined in $\hC$, moreover it is undefined in $\hC$ iff
    $Y(A_k,a)$ is undefined.
\end{proof}

\subsection{Constant columns} \label{sec:cst_col}

We consider the special case where columns of $d$ are constant,
\emph{i.e.}~for some $(d_i)_{0\leq i\leq k-1}$, see also Figure~\ref{fig:Aztec_4},
\begin{equation}
  \label{eq:cst_weights}
  \forall\, 0\leq i,j\leq k-1,\quad  d_{i,j}=d_i.
\end{equation}

Denote by
$\tilde{a}=(\tilde{a}_{i,j})$ face weights obtained by a vertical cyclic shift:
\begin{equation}
  \label{eq:defatilde1}
  \forall 0\leq i,j\leq 2k \, \text{ s.t. } [i+j]_2=0, \quad\quad
  \tilde{a}_{i,j}= a_{i ,\,[j+2]_{2(k+1)}},
\end{equation}
then we have the following.

\begin{theorem}\label{thm:vert_translation}
Suppose that all odd columns of the Aztec diamond $A_k$
have constant weights
 $(d_i)_{0\leq i\leq k-1}$ as in
Equation~\eqref{eq:cst_weights}. Then, for every Kasteleyn orientation
$\vphi$, the partition function of oriented dimers associated to face
weights $a$, resp.~to vertically shifted face weights $\tilde{a}$
of Equation~\eqref{eq:defatilde1}, are equal up to an explicit sign:
\begin{equation*}
Z(A_k,a,\vphi)=(-1)^k Z(A_k,\tilde{a},\vphi).
\end{equation*}
Furthermore, the corresponding ratio functions of oriented dimers are equal:
\begin{equation*}
Y(A_k,a)=Y(A_k,\tilde{a}).
\end{equation*}
\end{theorem}

Theorem~\ref{thm:vert_translation} is a consequence of Theorem~\ref{thm:col_degen_I} below, which we state and prove first. We need the following definition. A \emph{permutation spanning forest $\Fs$ of $G^\bullet$ (rooted at $b_r$, $\tilde{b}$)} is a spanning forest $\Fs$ of $G^\bullet$ (with two connected components) rooted at $b_r,b$, such that:
\begin{itemize}
 \item it contains no vertical edge,
 \item it has one absent edge per row, and absent edges form a permutation of $\{0,\dots,k\}$. More precisely, the graph $G^\bullet$ has $k+1$ edge-rows, each
having $k+1$ edges. These edges are written as $(e_{i,j})_{0\leq i,j\leq k}$, where $i$ represents the column from left to right, and $j$ the row from bottom to top. Note that horizontal edges $(e_{i,j})_{0\leq i,j\leq k}$ of $G^\bullet$ are in correspondence with face weights $(c_{i,j})_{0\leq i,j\leq k}$ as defined in Equation~\eqref{eq:cst_weights}.
For the permutation spanning forest $\Fs$, denote by $e_{\tau(0),0},\dots,e_{\tau(k),k}$ the absent edges, then $\tau$ is a permutation of $\S_{k+1}$, see Figure~\ref{fig:Aztec_4} for an example.
\end{itemize}

\begin{remark}\label{rem:permutation_tree_forest}$\,$
\begin{itemize}
\item Given a permutation $\tau\in\S_{k+1}$, the edge configuration of $G^\bullet$ with only horizontal edges and absent edges $e_{\tau(0),0},\dots,e_{\tau(k),k}$ is a spanning forest rooted at $b_r,\tilde{b}$, \emph{i.e.}, a permutation spanning forest, denoted by $\Fs(\tau)$.
\item Let $\Fs$ be a permutation spanning forest of $G^\bullet$, and let $\Ts$ be the complementary edge configuration in $G^\bullet$. Then, $\Ts$ is spanning tree of $G^\bullet$ containing all vertical edges of $G^\bullet$, and we consider it as rooted towards $b_r$. That is, $(\Ts,\Fs)$ is a pair of complementary spanning tree/forest of $G^\bullet$ rooted at $b_r$, $\{b_r,\tilde{b}\}$.
\end{itemize}
\end{remark}

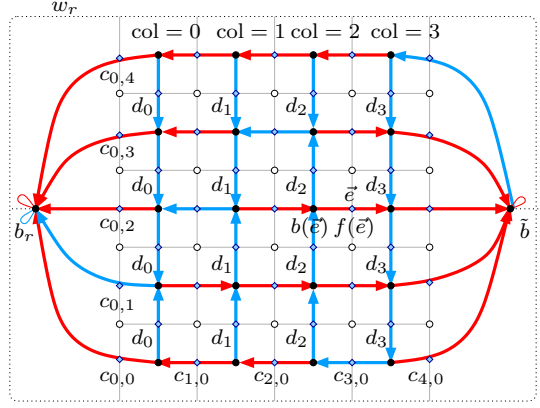
\begin{figure}[tb]
\begin{center}
\begin{overpic}[width=7cm]{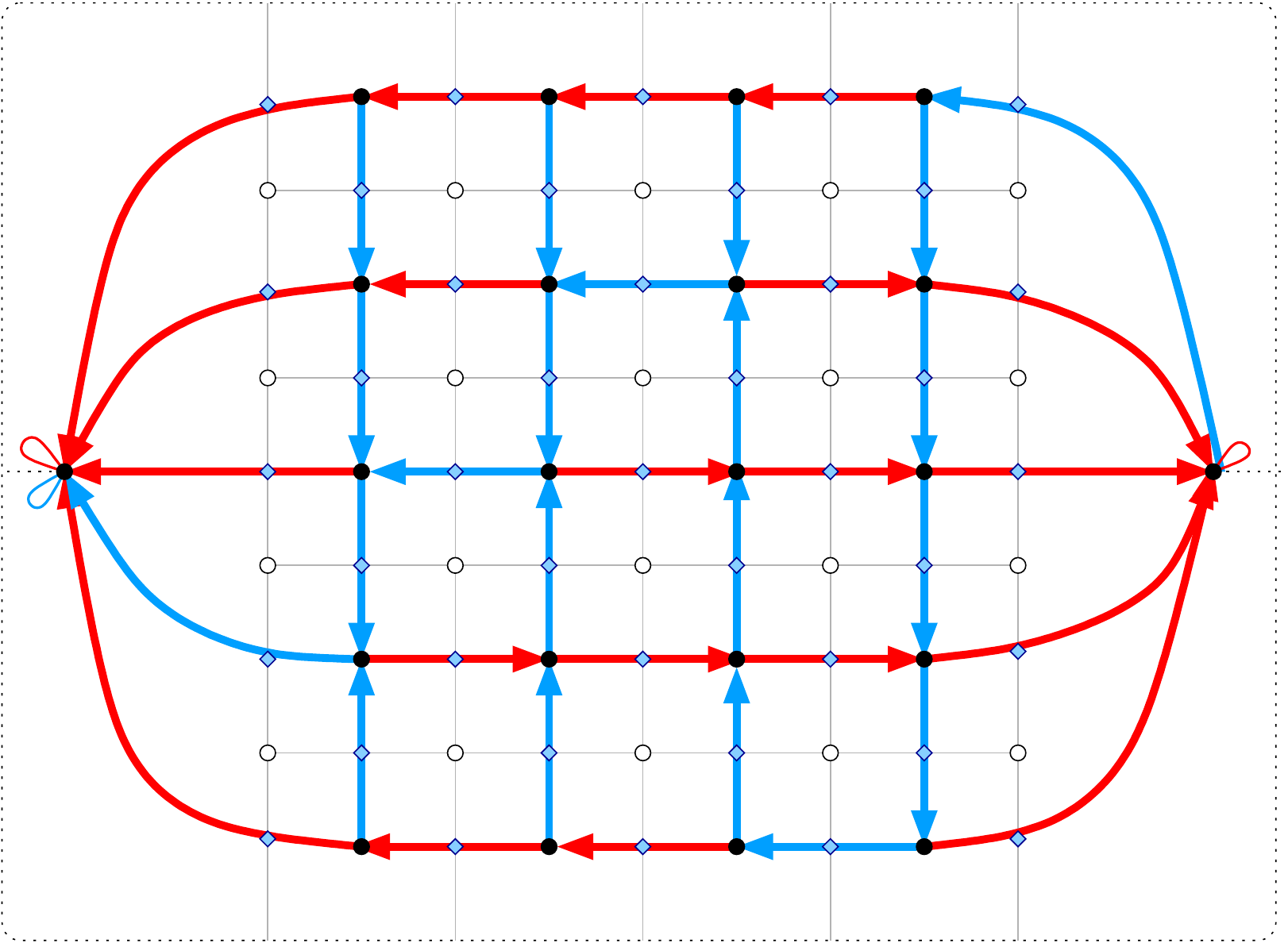}
\put(1,31){\scriptsize $b_r$}
\put(8,74){\scriptsize $w_r$}
\put(96,31){\scriptsize $\tilde{b}$}
\put(17,4){\scriptsize $c_{0,0}$}
\put(31,4){\scriptsize $c_{1,0}$}
\put(46,4){\scriptsize $c_{2,0}$}
\put(61,4){\scriptsize $c_{3,0}$}
\put(75,4){\scriptsize $c_{4,0}$}
\put(17,18){\scriptsize $c_{0,1}$}
\put(17,33){\scriptsize $c_{0,2}$}
\put(17,47){\scriptsize $c_{0,3}$}
\put(17,61){\scriptsize $c_{0,4}$}
\put(23,11){\scriptsize $d_0$}
\put(38,11){\scriptsize $d_1$}
\put(52,11){\scriptsize $d_2$}
\put(67,11){\scriptsize $d_3$}
\put(23,25){\scriptsize $d_0$}
\put(23,40){\scriptsize $d_0$}
\put(23,55){\scriptsize $d_0$}
\put(38,25){\scriptsize $d_1$}
\put(38,40){\scriptsize $d_1$}
\put(38,55){\scriptsize $d_1$}
\put(52,25){\scriptsize $d_2$}
\put(52,40){\scriptsize $d_2$}
\put(52,55){\scriptsize $d_2$}
\put(67,25){\scriptsize $d_3$}
\put(67,40){\scriptsize $d_3$}
\put(67,55){\scriptsize $d_3$}
\put(23,69){\scriptsize $\col =0$}
\put(39,69){\scriptsize $\col =1$}
\put(53,69){\scriptsize $\col =2$}
\put(68,69){\scriptsize $\col =3$}
\put(63,38){\scriptsize $\vec{e}$}
\put(53,32){\scriptsize $b(\vec{e})$}
\put(61,32){\scriptsize $f(\vec{e})$}
\end{overpic}
\end{center}
\caption{
        Case $n=4$; notation $(c_{i,j})_{0\leq i,j\leq 4},\, (d_i)_{0\leq i \leq 3}$ for faces weights; permutation spanning forest $\Fs$ (red) corresponding to the permutation
         $\tau=\scriptsize{\begin{pmatrix}
0&1&2&3&4\\
3&0&1&2&4
\end{pmatrix}}$,
 and its complementary tree (blue); a directed edge $\vec{e}$ such that $\col b(\vec{e})=2$.
}
\label{fig:Aztec_4}
\end{figure}

We need one more notation. Using the reverse Temperley trick, every directed edge $\vec{e}$ of $G^\bullet$ is in correspondence with an edge $b(\vec{e})f(\vec{e})$ of $\GD$ where $b(\vec{e})$ is a vertex of $G^\bullet$, $f(\vec{e})$ is a vertex of type $F$ of $\GD$, and the directed edge $(b(\vec{e}),f(\vec{e}))$ has the same orientation as $\vec{e}$, see Figure~\ref{fig:Aztec_4}. Now, the graph $G^\bullet$ has $k$ columns of $k$ vertical edges, labeled $\{0,\dots,k-1\}$ from left to right. For every directed edge $\vec{e}$ not oriented away from $b_r$ or $\tilde{b}$, we have $b(\vec{e})\notin \{b_r,\tilde{b}\}$, and we let $\col b(\vec{e})$ be the label of the column to which the vertex $b(\vec{e})$ belongs. Then, we prove the following.

\begin{theorem}\label{thm:col_degen_I}
Suppose that all odd columns of the Aztec diamond $A_k$ have constant weights $(d_i)_{0\leq i\leq k-1}$ as in Equation~\eqref{eq:cst_weights}. Then, for every Kasteleyn orientation $\vphi$, the following combinatorial identity holds for the partition function of oriented dimers:
\begin{equation}\label{equ:thm_col_degen_I}
Z(A_k,a,\vphi)=\pm \sum_{\tau\in \S_{k+1}} \sgn(\tau) \prod_{\vec{e}\in \Fs(\tau)}(a_{f(\vec{e})}-d_{\col b(\vec{e})}).
\end{equation}
\end{theorem}

\begin{proof}
We start from Corollary~\ref{cor:matrix_relation_gen}:
\begin{equation*}
Z(A_k,a,\vphi)=\pm \det
(C)=\pm \det
\begin{pmatrix}
C(1)^{\tilde{b}}&C(1)^{B}&C(a)^{B}
\end{pmatrix}.
\end{equation*}
For every $0\leq i\leq k-1$, consider the black vertices
$\{b_{i,0},\dots,b_{i,k}\}$ belonging to the column $i$ of vertical
edges of $G^\bullet$, and do the following operations: for every
$0\leq j\leq k$, multiply the column corresponding to $b_{i,j}$ in
$C(1)^{B}$ by $d_i$ and subtract it from the corresponding column in
$C(a)^{B}$;~the column in $C(1)^B$ is left unchanged. This operation yields a matrix $C'(a)^{B}$ and does not change the determinant. We have
\begin{equation*}
Z(A_k,a,\vphi)=\pm \det
\begin{pmatrix}
C(1)^{\tilde{b}}&C(1)^{B}&C'(a)^{B}
\end{pmatrix},
\end{equation*}
where for every $f\in F$, every $b\in B$,
\begin{equation*}
C'(a)^{B}_{f,b}=
\begin{cases}
0& \text{ if $fb$ corresponds to a vertical edge by the Temperley trick}\\
\pm(a_f-d_{\col b}) & \text{ if $fb$ corresponds to a horizontal edge by the Temperley trick},
\end{cases}
\end{equation*}
where the sign is defined as for the matrix $C(a)$, see Equation~\eqref{equ:choice_sign}.

We now compute the determinant similarly to what we have done in the proof of Theorem~\ref{thm:comb_int_II}, using the notation introduced for that purpose. Non-zero terms in the permutation expansion of the determinant correspond to pairs of matchings of $\Md$. Then, applying Temperley's trick, we show that the only remaining configurations are pairs of complementary spanning trees/forests rooted at $b_r$ and $\{b_r,\tilde{b}\}$. But, in the present setting, because of the definition of $C'(a)^{B}$, the graph $(G')^\bullet$, on which spanning forests rooted at $\{b_r,\tilde{b}\}$ live, is the graph $G^\bullet$ with no vertical edge (since they have weight 0 in the matrix). Returning to the definition of a spanning forest rooted at $\{b_r,\tilde{b}\}$, we deduce that this component must contain exactly $k$ edges per row. Consider such a spanning forest $\Fs$. Then we know that the complementary configuration $\Ts$ must be a spanning tree rooted at $b_r$. Since all vertical edges are absent from $(G')^\bullet$, they must all be present in $\Ts$. Suppose now that the absent horizontal edges of $\Fs$ do not form a permutation, then $\Ts$ must contain two horizontal edges $e_{i,j}$ $e_{i,j'}$ for some column $i$ and some distinct $j,j'$. This implies that $\Ts$ has a cycle which contradicts it being a spanning tree. Thus $\Fs$ must be a permutation spanning forest of $G^\bullet$ rooted at $\{b_r,\tilde{b}\}$. By Remark~\ref{rem:permutation_tree_forest}, we then have that $\Ts$ is indeed a spanning tree rooted at $b_r$. Using the specific form of the matrix $C'(a)$, we have so far proved that,
\begin{equation*}
Z(A_k,a,\vphi)=\pm \sum_{(\Ts,\Fs)\in\F'} \sign(\Ts,\Fs)\prod_{\vec{e}\in \Fs}(a_{f(\vec{e})}-d_{\col b(\vec{e})}),
\end{equation*}
where $\F'$ is the set of pairs of complementary spanning trees/forests rooted at $b_r$, $\{b_r,\tilde{b}\}$, such that $\Fs$ is a permutation spanning forest rooted at $\{b_r,\tilde{b}\}$.

We are thus left with showing that $\sign(\Ts,\Fs)$ is equal to the signature of the permutation $\tau$ corresponding to $\Fs$ (up to a global $\pm$ sign). To this purpose, it suffices to show that if $\tau,\tau'$ are two permutations differing by a transposition, corresponding to two permutation spanning forests $\Fs,\Fs'$, then the product of $\sign(\Ts,\Fs)$ and $\sign(\Ts',\Fs')$ is equal to $-1$. Let $j< j'\in\{0,\dots,k\}$ be the indices of the rows such that $\tau(j)=\tau'(j')=i$, $\tau(j')=\tau'(j)=i'$, for $i,i'\in\{0,\dots,k\}$, and without loss of generality suppose that $i<i'$.

Denote by $(\Ms_1,\Ms_2)$, resp. $(\Ms_1',\Ms_2')$, the pair of matchings of $\Md$ corresponding to $(\Fs,\Ts)$, resp. $(\Fs',\Ts')$, and by $\sigma$, resp. $\sigma'$, the permutation associated to $(\Ms_1,\Ms_2)$, resp. $(\Ms_1',\Ms_2')$, see Equation~\eqref{equ:def_sign} for definition. Our goal is to prove that
\begin{equation}\label{equ:proof_equ_1}
\sgn(\sigma)\sgn(\sigma')C(1)_{f_{\sigma(1),b_1}}\dots C(1)_{f_{\sigma(\ell),b_\ell}}
C(1)_{f_{\sigma'(1),b_1}}\dots C(1)_{f_{\sigma'(\ell),b_\ell}}=-1.
\end{equation}
To this purpose, we need to study the superimposition of $(\Ms_1,\Ms_2)$ and $(\Ms_1',\Ms_2')$, see Figure~\ref{fig:Aztec_5}. By definition of $\Md$, we know that it consists of cycles such that: each vertex of $B$ has degree 4 (1 from each of $\Ms_1,\Ms_2,\Ms_1',\Ms_2'$), each vertex of $F$ has degree 2 (1 from each of $(\Ms_1,\Ms_2)$, $(\Ms_1',\Ms_2')$), $\tilde{b}$ has degree 2 (1 from each of $\Ms_1$, $\Ms_1'$), $b_r$ has degree $0$. Recall that vertices of $B$ receive two labels; using colors, this translates in the fact that blue edges incident to a vertex of $B$ and red ones come from the two copies of that vertex.

\begin{figure}[tb]
\begin{center}
\begin{overpic}[width=\linewidth]{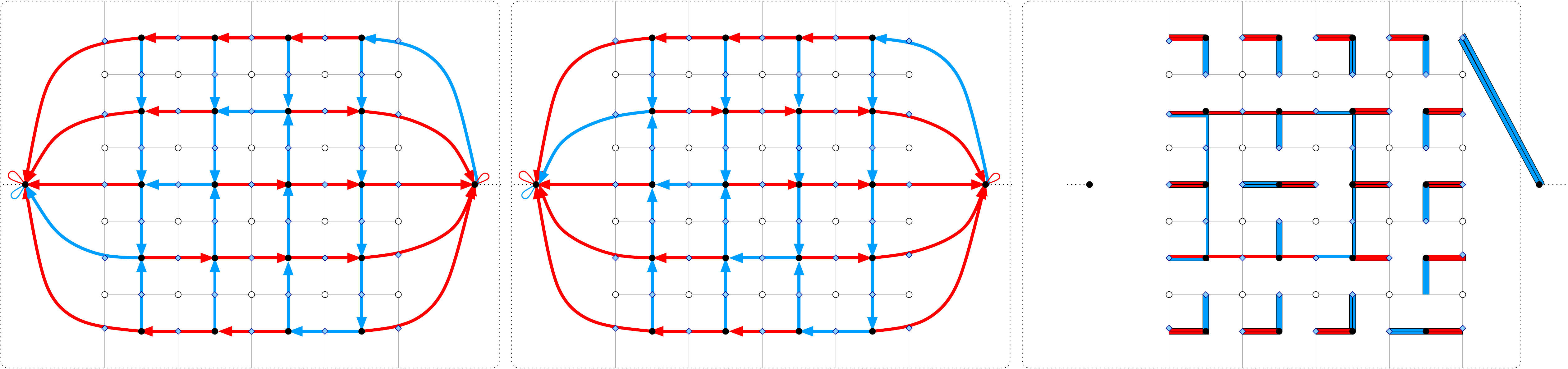}
\put(5,6){\scriptsize $e_{0,1}$}
\put(15,17){\scriptsize $e_{2,3}$}
\put(37,14){\scriptsize $e_{0,3}$}
\put(47,6){\scriptsize $e_{2,1}$}
\put(73,22){\scriptsize $\col=0$}
\put(83,22){\scriptsize $\col=2$}
\end{overpic}
\end{center}
\caption{Left and middle: permutation spanning forests $\Fs,\Fs'$
  (red) differing by a transposition: the associated permutations
  $\tau,\tau'$ are such that $\tau(1)=\tau'(3)=0$,
  $\tau(3)=\tau'(1)=2$, \emph{i.e.}, $i=0, i'=2, j=1, j'=3$, and their
  complementary spanning trees $\Ts,\Ts'$ (blue). Right: corresponding
  superimposition of $(\Ms_1,\Ms_2), (\Ms_1',\Ms_2')$.}
\label{fig:Aztec_5}
\end{figure}
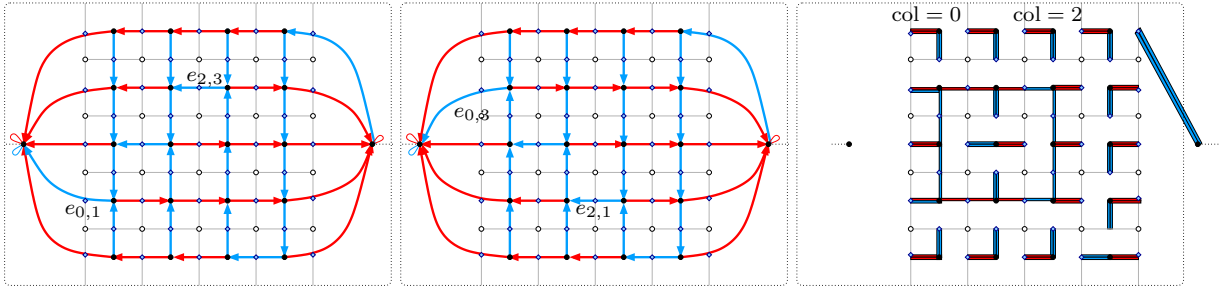

Looking at the orientation of the edges of $(\Ts,\Fs)$, $(\Ts',\Fs')$, and using that the permutation spanning forests $\Fs$, $\Fs'$ differ by a transposition, we have that the superimposition of $(\Ms_1,\Ms_2)$ and $(\Ms_1',\Ms_2')$ consists of doubled edges of the same color, and a cycle of $\GD$ between the $i$-th and $i'$-th columns of $G^\bullet$, and the $j$-th and $j'$-th rows of $G^\bullet$, with two length-two detours on the left, at the level of the $j$-th and $j'$-th rows,
see Figure~\ref{fig:Aztec_5} (right). When $i'\neq k$, resp. $i'=k$, the cycle has length $4[(i'-i)+(j'-j)+1]$, resp. $4(i'-i)+2(j'-j)+4$ (the last column of the cycle is reduced to a point). This implies  that
\begin{align*}
\sgn(\sigma)\sgn(\sigma')=
\begin{cases}
(-1)^{2[(i'-i)+(j'-j)]+1}& \text{ if $i'\neq k$},\\
(-1)^{2(i'-i)+(j'-j)+1}& \text{ if $i'= k$}.
\end{cases}
\end{align*}
Observing that by our choice of signs for the matrix $C(a)$, the two half edges of $\GD$ corresponding to an edge of $G^\bullet$ have opposite signs, we deduce that
\begin{align*}
C(1)_{f_{\sigma(1),b_1}}\dots C(1)_{f_{\sigma(\ell),b_\ell}}
C(1)_{f_{\sigma'(1),b_1}}\dots C(1)_{f_{\sigma'(\ell),b_\ell}}=
\begin{cases}
(-1)^{2[(i'-i)+(j'-j)]} & \text{ if $i'\neq k$},\\
(-1)^{2(i'-i)+(j'-j)} & \text{ if $i'= k$}.
\end{cases}
\end{align*}
Taking the product of the signature and coefficients contributions, we deduce that Equation~\eqref{equ:proof_equ_1} is indeed true.
\end{proof}

We are now ready to prove Theorem~\ref{thm:vert_translation}.

\begin{proof}[Proof of Theorem~\ref{thm:vert_translation}]
Let $c=(0\, \dots\, k)$ be the permutation cycle corresponding to the vertical cyclic shift of the weights. Then, there is a bijection between $\{\tau:\tau\in \S_{k+1}\}$ and $\{\tau\circ c: \tau\in S_{k+1}\}$. Moreover, given $\tau\in \S_{k+1}$, the product of the directed edge weights of $\Fs(\tau)$ in the expansion~\eqref{equ:thm_col_degen_I} with weight function $a$, is equal to that of $\Fs(\tau\circ c)$ with weight function $\tilde{a}$. As a consequence, by Equation~\eqref{equ:thm_col_degen_I}, we have that the oriented dimer partition functions are related by:
\begin{equation}\label{equ:proof_Z_vert_translation}
Z(A_k,\tilde{a},\varphi)=\sgn(c) Z(A_k,a,\varphi)=(-1)^{k}Z(A_k,a,\varphi).
\end{equation}
The equality between the ratio functions $Y(A_k,a)$ and $Y(A_k,\tilde{a})$ is obtained by returning to Equation~\eqref{eq:dskp_Aztec_dim}, giving the explicit computation of $Y(A_k,a)$ in the Aztec diamond case, applying Equation~\eqref{equ:proof_Z_vert_translation} to the face weights $(a^{-1})$, and using that $(\prod_{f\in F}a_f)=(\prod_{f\in F}\tilde{a}_f)$.
\end{proof}

\subsection{Schwarzian Dodgson condensation}\label{sec:Dodgson_condens}

We now suppose that all odd columns are set to the \emph{same} value,
that is, for some $d$,
\begin{equation}
  \label{eq:cst_cst_weights}
  \forall\, 0\leq i,j\leq  k-1, \quad d_{i,j}=d.
\end{equation}

Let $N$ be the matrix of size $(k+1)\times(k+1)$ whose coefficients are defined by
\[
\forall\, 0\leq i,j \leq k,\quad N_{i,j}=\frac{1}{c_{i,j}-d}.
\]
Then, as a consequence of Theorem~\ref{thm:col_degen_I}, we obtain
\begin{corollary}
  \label{cor:dodgson}
Suppose that \emph{all} odd columns of the Aztec diamond $A_k$ have constant weight $d$ as in Equation~\eqref{eq:cst_cst_weights}. Then, for every Kasteleyn orientation $\vphi$, the following combinatorial identity holds for the partition function of oriented dimers:
\begin{equation*}
Z(A_k,a,\vphi)=\pm \prod_{0\leq i,j\leq k}(c_{i,j}-d)\cdot \det N.
\end{equation*}
Moreover, for the ratio function of oriented dimers, we have
\begin{equation*}
Y(A_k,a)=d + \sum_{0\leq i,j \leq k} \left( N^{-1} \right)_{i,j}.
\end{equation*}

\end{corollary}

\begin{proof}
We start from Equation~\eqref{equ:thm_col_degen_I} of Theorem~\ref{thm:col_degen_I}.
Let $\tau$ be a permutation of $\S_{k+1}$, and recall that the permutation spanning forest $\Fs(\tau)$ contains all horizontal edges except
$e_{\tau(0),0},\dots,e_{\tau(k),k}$. Since the weights of the faces of the Aztec diamond labeled $d_{i,j}$ are all equal to $d$, the product on the right-hand-side of Equation~\eqref{equ:thm_col_degen_I} is independent of the orientation of the edges of $\Fs(\tau)$. As a consequence, we can write
\begin{equation*}
\prod_{e\in \Fs(\tau)}(a_{f(e)}-d)=
\prod_{0\leq i,j\leq k} (c_{i,j}-d)
\prod_{e\notin \Fs(\tau):\, e \text{ horizontal}} \frac{1}{a_{f(e)}-d}.
\end{equation*}
Observing that
\begin{equation*}
\prod_{e\notin \Fs(\tau):\, e \text{ horizontal}} \frac{1}{a_{f(e)}-d}=
\prod_{j=0}^k \frac{1}{c_{\tau(j),j}-d},
\end{equation*}
we deduce from Equation~\eqref{equ:thm_col_degen_I} that,
\begin{equation}
  \label{eq:zan_detN}
Z(A_k,a)=\pm \prod_{0\leq i,j\leq k} (c_{i,j}-d)
\sum_{\tau\in \S_{k+1}}\sgn(\tau)\prod_{j=0}^k \frac{1}{c_{\tau(j),j}-d}=
\pm \prod_{0\leq i,j\leq k} (c_{i,j}-d)\cdot \det(N^t).
\end{equation}

To compute $Y(A_k,a)$, recalling \eqref{eq:dskp_Aztec_dim}, we first compute
the numerator. Since there are $k^2$ faces with weight $d$, and using
Equation~\eqref{eq:zan_detN}, it is equal to
\begin{equation*}
  \begin{split}
     \prod_{0\leq i,j \leq k} c_{i,j} \ \cdot d^{k^2}
    \cdot Z(A_k,a^{-1},\varphi)
    = & \pm \prod_{0\leq i,j \leq k} c_{i,j} \left(
        c_{i,j}^{-1} - d^{-1} \right) \ \cdot d^{k^2} \cdot
    \det\left( (c_{i,j}^{-1}-d^{-1})^{-1} \right)_{0\leq i,j \leq k}
    \\
    = &  \pm \prod_{0\leq i,j \leq k} (d-c_{i,j}) \ \cdot d^{-1-2k} \cdot
    \det\left( \frac{d c_{i,j}}{d-c_{i,j}} \right)_{0\leq i,j \leq k} \\
    = &  \pm \prod_{0\leq i,j \leq k} (c_{i,j}-d) \  \cdot d^{-k} \cdot
    \det\left( \frac{c_{i,j}}{c_{i,j}-d} \right)_{0\leq i,j \leq k}.
  \end{split}
\end{equation*}
In the last line, we moved the factor $d$ out of the determinant, and
we changed the sign in both the product and the matrix, which gives
signs that cancel out. Then, we write $\frac{c_{i,j}}{c_{i,j}-d}
= 1 + \frac{d}{c_{i,j}-d}$, and we use multi-linearity of the columns in the
determinant. In the resulting expression, terms with at least two
columns of ones disappear. When there is exactly one column of ones,
we may expand on this column, and get a sum on minors of size
$k$. This gives

\begin{equation*}
  \det\left(\frac{c_{i,j}}{c_{i,j}-d} \right)_{0\leq i,j \leq k} =
  \det \left( \frac{d}{c_{i,j}-d} \right)_{0\leq i,j \leq k} +
  \sum_{0\leq i_0,j_0 \leq k} (-1)^{i_0 + j_0} \det \left( \frac{d}{c_{i,j}-d}
  \right)_{i \neq i_0, j \neq j_0}.
\end{equation*}
Putting this back into the previous equation, and extracting again the
factors $d$ from the matrices, we recognize $N$ and the entries of
$N^{-1}$:
\begin{equation*}
  \prod_{0\leq i,j \leq k} c_{i,j} \ \cdot d^{k^2}
  \cdot Z(A_k,a^{-1},\varphi)
  = \pm \prod_{0\leq i,j \leq k} (c_{i,j}-d) \ \cdot \det N \cdot \left( d  +
     \sum_{0\leq i_0, j_0 \leq k} N^{-1}_{i_0,j_0}\right).
\end{equation*}
Dividing by Equation~\eqref{eq:zan_detN} gives the formula for $Y(A_k,a)$.
\end{proof}

\subsection{Periodically constant columns}\label{sec:periodic_columns}

We turn to a case where constant columns appear periodically, which is
a generalization of Section~\ref{sec:cst_col}.

Let $m\geq 2, p\geq 1$, and let $k=mp-2p+1$. Suppose that the weights
$\left(c_{i,j}\right),\left(d_{i,j}\right)$ are $(0,m)$-periodic (or
equivalently that $\left(a_{i,j}\right)$ are $(0,2m)$-periodic), and
that every $p$-th odd column is constant, that is
\begin{equation}
  \label{eq:per_weights}
  \begin{split}
    \forall i,j, \quad \ & c_{i,j}=c_{i,j+m}, \\
    & d_{i,j}=d_{i,j+m}, \\
    & \text{if } [i]_p = 0, \text{ then } d_{i,j}=d_{i/p},
  \end{split}
\end{equation}
whenever these are well-defined,
see
Figure~\ref{fig:single_col}, left; note that we switched the role of black
and white vertices compared to Figure~\ref{fig:Aztec_1}, which will be
useful in the forthcoming proof.

We again consider translated weights, taking
periodicity into account:

\begin{equation}
  \label{eq:defatilde2}
\forall i,j \in \Z \, \text{ s.t. } [i+j]_2=0, \, \quad\quad  \ \tilde{a}_{i,j}=a_{i, [j+2]_{2m}},
\end{equation}
see Figure~\ref{fig:single_col}, right.

\begin{figure}[tb]
  \centering
  \includegraphics[width=14cm]{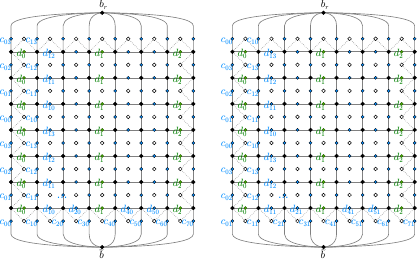}
  \caption{Left: the weights $a$ corresponding to the application of
    Theorem~\ref{thm:single_col}. Here $m=4$ and $p=3$, so the Aztec
    diamond $A_k$ has size $k=mp-2p+1=7$; its boundary is shown in
    dotted lines. The graph $G^\bullet$ is shown in solid lines and
    black dots, recall that its vertices are
    $\{B,\tilde{b},b_r\}$. The weights $(c_{i,j})$ and $(d_{i,j})$ are
    attached to elements of $F$, or equivalently edges of $G^\bullet$,
    shown as green and blue diamonds. The weights are periodic with
    period $(0,4)$. In $(d_{i,j})$, constant columns appear every $3$
    columns, and correspond to green diamonds. No other periodicity is
    assumed on the weights. Right: the shifted weights $\tilde{a}$ on
    the same graph.
  }
  \label{fig:single_col}
\end{figure}

\begin{theorem}\label{thm:single_col}
  Let $m\geq 2, p\geq 1$, and let $k=mp-2p+1$. Suppose that the
  weights of the Aztec diamond $A_k$ satisfy \eqref{eq:per_weights}.
  Then
  the ratio function of oriented
  dimers associated to face weights $a$ and translated weights
  $\tilde{a}$ of \eqref{eq:defatilde2} are equal:
  \begin{equation*}
    Y(A_k,a)=Y(A_k,\tilde{a}).
  \end{equation*}
\end{theorem}

The proof is more abstract than those of the previous
sections. It uses Theorem~\ref{theo:kery},
the matrix $C$ defined in Sections~\ref{sec:comb_sol_II_2}~and~the associated matrix $D$ of Section~\ref{sec:Appl_Aztec_diamond}; but this time, the proof
is not based on a combinatorial identification of the ratios of
partition functions. A purely combinatorial proof still eludes us.

\begin{proof}

  Recall that Theorem~\ref{theo:kery} expresses the ratio function $Y(A_k,a)$ of oriented dimers using a non-zero vector $v\in\C^F$ in the kernel of $D^T$,
  where $D=\begin{pmatrix}
    C(1)^B& C(a)^B
  \end{pmatrix},$ is the matrix obtained from the matrix $C$ by
  removing the column corresponding to $\tilde{b}$.  The proof
  consists in creating such a vector $v$ that is in addition
  $(0,m)$-periodic, and using it to prove the invariance result.

  For that purpose, we consider a graph $\bar{G^\bullet}$ on a
  cylinder, obtained as a quotient of
  $G^\bullet\setminus \{b_r,\tilde{b}\}$ by $(0,m)$, see
  Figure~\ref{fig:gcyl}. This graph has vertices $\bar{B}$ and edges
  $\bar{F}$ equipped with weights inherited from that of
  $G^\bullet$. On $\bar{G^\bullet}$ we define the analogous operators
  $\bar{C}(1)^{\bar{B}}, \bar{C}(a)^{\bar{B}},$ and
  $\bar{D} =\begin{pmatrix} \bar{C}(1)^{\bar{B}} &
    \bar{C}(a)^{\bar{B}}
  \end{pmatrix}$. Our goal is to find a vector in $\ker \bar{D}^T$,
  and lift it to a vector in $\ker D^T$.

  \begin{figure}[tb]
    \centering
    \includegraphics[width=8cm]{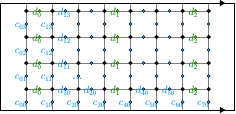}
    \caption{The quotient graph $\bar{G^\bullet}$, defined on a
      cylinder, corresponding to the setting of
      Figure~\ref{fig:single_col}. Constant columns are
      shown with green diamonds.}
    \label{fig:gcyl}
  \end{figure}

  We claim that
  \begin{equation}
    \label{eq:dimkerdt}
    \dim \ker \bar{D}^T \geq 1.
  \end{equation}
  To show this, let us first compute the dimensions of the initial and
  target space of $\bar{D}$. Simple counting shows that
  $2|\bar{B}| = 2mk+2m$ and $|\bar{F}|=2mk+m$. Then, applying the
  rank-nullity theorem to $\bar{D}$ and $\bar{D}^T$, we get that the
  statement \eqref{eq:dimkerdt} is equivalent to
  $\dim \ker \bar{D} \geq 1+2|\bar{B}|-|\bar{F}| = m+1$. Let us find $m+1$
  free vectors in $\ker \bar{D}= \ker \begin{pmatrix} \bar{C}(1)^{\bar{B}} &
    \bar{C}(a)^{\bar{B}}
  \end{pmatrix}$. Again we are considering these vectors as defined on two
  copies of $\bar{B}$.

  Consider the connected components of black vertices obtained by
  removing the horizontal edges on the ``constant columns'' (shown in
  green in Figure~\ref{fig:gcyl}). Using the exact value of
  $k=mp-2p+1$, we get that there are $m$ such connected
  components. Consider a vector that is equal to $\alpha_i$
  (resp. $\beta_i$) on the whole $i$-th connected component, in the
  first (resp. second) copy of $\bar{B}$. Then, for any edge in
  $\bar{F}$ that is not in one of the constant columns, this vector
  will produce a zero (as $\bar{C}(1)^{\bar{B}}$ outputs the
  difference of the two values adjacent to $f$, and  $\bar{C}(a)^{\bar{B}}$
  outputs this difference multiplied by $a_f$). So this vector is in
  $\ker \bar{D}$ \emph{iff} edges in the constant columns also output
  zeros, which amounts to
  \begin{equation*}
    \forall i \in \{0,\dots, m-2\}, \ \alpha_{i+1} - \alpha_i + d_i
    \left( \beta_{i+1}-\beta_i \right) = 0.
  \end{equation*}
  This is a system of $m-1$ equation on $2m$ variables $\left(
    \alpha_i,\beta_i \right)$, so it has rank at most $m-1$, and its
  kernel has dimension at least $2m - (m-1) = m+1$. It is clear that
  these $m+1$ free solutions produce $m+1$ free vectors in $\ker
  \bar{D}$. This proves Equation~\eqref{eq:dimkerdt}.

  Therefore, we can fix a nonzero vector $\bar{v} \in \ker
  \bar{D}^T$. Consider the weight-preserving quotient by $(0,m)$,
  which maps $F$ onto $\bar{F}$. Using this application, we can lift
  the vector $\bar{v}\in\C^{\bar{F}}$ to a vector $v\in \C^F$, by
  setting the value $v_f$ to be $\bar{v}_{\bar{f}}$, for any $f\in
  F$. The crucial observation is that $v\in \ker D^T$. Indeed, for any
  vertex $b\in B$, the neighbouring elements of $F$ are the same in
  the initial graph $G^\bullet$ as in the quotient graph
  $\bar{G^\bullet}$, so the computation of rows labeled $b$ in the
  expression $D^Tv$ is the same as that labeled $\bar{b}$ in
  $\bar{D}^T\bar{v}$. Therefore, by Theorem~\ref{theo:kery},
  \begin{equation}
    \label{eq:ya_quot_v}
    Y(A_k,a) = - \frac{\sum_{j=0}^k c_{j,0} \bar{v}_{j,0}}{\sum_{j=0}^k \bar{v}_{j,0}},
  \end{equation}
  where $\bar{v}_{j,0}$ is the value of $\bar{v}$ at the element of
  $\bar{F}$ with weight $c_{j,0}$; note that the indices and
    the sign have been adapted due to the choice of the position of
    $\tilde{b}$ in Figure~\ref{fig:single_col}.

  Now consider the Aztec diamond with shifted weights $\tilde{a}$ (see
  Figure~\ref{fig:single_col}, right). Its quotient graph on the cylinder
  is exactly the same graph $\bar{G}^\bullet$, with the same
  weights. Therefore we can use the same vector $\bar{v}$ to apply the
  previous procedure, which gives
  \begin{equation}
    \label{eq:yat_quot_v}
    Y(A_k,\tilde{a}) = - \frac{\sum_{j=0}^k c_{j,1} \bar{v}_{j,1}}{\sum_{j=0}^k \bar{v}_{j,1}}.
  \end{equation}

  Since $\bar{v}\in \ker{\bar{D}^T}$, in particular $\bar{v} \in \ker
  \left( \bar{C}(1)^{\bar{B}} \right)^T$. Writing what this means for
  the $k+1$ elements of $\bar{B}$ shown as the bottom row in
  Figure~\ref{fig:gcyl}, and summing all these equations, we get
  \begin{equation}
    \label{eq:sumv01}
    \sum_{j=0}^k \bar{v}_{j,0} = \sum_{j=0}^k \bar{v}_{j,1}.
  \end{equation}
  Doing the same for $\left( \bar{C}(a)^{\bar{B}} \right)^T$ gives
  \begin{equation}
    \label{eq:sumcv01}
    \sum_{j=0}^k c_{j,0} \bar{v}_{j,0} = \sum_{j=0}^k c_{j,1} \bar{v}_{j,1}.
  \end{equation}
  Using Equations~\eqref{eq:sumv01}, \eqref{eq:sumcv01}, we get that
  \eqref{eq:ya_quot_v}, \eqref{eq:yat_quot_v} are equal.
\end{proof}

\begin{remark}
Note that this generalizes the second result of Theorem~\ref{thm:vert_translation}. Nevertheless, we chose a combinatorial approach there, providing an invariance result for the partition function $Z(A_k,a)$ itself, which is stronger, and cannot be reached by this technique.
\end{remark}

\subsection{dSKP Devron properties}
\label{sec:dskp-devr-prop}

We finish this section with the proofs of Devron properties as they
are stated in the introduction, namely
Theorem~\ref{theo:sing_dodgson_intro}, Theorem~\ref{theo:sing_devron_intro}
and Corollary~\ref{cor:harm_mean}. For this purpose, we go back to
the initial convention for indices on the lattice $\calL$.
We first rephrase
Corollary~\ref{cor:dodgson} in terms of the dSKP solution. Consider again some
function $x:\calL \to \hat{\C}$ satisfying the dSKP recurrence, the
height function $h(i,j)=[i+j]_2$, and initial data
$(a_{i,j})=(x(i,j,h(i,j)))$, with no periodicity assumption.
\begin{corollary}
  \label{cor:dodgson_L}
  Suppose that the initial data are such that for some $d \in
  \hat{\C}$, for all $(i,j)\in\Z^2$ such that $[i+j]_2=0$, $a_{i,j}=d$.
  Let $(i,j,k) \in \calL$ with $k\geq 1$. Consider the
  matrix $N=\left( N_{i',j'} \right)_{0\leq i',j' \leq k-1}$ with entries
  \begin{equation*}
    N_{i',j'} = \left( a_{i-i'+j',j+k-1-i'-j'} - d\right)^{-1}.
  \end{equation*}
  Then
  \begin{equation*}
    x(i,j,k) = d+\sum_{0\leq i',j' \leq k-1} \left( N^{-1} \right)_{i',j'},
  \end{equation*}
  where the sum is over entries of the inverse matrix $N^{-1}$.
\end{corollary}

\begin{proof}
  With the stated hypothesis, via
  Theorem~\ref{theo:dskp_dimers} as in Example~\ref{ex:Aztec}, the
  value of $x(i,j,k)$ is $Y(A_{k-1},a)$ for Aztec diamond of size
  $k-1$, with the central face labeled $a_{i,j}$. This
  Aztec diamond has face weights as in Corollary~\ref{cor:dodgson}.
  Rewriting the matrix $N$ in the original coordinate system gives the result.
\end{proof}

\begin{proof}[Proof of Theorem~\ref{theo:sing_dodgson_intro}]
  We use Corollary~\ref{cor:dodgson_L}. In the lattice
  $\calL$, going from $(i,j,m)$ to $(i+1,j+1,m)$ changes the matrix
  $N$ simply by a cyclic permutation of
  the rows, which does not change the value of the sum of coefficients
  of $N^{-1}$. Therefore $x(i,j,m)=x(i+1,j+1,m)$, and similarly
  $x(i,j,m)=x(i+1,j-1,m)$, which proves that this value is independent
  of $i,j$ as long as $(i,j,m)\in \calL$.
\end{proof}

\begin{proof}[Proof of Theorem~\ref{theo:sing_devron_intro}]
  The argument is almost the same, except now the weights on the
  corresponding Aztec diamond of size $k-1=(m-2)p+1$ satisfy the
  hypothesis of Theorem~\ref{thm:single_col}, in particular the fact
  that $[i-j-mp]_{2p}=0$ translates into the fact that constant columns
  appear at the leftmost and rightmost columns of inner faces, as
  in~\eqref{eq:per_weights} and Figure~\ref{fig:single_col}. Going in
  $\calL$ from $(i,j,k)$ to $(i+1,j+1,k)$ has the effect of
  changing the Aztec diamond weights $a$ into $\tilde{a}$, so the
  result is a rephrasing of Theorem~\ref{thm:single_col}.
\end{proof}

For the proof of Corollary~\ref{cor:harm_mean},
we need the following basic lemma:
\begin{lemma}
  \label{lem:markovdodgson}
 Let $N$ be an invertible $m\times m$ matrix. Suppose that there is a
 $\lambda \in \C^*$ such that for all $i,$
  $\sum_{j} N_{ij} = \lambda$. Then
  $\sum_{i,j} (N^{-1})_{ij} = m\lambda^{-1}$.
\end{lemma}

\begin{proof}
  The vector $o\in \C^m$ with $o_i = 1$ for all
  $i$ is clearly an eigenvector of $N$ for the eigenvalue
  $\lambda$. Therefore $\lambda^{-1}$ is an eigenvalue to
  eigenvector $o$ for the inverse matrix $N^{-1}$. Thus
  \begin{equation*}
    \sum_{i,j} (N^{-1})_{ij} = o^{T}M^{-1}o = o^{T}\lambda^{-1}o = m \lambda^{-1}.
  \end{equation*}
\end{proof}

\begin{proof}[Proof of Corollary~\ref{cor:harm_mean}]
  For $m$-Dodgson initial conditions, whenever $(i,j,m)\in \calL$, the
  value of $x(i,j,m)$ is given by
  Corollary~\ref{cor:dodgson_L}. Under the condition
  $a_{i,j}=a_{i+p+1,j-p+1}$, note that the successive rows of $N$
  contain the same variables, shifted by $p$ from one row to the
  next. The fact that $p\notin m\Z$ guarantees that $N$ is invertible. Therefore $N$ satisfies the conditions of
  Lemma~\ref{lem:markovdodgson} with $\lambda$ being the sum of the $\frac{1}{a-d}$
  on one row of $N$.
\end{proof}

Finally, we describe how these singularities occur for initial data
with two different periodicities.
For a vector $(s,t)$ in $\Z^2$
with $[s+t]_2=0$, we say that the initial data is $(s,t)$-\emph{periodic} if
$a_{i,j}=a_{i+s,j+t}$ for all $(i,j)\in \Z^2$.
The following describes how many steps one has to go until
singularities reoccur when initial conditions have two such periodicities.
\begin{corollary}
  \label{cor:devron_periodic_vectors}
  Let $(s,t)$ and $(u,v)$ be two non-collinear vectors in $\Z^2$, with
  $[s+t]_2=[u+v]_2=0$. Suppose that the initial condition is both
  $(s,t)$ and $(u,v)$ periodic. Let $g=\gcd(s-t,u-v)$, and
  $\calA=|sv-tu|$.

  \begin{itemize}
  \item If the initial data is such that for all $(i,j)\in \Z^2$ with
    $[i+j]_2=0$,
    \begin{equation*}
      a(i,j)=a(i+1,j+1),
    \end{equation*}
    then for $k=\frac{\calA}{g}$ the same is true at height $k$. More precisely,
    \begin{equation*}
      \forall i,j \text{ s.t. } (i,j,k) \in \calL, \ x(i,j,k)=x(i+1,j+1,k).
    \end{equation*}
  \item If the
  initial data is such that
  \begin{equation*}
    \forall i \in \Z, a_{i,i}=a_{i+1,i+1},
  \end{equation*}
  then let $k=\frac{\calA}{2}-g+2$.
  After $k-1$ iterations of the dSKP recurrence, the values of
  $x$ have $\frac{g}{2}$-periodic constant columns.
  More precisely, for all $i,j$ such that
  $\left[i-j-\frac{\calA}{2}\right]_{\frac{g}{2}}=0$,
  \begin{equation*}
    x(i,j,k)=x(i+1,j+1,k).
  \end{equation*}
  \end{itemize}

\end{corollary}

\begin{proof}
  We are relating the initial data $a$ to the hypothesis of
  Theorem~\ref{theo:sing_devron_intro}. First we are looking for $m$
  such that $a$ is $m$-simply periodic. This is equivalent to
  \begin{equation*}
    (m,m) \in \Z (s,t) + \Z (u,v).
  \end{equation*}
  Classical arithmetic computations show that the smallest positive
  such $m$ is $\frac{\calA}{g}$. In the first point, this shows that the
  initial data is $(m,1)$-Devron, and we get the result by applying
  Theorem~\ref{theo:sing_devron_intro}.

  In the second point, we want to find $p$ such that the constant
  ``column'' $(a_{i,i})_{i\in \Z}$ repeats every $p$ even
  column. Since the initial data is $(s,t)$-periodic, this is true for
  $p_1=\frac{s-t}{2}$, and similarly for $p_2=\frac{u-v}{2}$, so the
  smallest positive $p$ is $\gcd\left( p_1,p_2 \right)=\frac{g}{2}$.
  For these values of $m$ and $p$, the initial data is $(m,p)$-Devron,
  and we apply Theorem~\ref{theo:sing_devron_intro}.
\end{proof}

\section{Limit shapes}\label{sec:limitshapes}

Suppose that $x:\calL \to \hat\C$ satisfies the dSKP recurrence.
Let
$h(i,j)=[i+j]_2$
then for any $p\in \calU_h$, we may consider $x(p)$ as a
function of the $(a_{i,j})=(x(i,j,h(i,j)))$, given by Theorem~\ref{theo:dskp_dimers}
We are interested in the effect of the initial
condition at $(0,0,0)$ on the value of $x_p$. Let
\begin{equation}
  \label{eq:defrho}
  \rho(i,j,k) = \dfrac{\partial
    x(i,j,k)}{\partial a_{0,0}}.
\end{equation}
We are looking for the order of magnitude of $\rho(i,j,k)$ when
$k\to \infty$ and $i/k, j/k$ both converge to some constants. This
behaviour is made explicit in the forthcoming
Proposition~\ref{prop:limitshape}, when the derivative
\eqref{eq:defrho} is evaluated at specific dSKP solutions of
Example~\ref{ex:sol}: $x(i,j,k)=ia+jb+kc+d$ for some fixed $a,b,c,d$ taken in $\R$ hereafter. By a slight abuse of notation, for $x,y\in \R$,
we denote
$\rho(xk,yk,k) := \rho\left( \lfloor xk \rfloor, \lfloor yk \rfloor, k
\right)$.

\begin{figure}[h]
  \centering
  \includegraphics[width=14cm]{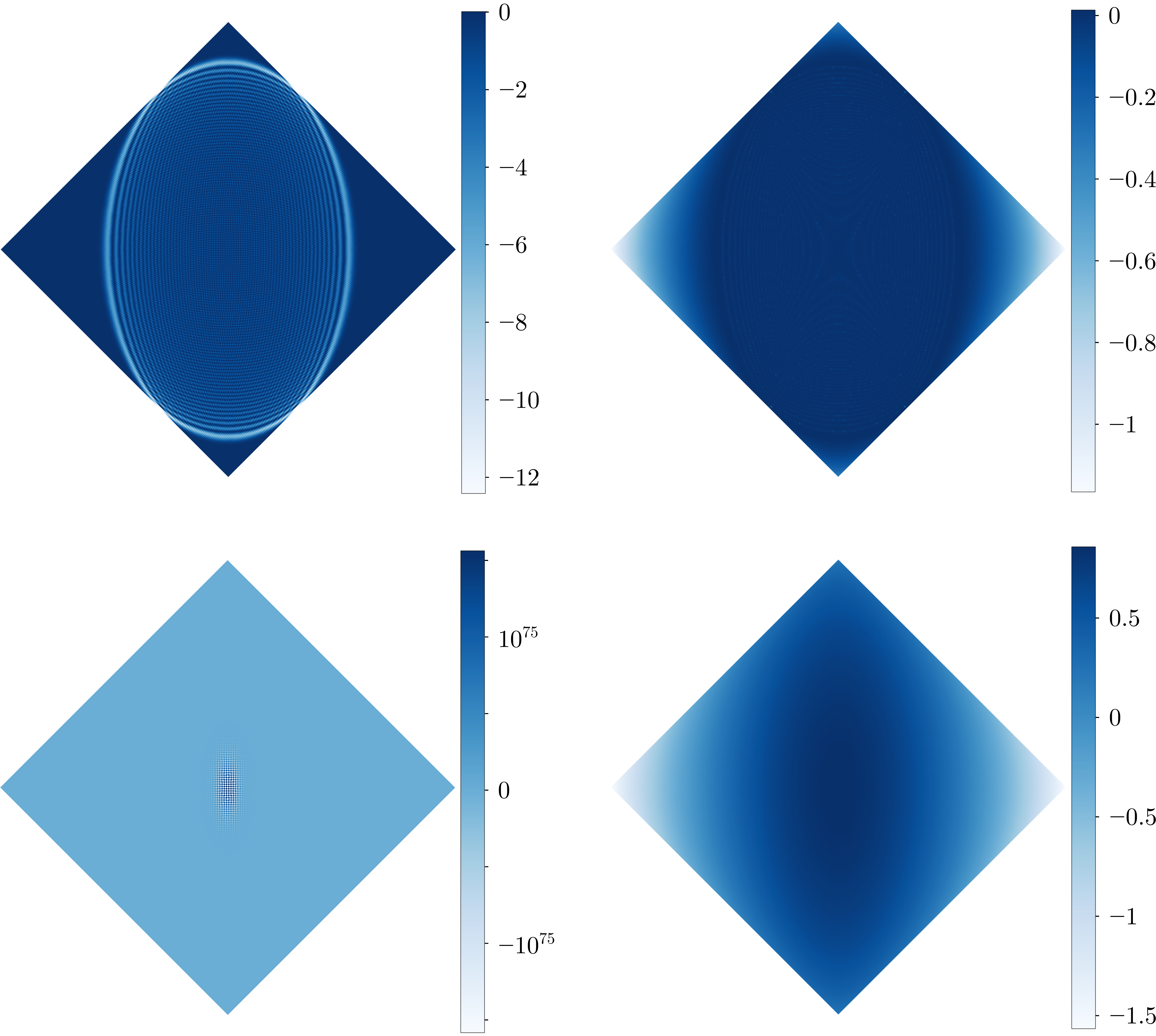}
  \caption{Values of $k\rho(i,j,k)$ (left) and
    $\frac{\log |k \rho(i,j,k)|}{k}$ (right) for $k=200$, evaluated at
    the solution $x(i,j,k)=ia+jb+kc+d$, with
    $q:=\frac{c^2-b^2}{a^2-b^2}=0.7$ (top) and $q=1.2$ (bottom).}
  \label{fig:heatmaprho}
\end{figure}

The identification of the asymptotic behaviour of $\rho(xk,yk,k)$ when
$k\to \infty$ follows the method of \cite{PSarctic} and is closely related to
\cite[Section~2.3]{dFSG}, where the analogous problem for the dKP
relation is treated. In fact dKP and dSKP produce almost the same generating
functions, however we use different references to deduce the
asymptotics. We mainly identify the generating functions with that of
\cite{BBCCR}, see also references therein, from which point the asymptotic behaviour can be
recovered almost directly from \cite{CEP}.

\begin{proposition}
  \label{prop:limitshape}
  Let $a,b,c,d \in \R$. Let $q=\frac{c^2-b^2}{a^2-b^2}$. Consider the
  partial derivatives $\rho(i,j,k)$ evaluated at
  $x(i,j,k)=ia+jb+kc+d$. Let $x,y\in\R$ be such that $|x+y| < 1$ and
  $|x-y| < 1$.

  \begin{enumerate}
  \item \label{it:1} If $q\in(0,1)$,
    \begin{enumerate}
    \item \label{it:a} for all $x,y$ such that
      $\frac{x^2}{1-q}+\frac{y^2}{q} < 1$ and $(kx,ky,k) \in \calL$,
      \begin{equation*}
        \rho(xk,yk,k) = \frac{2\cos^2(\Phi_q(xk,yk,k))}{k \pi \sqrt{q(1-q)
            - qx^2 - (1-q)y^2}} \left( 1+O\left( k^{-1} \right) \right)),
      \end{equation*}
      where the function $\Phi_q:\calL \to \R$ can be determined
      explicitly.
    \item \label{it:b} for all $x,y$ such that $\frac{x^2}{1-q}+\frac{y^2}{q} >
      1$,
      \begin{equation*}
        \rho(xk,yk,k) = \pm \frac{c_q(x,y)}{k} \exp(\xi_q(x,y)k) \ \left( 1+ o(1) \right),
      \end{equation*}
      where $\xi_q(x,y) < 0$ and $c_q(x,y)\in \R$ can be determined
      explicitly.
    \end{enumerate}
  \item \label{it:2} If $q\in (-\infty,0) \cup (1,\infty)$, for all
    $x,y$,
    \begin{equation*}
      \rho(xk,yk,k) = \pm \frac{c_q(x,y)}{k} \exp(\xi_q(x,y)k) \ \left( 1+ o(1) \right),
    \end{equation*}
    where $\xi_q(x,y) \in \R$ and $c_q(x,y)\in \R$ can be determined
    explicitly. Moreover, $\xi_q(x,y)>0$ in a neighbourhood of
    $(0,0)$.
  \end{enumerate}

\end{proposition}

\begin{proof}
  Let us differentiate \eqref{eq:dskp_x} with respect
  to $a_{0,0}$. This gives the following, evaluated at any $p=(i,j,k) \in
  \Z^3\setminus \calL$ with $k\geq 1$:
  \begin{equation*}
    \begin{split}
      0 = & \rho_{-e_0} \frac{x_{e_2}-x_{-e_1}}{(x_{-e_0}-x_{-e_1})(x_{e_2}-x_{-e_0})} +
      \rho_{-e_1}
      \frac{x_{-e_2}-x_{-e_0}}{(x_{-e_1}-x_{-e_0})(x_{-e_2}-x_{-e_1})}
      \\
      & + \rho_{-e_2}
      \frac{x_{-e_1}-x_{e_0}}{(x_{-e_2}-x_{e_0})(x_{-e_1}-x_{-e_2})} +
      \rho_{e_0}
      \frac{x_{e_1}-x_{-e_2}}{(x_{e_0}-x_{-e_2})(x_{e_1}-x_{e_0})}
      \\
      & + \rho_{e_1} \frac{x_{e_0}-x_{e_2}}{(x_{e_1}-x_{e_2})(x_{e_0}-x_{e_1})} +
      \rho_{e_2} \frac{x_{-e_0}-x_{e_1}}{(x_{e_2}-x_{e_1})(x_{-e_0}-x_{e_2})}.
    \end{split}
  \end{equation*}
  After evaluating at $x(i,j,k)=ia+jb+kc+d$, the relation
  becomes
  \begin{equation}
    \label{eq:recrho}
    \rho_{-e_2} + \rho_{e_2} = q\left(
      \rho_{-e_0}+\rho_{e_0} \right) + (1-q) \left( \rho_{-e_1}+\rho_{e_1} \right),
  \end{equation}
  where $q=\frac{c^2-b^2}{a^2-b^2}$.
  From this linear relation, and handling separately the case of
  $\rho(i,j,k)$ with $k=0$ or $1$, we deduce the generating function:
  \begin{equation*}
    \begin{split}
      F(u,v,t) := & \sum_{\substack{(i,j,k) \in \calL, \\ k\geq 0}} \rho(i,j,k) u^i
      v^j t^k \\
      = & 1 - \frac{t^2}{1+t^2-qt(u+u^{-1})-(1-q)t(v+v^{-1})}.
    \end{split}
  \end{equation*}
  This is directly related to the generating function of the
  \emph{creation rate} of the Aztec diamond that can be found in
  \cite[Remark~23]{BBCCR} for $\lambda=\frac{1-q}{q}$. By this
  identification and \cite[Proposition~21]{BBCCR}, we get
  \begin{equation}
    \label{eq:rhocq}
    \forall (i,j,k) \in \calL, k\geq1, \
    \rho(i,j,k) = -(1-q)^{k-2}C_q(A,B,k-2)C_q(B,A,k-2),
  \end{equation}
  where $C_q(A,B,n)$ is the coefficient of $z^A$ in
  $(1-z)^B\left(1+ \frac{q}{1-q}z\right)^{n-B}$, and $A=\frac{k-2-i-j}{2}$,
  $B=\frac{k-2+i-j}{2}$.

  The asymptotics of $C_q(A,B,n)$ when $n\to \infty$ and $A,B$ are
  proportional to $n$ can be analyzed by the saddle point method (see
  e.g. \cite[Section VIII]{FS}). In case~\ref{it:a}, this is done in
  \cite[Proposition~24]{CEP}, which gives the result after changing
  variables back to our setting. Case~\ref{it:b} follows closely the
  proof of \cite[Proposition~8]{CEP}, the main difference being that
  we apply the full saddle-point method instead of just looking for an
  upper bound; therefore, we only give details for case~\ref{it:2} below.

  Let us suppose that $q>1$, the case $q<0$ being almost identical. We
  are looking for the asymptotics of $C_q(A,B,k-2)$, when $k\to\infty$
  and $A = \frac{k-2-\lfloor xk \rfloor-\lfloor yk \rfloor}{2}, B = \frac{k-2+\lfloor xk \rfloor-\lfloor yk \rfloor}{2}$.
  Performing a contour integral around $0$ on a circle of arbitrary radius $r>0$, this is
  \begin{equation*}
    C_q(A,B,k-2) = \frac{1}{2i\pi} \oint G(z) z^{-A-1} dz = \frac{r^{-A}}{2\pi} \int_0^{2\pi} G(re^{i\theta}) e^{-iA\theta} d\theta,
  \end{equation*}
  where $G(z)=(1-z)^B \left(1-\frac{q}{q-1}z\right)^{k-2-B}$. The function $\theta \mapsto |G(re^{i\theta})|$ is increasing on $[0,\pi]$ and decreasing on $[\pi,2\pi]$, as both $|1-re^{i\theta}|$ and $\left|1-\frac{q}{q-1}re^{i\theta}\right|$ are increasing and decreasing as well (note that $\frac{q}{q-1}>0$). The saddle-point method consists first in setting $r=|\zeta|$ where $\zeta$ satisfies the saddle-point equation, here $\zeta \frac{G'(\zeta)}{G(\zeta)}=A$. This gives a quadratic equation with only one real negative solution, which is $\zeta = \zeta_0 + O(k^{-1})$ where
  \begin{equation*}
    \zeta_0 = \frac{(q-1)x+qy - \sqrt{(1-q)x^2+qy^2-q(1-q)}}{q(x+y+1)}.
  \end{equation*}
  From now on $r=|\zeta|$, so that the integral goes through the saddle-point. By the previous remark on $|G(re^{i\theta})|$, the part of the integral around $\pi$ of a sufficiently large interval (that may still go to $0$ as $k\to \infty$) exponentially dominates the rest. We choose an interval of half-length $\eta_k=k^{-2/5}$ and we split the integral into a centre part and a tail part:
  \begin{equation*}
  \int_0^{2\pi} G(re^{i\theta}) e^{-iA\theta} d\theta = \int_{|\theta-\pi|<\eta_k} G(re^{i\theta}) e^{-iA\theta} d\theta \ + \ \int_{\eta_k \leq |\theta-\pi| < \pi} G(re^{i\theta}) e^{-iA\theta} d\theta.
  \end{equation*}
  In the centre part, we can write $G(re^{i\theta}) e^{-iA\theta} = \exp(g(\theta))$ where
  \begin{equation*}
  g(\theta) = B \log\left(1-re^{i\theta}\right) + (k-2-B) \log\left(1-\frac{q}{q-1}re^{i\theta}\right) -iA\theta
  \end{equation*}
  using the principal value of the complex logarithm. The saddle-point equation gives $g'(\pi)=0$. As all derivatives of $g$ are of order $O(k)$ with uniform constants on a neighbourhood of $\pi$, we have the Taylor expansion
  \begin{equation*}
  g(\theta) = g(\pi) + \frac{(\theta-\pi)^2}{2} g''(\pi) + O(k\eta_k^3),
  \end{equation*}
  with $\real(g''(\pi))<0$. The fact that $k \eta_k^3 \to 0$ ensures that this holds uniformly on the centre part. On the other hand, in the tails part, the modulus $|G(re^{i\theta})|$ is smaller than that at $\pi \pm \eta_k$, which is of order $\exp\left[\real(g(\pi)) - O(k\eta_k^2)\right]$, where the constant in the $O$ is positive and uniform in $k$; the fact that $k\eta_k^2\to \infty$ ensures that the centre part exponentially dominates the tail part; it also ensures that the centre part can be completed to a complete gaussian integral.

The setup of the saddle-point method is complete, see~\cite[Section VIII]{FS}, and it gives
  \begin{equation}
    \label{eq:asympt_cq}
    C_q(A,B,k-2) = \pm \frac{|\zeta|^{-A}}{\sqrt {2\pi |g''(\pi)|}}
    \exp\left[ \real(g(\pi)) \right] \ \left( 1+o(1) \right).
  \end{equation}

  Recall that $A$, $g''(\pi)$ and $g(\pi)$ are all of order $k$. A similar formula holds for $C_q(B,A,k-2)$, therefore combining \eqref{eq:asympt_cq} and \eqref{eq:rhocq} yields the asymptotic announced in the theorem for $\rho(xk,yk,k)$, with continuous functions $c_q,\xi_q$. To get the last claim, it
  suffices to show that $\xi_q(0,0)>0$. For $x=y=0$ that is $A=B=\frac{k-2}{2}$,  explicit computations give successively
  \begin{equation*}
  \begin{split}
  \zeta & = \zeta_0 + O(k^{-1}) = -\sqrt{\frac{q-1}{q}} + O(k^{-1}), \\
  \real(g(\pi)) & = \frac{k}{2} \log\left( 2+ \sqrt{\frac{q-1}{q}} + \sqrt{\frac{q}{q-1}} \right) + O(1),\\
  \xi_q(0,0) & = -\log|\zeta_0|+\frac{2}{k}\real(g(\pi))+\log(q-1) + O(k^{-1}) \\
  &  = 2\log\left(\sqrt{q}+\sqrt{q-1}\right) + O(k^{-1})  >0.
  \end{split}
  \end{equation*}
\end{proof}

\begin{remark}
  The case $q\in (0,1)$ exhibits a phenomenon commonly known as a
  \emph{limit shape}. The decay of $\rho(xk,yk,k)$ to zero is either
  of order $k^{-1}$ inside the \emph{arctic ellipse}
  $\frac{x^2}{1-q} + \frac{y^2}{q} = 1$, or exponential outside the
  ellipse. By contrast, for $q\in (-\infty, 0) \cup (1,\infty)$, the
  derivative $\rho(xk,yk,k)$ always has an exponential behaviour, and the rate $\xi_q(x,y)$, which can bee seen as a Lyapunov exponent for the dynamics, is
  positive on some neighbourhood of $(0,0)$ (and for some
  values of $q$, it seems, for all possible $x,y$). This is indicative
  of a chaotic behaviour of the dSKP recurrence.
  Proposition~\ref{prop:limitshape}, although restricted to very
  special initial conditions, gives quantitative estimates for such a
  chaotic behaviour.
\end{remark}

\begin{remark}
  In the case of the other special solution of Example~\ref{ex:sol},
  $x(i,j,k)=a^ib^jc^kd$, relation \eqref{eq:recrho} is satisfied for
  the following logarithmic derivative:
  \begin{equation*}
    \tilde{\rho}(i,j,k) := \frac{a_{0,0}}{x(i,j,k)} \dfrac{\partial
      x(i,j,k)}{\partial a_{0,0}},
  \end{equation*}
  with $q=\frac{a(c-b)(bc-1)}{c(a-b)(ab-1)}$. The rest of the proof is
  unchanged, so Proposition~\ref{prop:limitshape} applies to
  $\tilde{\rho}(i,j,k)$, governed by that value of $q$. Of course,
  using the asymptotics of $\tilde{\rho}(i,j,k)$ and the explicit
  value of $x(i,j,k)$, one also gets the asymptotics of
  $\rho(i,j,k)$ in that case. On the other hand, for other recurrences
  where there is a probabilistic interpretation such as dKP,
  $\tilde{\rho}$ is a commonly considered observable that carries a
  statistical meaning sufficient to describe the existence of a limit
  shape with high probability; see again \cite{dFSG}.
\end{remark}

\section{The other consistent equations of octahedron type}\label{sec:other_consistent_equ}

In \cite{abs}, Adler, Bobenko and Suris introduce the notion of
\emph{multidimensional consistency} for system of equations on the
root lattice $Q(A_3)$, see Remark~\ref{rk:A3}. They show that up to a
set of transformations called \emph{admissible}, any such system can
be transformed into a system belonging to a finite list
$\{\chi_1,\chi_2,\chi_3,\chi_4,\chi_5\}$. By changing variables into
the lattice $\calL$ using Remark~\ref{rk:A3}, $\chi_1$ corresponds to
the dKP recurrence (Definition~\ref{def:dkp_recurrence}), and $\chi_2$
to the dSKP recurrence (Definition~\ref{def:dSKP_recurrence}). We now
turn our attention to the last three systems. Translated into $\calL$,
they give the following definitions:

\begin{definition}
  A function $x: \calL \rightarrow \hat{\C}$ satisfies the
  $\chi_3$, resp. $\chi_4, \chi_5$ \emph{recurrence} if
  \begin{align*}
    \chi_3: \ \ & (x_{e_3}-x_{-e_2})x_{-e_1} +
                  (x_{-e_2} - x_{e_1})x_{-e_3}  +
                  (x_{e_1}-x_{e_3})x_{e_2}  = 0, \\
    \text{resp. } \chi_4: \ \ & \frac{x_{e_3}-x_{-e_2}}{x_{-e_1}} +
                  \frac{x_{-e_2} - x_{e_1}}{x_{-e_3}}  +
                  \frac{x_{e_1}-x_{e_3}}{x_{e_2}}  = 0, \\
    \text{resp. } \chi_5: \ \ & \frac{x_{e_3}-x_{e_1}}{x_{e_2}}
                  = x_{-e_2} \left( \frac{1}{x_{-e_3}} - \frac{1}{x_{-e_1}} \right),
  \end{align*}
  holds evaluated at any $p$ of $\Z^3\setminus \calL$.
\end{definition}

For $\varepsilon>0$ and a rational fraction $Q(\varepsilon)$, we
denote by $\lc_\varepsilon Q(\varepsilon)$ the leading coefficient in
$\varepsilon$ in the asymptotic development of $Q(\varepsilon)$ at
$\varepsilon \to 0$. Solutions of the $\chi_3,\chi_4$ and $\chi_5$
recurrences can be deduced from that of dSKP by taking successive
leading terms in $\varepsilon$ limits of the initial conditions:

\begin{theorem}
  \label{theo:chi_k}
  Let $x: \calL \rightarrow \hat{\C}$ be a function that
  satisfies the $\chi_3$, resp. $\chi_4,\chi_5$ recurrence. Let $h$ be
  a height function, and $\calI, \calU$ be defined as in
  Equations~\eqref{eq:defI},~\eqref{eq:defU}. Let
  $(a_{i,j})=(x(i,j,h(i,j)))$ be the initial data indexed by points of
  $\calI$. For $\varepsilon, \delta, \rho > 0$, let
  \begin{equation*}
    \begin{split}
      a_{i,j}^{\varepsilon,\rho} & = (1+\rho \varepsilon^{i-j+k}a_{i,j}),\\
      a_{i,j}^\varepsilon & = \varepsilon^{i-j+k}a_{i,j},\\
      a_{i,j}^{\varepsilon,\delta} & = \varepsilon^{i-j+k}\delta^{i+j+k}a_{i,j}.
    \end{split}
  \end{equation*}
  Then for every point $p$ of $\calU$,
  \begin{equation*}
    \begin{split}
      \chi_3: \ \ x(p) & = \lc_\rho \left(\lc_\varepsilon
        Y(G_p,a^{\varepsilon,\rho}) - 1\right), \\
      \text{resp. }  \chi_4: \ \ x(p) & = \lc_\varepsilon Y(G_p,a^\varepsilon), \\
      \text{resp. }  \chi_5: \ \ x(p) & = \lc_\delta \left(\lc_\varepsilon Y(G_p,a^{\varepsilon,\delta})\right),
    \end{split}
  \end{equation*}
  where $G_p$ is the crosses-and-wrenches graph corresponding to $p$.
\end{theorem}

\begin{proof}
  We use \cite[Remark 3]{abs}. Changing variables using
  Remark~\ref{rk:A3}, it states that if one starts with the initial data
  $a^\varepsilon_{i,j}$ and applies the dSKP recurrence, then the
  leading coefficient in $\varepsilon$ satisfies the
  $\chi_4$ recurrence. More precisely, it can be checked directly by induction
  that if $x$ is the solution of the $\chi_4$ recurrence with initial
  data $a$, and $x^\varepsilon$ is the solution of the dSKP recurrence
  with initial data $a^\varepsilon$, then for any $p=(i,j,k)$,
  $x^\varepsilon(p) = \varepsilon^{i-j+k} \left(x(p)  +
  O(\varepsilon^2)\right)$. Therefore, the formula for $x(p)$ in the $\chi_4$
  case is a consequence of Theorem~\ref{theo:dskp_dimers}.

  The $\chi_5$ and $\chi_3$ cases are then limits from the $\chi_4$
  solution itself. For $\chi_5$, it can also be deduced from
  \cite[Remark 3]{abs}, where the expression as a $\delta$ limit from
  $\chi_4$ is stated; the proof is identical to the previous one. For
  $\chi_3$, the $\rho$ limit seems to be new. The proof consists in
  checking that if $x$ is the solution of the $\chi_3$ recurrence with
  initial data $b$, and if $x^\rho$ is the solution of the $\chi_4$
  recurrence with initial data $1+\rho b$, then for any
  $p, x^\rho(p) = 1+\rho x(p) + O(\rho^2)$.
\end{proof}

\begin{table}[h]
  \centering
  \begin{tabular}{|l|l|l|l|l|l|}
\hline
                          & Aztec diamond size & $1$ & $2$   & $3$     & $4$     \\ \hline
\multirow{2}{*}{$\chi_2$} & numerator          & $6$ & $220$ & $49224$ &    ?    \\ \cline{2-6}
                          & denominator        & $6$ & $220$ & $49224$ &    ?    \\ \hline
\multirow{2}{*}{$\chi_3$} & numerator          & $4$ & $30$  & $680$   & $45188$ \\ \cline{2-6}
                          & denominator        & $2$ & $14$  & $300$   & $19044$ \\ \hline
\multirow{2}{*}{$\chi_4$} & numerator          & $4$ & $56$  & $2656$  &    ?    \\ \cline{2-6}
                          & denominator        & $2$ & $14$  & $328$   &    ?    \\ \hline
\multirow{2}{*}{$\chi_5$} & numerator          & $3$ & $23$  & $433$   & $19705$ \\ \cline{2-6}
                          & denominator        & $1$ & $3$   & $23$    &  $433$  \\ \hline
  \end{tabular}
  \caption{Number of contributing configurations of Aztec diamonds
    of small size.}
  \label{tab:chi_k}
\end{table}

In fact, at least in the case where the crosses-and-wrenches graph is
the Aztec diamond (as in Example~\ref{ex:Aztec}), it is relatively
easy to describe exactly the combinatorics of configurations that
appear in the leading terms of Theorem~\ref{theo:chi_k}, at least in
the $\chi_4$ and $\chi_5$ cases. We describe these configurations
here. We leave the $\chi_3$ case as an open problem, as well as the
case of more generic height functions.

Recall that for the height function $h(i,j)=[i+j]_2$, the solution of
the dSKP recurrence at $p=(i,j,k+1)$ is expressed as $Y(A_k,a)$ for an
Aztec diamond of size $k$. By Equation~\eqref{eq:dskp_Aztec_dim} and Corollary~\ref{cor:Z_Aztec_diamond},
this solution may be expressed as
\begin{equation*}
  Y(A_k,a) = \biggl(\,\prod_{f\in F_p}a_f\biggr)
  \frac{Z(A_k,a^{-1},\varphi)}{Z(A_k,a,\varphi)}
  = \frac{\sum_{(\Ts,\Fs)\in \F}\sign(\Ts,\Fs)\prod_{\vec{e}\in
      \Ts}a_{f_{\vec{e}}}}{\sum_{(\Ts,\Fs)\in \F}\sign(\Ts,\Fs)\prod_{\vec{e}\in \Fs}a_{f_{\vec{e}}}},
\end{equation*}
where the sums are over pairs of complementary trees/forests on
$G^\bullet$, with a one-to-one correspondence between configurations
and monomials. To identify the terms that contribute in the $\chi_4$
case, let us apply the previous equation at
$a_{i,j}^\varepsilon = \varepsilon^{i-j+k}a_{i,j}$, and look for the terms
that minimize the $\varepsilon$ exponent. We do this first in the
denominator, so we are looking for configurations $(\Ts,\Fs)$ such
that the $\varepsilon$ exponent of variables in $\Fs$ is minimal. Recall
that in the forest $\Fs$, every black vertex has one outgoing
edge. Looking at the exponents of $\varepsilon$ around a black vertex in
Figure~\ref{fig:chi_k}, top left, this exponent is minimal if the
forest never has an edge oriented towards the South-East. We claim
that there exist such configurations, an example is displayed in
Figure~\ref{fig:chi_k}, top left. Thus, the previous constraint
completely characterize those configurations that appear in the
denominator of the solution of the $\chi_4$ recurrence.

\begin{figure}[tb]
  \centering
  \includegraphics[width=13cm]{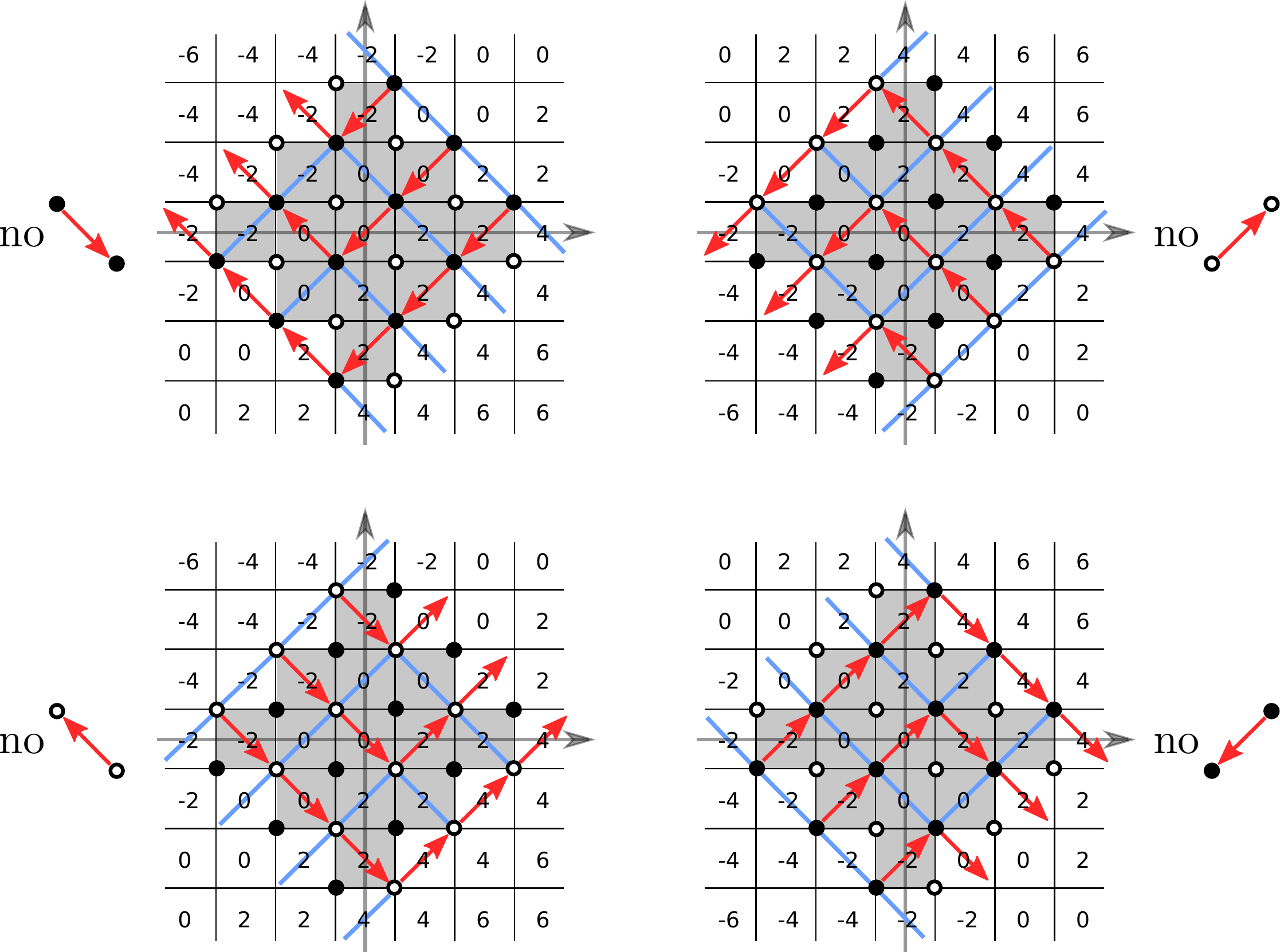}
  \caption{Example configurations contributing to the solutions of the
    $\chi_4$ and $\chi_5$ recurrence, for the initial data of
    Example~\ref{ex:Aztec}. Top left: the numbers indicate the
    exponent of $\varepsilon$ in $a^\varepsilon_{i,j}$; a tree/forest
    configuration on $G^\bullet$ is shown, with the forest shown in
    red and oriented, and vertices $b_r,\tilde{b}$ thought of as being
    ``at infinity''; this configuration minimizes the $\varepsilon$
    exponent, so it contributes to the denominator of the $\chi_4$
    solution. Top right: the exponent of $\delta$ in
    $a^{\varepsilon,\delta}_{i,j}$; the tree/forest configuration on
    $G^\circ$ minimizes the $\delta$ exponent, and corresponds to the
    top left configuration via the bijection; hence it also
    contributes to the denominator for $\chi_5$. Bottom: similarly, a
    configuration shown either on $G^\circ$ or on $G^\bullet$, such
    that the forest maximizes the $\varepsilon$ and $\delta$ exponents,
    hence contributing to the numerator of both $\chi_4$ and
    $\chi_5$.}
  \label{fig:chi_k}
\end{figure}

For the numerator, we have to minimize the $\varepsilon$ exponent in
$\Ts$ instead, so have to maximize it in $\Fs$. Consider the graph
$G^\circ$ analogous to $G^\bullet$ for white vertices. There is a bijection
between pairs of complementary trees/forests on $G^\bullet$ and those
on $G^\circ$, obtained by
replacing every black diagonal in $\Fs$ into the corresponding white
diagonal. Therefore, we may consider configurations on $G^\circ$. On
that graph, the same argument shows that the $\varepsilon$ exponent in
$\Fs$ is maximal when there is no edge oriented towards the
North-West. An example is shown in Figure~\ref{fig:chi_k}, bottom left.

For the $\chi_5$ case, we now have to minimize both the $\varepsilon$
exponent and the $\delta$ exponent. Similar considerations show that
contributing configurations for the denominator are those for which in
$G^\bullet$, $\Fs$ has no South-East going edge (to minimize the
$\varepsilon$ exponent as before), \textit{and} for the corresponding
forest on $G^\circ$, the forest has no North-East going edge (to
minimize the $\delta$ exponent). The previous example satisfies both
these constraints, see Figure~\ref{fig:chi_k}, top right. In the
numerator, the contributing configurations are such that in the forest
on $G^\circ$ there is no North-West going edge, \textit{and} in the
corresponding forest on $G^\bullet$ there is no South-West going edge,
see Figure~\ref{fig:chi_k}, bottom right.

As for the $\chi_3$ case, note that we would have to consider the
configurations coming from $\chi_4$ and write their weight in
variables $1+\rho a_{i,j,k}$, in order to expand each weight and find
the leading coefficient in $\rho$. Computing these by hand for small
values of $k$, it turns out that both in the numerator and
denominator, the coefficient in $\rho^0$ vanishes. We were not able to
find a simple combinatorial identification for the coefficients in
$\rho^1$ (which a priori may come from the expansion of several of the
$\chi_4$ configurations).

The number of contributing configurations for the first Aztec diamond
sizes (\emph{i.e.}~the number of monomials in the corresponding
solution of the $\chi_i$ recurrences) that we were able to compute are
given in Table~\ref{tab:chi_k}. For $\chi_5$, it seems that there are
as many monomials in the numerator for $A_k$ as in the denominator for
$A_{k+1}$, but we do not see a simple explanation for this pattern,
and we do not know if it continues.

\bibliographystyle{alphaurl}
\bibliography{references}

\newcommand{\etalchar}[1]{$^{#1}$}
\begin{thebibliography}{NCWQ84}

\bibitem[ABS12]{abs}
Vsevolod~E. Adler, Alexander~I. Bobenko, and Yuri~B. Suris.
\newblock {Classification of Integrable Discrete Equations of Octahedron Type}.
\newblock {\em International Mathematics Research Notices}, 2012(8):1822--1889,
  01 2012.
\newblock \href {https://doi.org/10.1093/imrn/rnr083}
  {\path{doi:10.1093/imrn/rnr083}}.

\bibitem[AdTM22]{paper2}
Niklas~C. Affolter, Béatrice de~Tilière, and Paul Melotti.
\newblock {The Schwarzian octahedron recurrence (dSKP equation) II: geometric
  systems}.
\newblock Preprint on arXiv, 2022.
\newblock \href {https://arxiv.org/abs/2208.00244} {\path{arXiv:2208.00244}}.

\bibitem[Aff21]{amiquel}
Niklas~C. Affolter.
\newblock {Miquel dynamics, Clifford lattices and the Dimer model}.
\newblock {\em Letters in Mathematical Physics}, 111:1--23, 2021.
\newblock \href {https://doi.org/10.1007/s11005-021-01406-0}
  {\path{doi:10.1007/s11005-021-01406-0}}.

\bibitem[BBC{\etalchar{+}}17]{BBCCR}
C{\'e}dric Boutillier, J{\'e}r{\'e}mie Bouttier, Guillaume Chapuy, Sylvie
  Corteel, and Sanjay Ramassamy.
\newblock Dimers on rail yard graphs.
\newblock {\em Annales de l’Institut Henri Poincar{\'e} D}, 4(4):479--539,
  2017.
\newblock \href {https://doi.org/10.4171/AIHPD/46}
  {\path{doi:10.4171/AIHPD/46}}.

\bibitem[BK98a]{BK1}
Leonid~V. Bogdanov and Boris~G. Konopelchenko.
\newblock {Analytic-bilinear approach to integrable hierarchies. I. Generalized
  KP hierarchy}.
\newblock {\em Journal of Mathematical Physics}, 39(9):4683--4700, 1998.
\newblock \href {https://doi.org/10.1063/1.532540}
  {\path{doi:10.1063/1.532540}}.

\bibitem[BK98b]{BK2}
Leonid~V. Bogdanov and Boris~G. Konopelchenko.
\newblock {Analytic-bilinear approach to integrable hierarchies. II.
  Multicomponent KP and 2D Toda lattice hierarchies}.
\newblock {\em Journal of Mathematical Physics}, 39(9):4701--4728, 1998.
\newblock \href {https://doi.org/10.1063/1.532531}
  {\path{doi:10.1063/1.532531}}.

\bibitem[BS02]{BSbook}
Michael Brin and Garrett Stuck.
\newblock {\em Introduction to Dynamical Systems}.
\newblock Cambridge University Press, 2002.
\newblock \href {https://doi.org/10.1017/CBO9780511755316}
  {\path{doi:10.1017/CBO9780511755316}}.

\bibitem[BV99]{BV}
Marc~P. Bellon and Claude-Michel Viallet.
\newblock Algebraic entropy.
\newblock {\em Communications in mathematical physics}, 204(2):425--437, 1999.
\newblock \href {https://doi.org/10.1007/s002200050652}
  {\path{doi:10.1007/s002200050652}}.

\bibitem[CEP{\etalchar{+}}96]{CEP}
Henry Cohn, Noam Elkies, James~G. Propp, et~al.
\newblock Local statistics for random domino tilings of the aztec diamond.
\newblock {\em Duke Mathematical Journal}, 85(1):117--166, 1996.
\newblock \href {https://doi.org/10.1215/S0012-7094-96-08506-3}
  {\path{doi:10.1215/S0012-7094-96-08506-3}}.

\bibitem[Ciu03]{CiucuUR}
Mihai Ciucu.
\newblock Perfect matchings and perfect powers.
\newblock {\em Journal of Algebraic Combinatorics}, 17(3):335--375, 2003.
\newblock \href {https://doi.org/10.1023/A:1025005023573}
  {\path{doi:10.1023/A:1025005023573}}.

\bibitem[CKP01]{CKP}
Henry Cohn, Richard Kenyon, and James Propp.
\newblock A variational principle for domino tilings.
\newblock {\em Journal of the American Mathematical Society}, 14(2):297--346,
  2001.
\newblock \href {https://doi.org/10.1090/S0894-0347-00-00355-6}
  {\path{doi:10.1090/S0894-0347-00-00355-6}}.

\bibitem[CLR20]{clrtembeddings}
Dmitry Chelkak, Benoît Laslier, and Marianna Russkikh.
\newblock Dimer model and holomorphic functions on t-embeddings of planar
  graphs.
\newblock Preprint, arXiv:2001.11871, 2020.
\newblock \href {https://arxiv.org/abs/2001.11871} {\path{arXiv:2001.11871}}.

\bibitem[CS04]{CScube}
Gabriel~D. Carroll and David~E Speyer.
\newblock The cube recurrence.
\newblock {\em the electronic journal of combinatorics}, 11(R73):1, 2004.
\newblock \href {https://doi.org/10.37236/1826} {\path{doi:10.37236/1826}}.

\bibitem[DFSG14]{dFSG}
Philippe Di~Francesco and Rodrigo Soto-Garrido.
\newblock Arctic curves of the octahedron equation.
\newblock {\em Journal of Physics A: Mathematical and Theoretical},
  47(28):285204, 2014.
\newblock \href {https://doi.org/10.1088/1751-8113/47/28/285204}
  {\path{doi:10.1088/1751-8113/47/28/285204}}.

\bibitem[DN91]{dndskp}
Irene~Ya. Dorfman and Frank~W. Nijhoff.
\newblock {On a (2+1)-dimensional version of the Krichever-Novikov equation}.
\newblock {\em Physics Letters A}, 157(2):107--112, 1991.
\newblock \href {https://doi.org/10.1016/0375-9601(91)90080-R}
  {\path{doi:10.1016/0375-9601(91)90080-R}}.

\bibitem[Dod67]{dodgson}
Charles~L. Dodgson.
\newblock {IV. Condensation of determinants, being a new and brief method for
  computing their arithmetical values}.
\newblock {\em Proceedings of the Royal Society of London}, 15:150--155, 1867.
\newblock \href {https://doi.org/10.1098/rspl.1866.0037}
  {\path{doi:10.1098/rspl.1866.0037}}.

\bibitem[FS09]{FS}
Philippe Flajolet and Robert Sedgewick.
\newblock {\em Analytic Combinatorics}.
\newblock Cambridge University Press, 2009.
\newblock URL:
  \url{http://www.cambridge.org/uk/catalogue/catalogue.asp?isbn=9780521898065}.

\bibitem[Geo21]{Ggroves}
Terrence George.
\newblock Grove arctic curves from periodic cluster modular transformations.
\newblock {\em International Mathematics Research Notices},
  2021(20):15301--15336, 2021.
\newblock \href {https://doi.org/10.1093/imrn/rnz367}
  {\path{doi:10.1093/imrn/rnz367}}.

\bibitem[Gli15]{gdevron}
Max Glick.
\newblock {The Devron property}.
\newblock {\em Journal of Geometry and Physics}, 87:161--189, 2015.
\newblock \href {https://doi.org/10.1016/j.geomphys.2014.07.029}
  {\path{doi:10.1016/j.geomphys.2014.07.029}}.

\bibitem[Gor21]{Gbook}
Vadim Gorin.
\newblock {\em Lectures on Random Lozenge Tilings}.
\newblock Cambridge Studies in Advanced Mathematics. Cambridge University
  Press, 2021.
\newblock \href {https://doi.org/10.1017/9781108921183}
  {\path{doi:10.1017/9781108921183}}.

\bibitem[Hir81]{hirota}
Ryogo Hirota.
\newblock {Discrete analogue of a generalized Toda equation}.
\newblock {\em Journal of the Physical Society of Japan}, 50(11):3785--3791,
  1981.
\newblock \href {https://doi.org/10.1143/JPSJ.50.3785}
  {\path{doi:10.1143/JPSJ.50.3785}}.

\bibitem[{Kas}61]{Kasteleyn}
Pieter~W. {Kasteleyn}.
\newblock {The statistics of dimers on a lattice : I. The number of dimer
  arrangements on a quadratic lattice}.
\newblock {\em Physica}, 27:1209--1225, December 1961.
\newblock \href {https://doi.org/10.1016/0031-8914(61)90063-5}
  {\path{doi:10.1016/0031-8914(61)90063-5}}.

\bibitem[Kas67]{Kasteleyn2}
Pieter~W. Kasteleyn.
\newblock Graph theory and crystal physics.
\newblock In {\em Graph {T}heory and {T}heoretical {P}hysics}, pages 43--110.
  Academic Press, London, 1967.

\bibitem[KLRR22]{klrr}
Richard Kenyon, Wai~Yeung Lam, Sanjay Ramassamy, and Marianna Russkikh.
\newblock {Dimers and circle patterns}.
\newblock {\em {Annales Scientifiques de l'{\'E}cole Normale Sup{\'e}rieure}},
  55(3):863--901, 2022.
\newblock \href {https://doi.org/10.24033/asens.2507}
  {\path{doi:10.24033/asens.2507}}.

\bibitem[KP16]{KPddim}
Richard~W. Kenyon and Robin Pemantle.
\newblock Double-dimers, the ising model and the hexahedron recurrence.
\newblock {\em Journal of Combinatorial Theory, Series A}, 137:27--63, 2016.
\newblock \href {https://doi.org/10.1016/j.jcta.2015.07.005}
  {\path{doi:10.1016/j.jcta.2015.07.005}}.

\bibitem[KPW00]{KPW}
Richard~W. Kenyon, James~G. Propp, and David~B. Wilson.
\newblock Trees and matchings.
\newblock {\em Electron. J. Combin}, 7(1):R25, 2000.
\newblock \href {https://doi.org/10.37236/1503} {\path{doi:10.37236/1503}}.

\bibitem[KS02]{ksclifford}
Boris~G. Konopelchenko and Wolfgang~K. Schief.
\newblock {Menelaus' theorem, Clifford configurations and inversive geometry of
  the Schwarzian {KP} hierarchy}.
\newblock {\em Journal of Physics A: Mathematical and General},
  35(29):6125--6144, 07 2002.
\newblock \href {https://doi.org/10.1088/0305-4470/35/29/313}
  {\path{doi:10.1088/0305-4470/35/29/313}}.

\bibitem[Kup98]{Kuperberg}
Greg Kuperberg.
\newblock An exploration of the permanent-determinant method.
\newblock {\em Electron. J. Combin.}, 5(1):Research Paper 46, 34, 1998.
\newblock \href {https://doi.org/10.37236/1384} {\path{doi:10.37236/1384}}.

\bibitem[Mel18]{Mkash}
Paul Melotti.
\newblock {The free-fermionic $C_2^{(1)}$ loop model, double dimers and
  Kashaev's recurrence}.
\newblock {\em Journal of Combinatorial Theory, Series A}, 158:407--448, 2018.
\newblock \href {https://doi.org/10.1016/j.jcta.2018.04.005}
  {\path{doi:10.1016/j.jcta.2018.04.005}}.

\bibitem[NCWQ84]{ncwqdskp}
Frank~W. Nijhoff, Hans~W. Capel, Gerlof~L. Wiersma, and G.~Reinout~W. Quispel.
\newblock Bäcklund transformations and three-dimensional lattice equations.
\newblock {\em Physics Letters A}, 105(6):267--272, 1984.
\newblock \href {https://doi.org/10.1016/0375-9601(84)90994-0}
  {\path{doi:10.1016/0375-9601(84)90994-0}}.

\bibitem[Per69]{Percus}
Jerome~K. Percus.
\newblock {One More Technique for the Dimer Problem}.
\newblock {\em J. Math. Phys.}, 10(10):1881--1884, 1969.
\newblock \href {https://doi.org/10.1063/1.1664774}
  {\path{doi:10.1063/1.1664774}}.

\bibitem[Pro03]{ProppUR}
James~G. Propp.
\newblock Generalized domino-shuffling.
\newblock {\em Theor. Comput. Sci.}, 303:267--301, 2003.
\newblock \href {https://doi.org/10.1016/S0304-3975(02)00815-0}
  {\path{doi:10.1016/S0304-3975(02)00815-0}}.

\bibitem[PS05]{PSarctic}
T.~Kyle Petersen and David~E Speyer.
\newblock An arctic circle theorem for groves.
\newblock {\em Journal of Combinatorial Theory, Series A}, 111(1):137--164,
  2005.
\newblock \href {https://doi.org/10.1016/j.jcta.2004.11.013}
  {\path{doi:10.1016/j.jcta.2004.11.013}}.

\bibitem[RG11]{rgbook}
Jürgen Richter-Gebert.
\newblock {\em {Perspectives on Projective Geometry: A Guided Tour Through Real
  and Complex Geometry}}.
\newblock Springer Publishing Company, Incorporated, 1st edition, 2011.

\bibitem[Spe07]{Speyer}
David~E Speyer.
\newblock Perfect matchings and the octahedron recurrence.
\newblock {\em Journal of Algebraic Combinatorics}, 25(3):309--348, 2007.
\newblock \href {https://doi.org/10.1007/s10801-006-0039-y}
  {\path{doi:10.1007/s10801-006-0039-y}}.

\bibitem[Tem74]{Temperley}
Harold N.~V. Temperley.
\newblock Enumeration of graphs on a large periodic lattice.
\newblock In {\em Combinatorics ({P}roc. {B}ritish {C}ombinatorial {C}onf.,
  {U}niv. {C}oll. {W}ales, {A}berystwyth, 1973)}, London Math. Soc. Lecture
  Note Ser., No. 13, pages 155--159. Cambridge Univ. Press, London, 1974.

\bibitem[TF61]{TemperleyFisher}
Harold N.~V. Temperley and Michael~E. Fisher.
\newblock Dimer problem in statistical mechanics -- an exact result.
\newblock {\em The Philosophical Magazine: A Journal of Theoretical
  Experimental and Applied Physics}, 6(68):1061--1063, 1961.
\newblock \href {https://doi.org/10.1080/14786436108243366}
  {\path{doi:10.1080/14786436108243366}}.

\bibitem[Yao14]{yao}
Zijian Yao.
\newblock Glick's conjecture on the point of collapse of axis-aligned polygons
  under the pentagram maps.
\newblock Preprint on arXiv, 2014.
\newblock \href {https://arxiv.org/abs/1410.7806} {\path{arXiv:1410.7806}}.

\end{thebibliography}

\end{document}